\pdfoutput=1
\documentclass[11pt, letterpaper]{article}
\usepackage{blindtext}
\usepackage{hyperref}
\hypersetup{
	colorlinks=true,
	linkcolor=blue!70!black,
	citecolor=blue!70!black,
	urlcolor=blue!70!black
}

\usepackage[english]{babel}
\usepackage[T1]{fontenc}
\usepackage{parskip}
\usepackage[margin=1in]{geometry}

\usepackage{amsmath}
\usepackage{todonotes}
\usepackage{amssymb}
\usepackage{amsthm}
\usepackage{booktabs}
\usepackage{accents}
\usepackage{algorithm}
\usepackage{algorithmic}
\usepackage[lined,boxed,ruled,norelsize,algo2e,linesnumbered]{algorithm2e}
\usepackage{bbm}
\usepackage{bbold}
\usepackage{bm}
\usepackage{graphicx}
\graphicspath{{./figs/}}

\usepackage{cleveref}
\usepackage{thm-restate}
\usepackage{algorithm}
\usepackage{algorithmic}
\usepackage{mathtools,verbatim}

\newtheorem{theorem}{Theorem}
\newtheorem{lemma}{Lemma}

\newtheorem{definition}{Definition}
\newtheorem{corollary}{Corollary}
\newtheorem{proposition}{Proposition}

\newtheorem{remark}{Remark}
\newtheorem{setting}{Setting}
\newtheorem{wrapper}{Wrapper}
\usepackage[numbers]{natbib}
\usepackage{multirow}

\newcommand{\fm}{\mathcal{S}}
\newcommand{\bp}{\alpha}

\newcommand{\qa}{8\log(1/\bp)}
\newcommand{\rd}{\color{red}}
\newcommand{\na}{n_{\bp}}
\newcommand{\bk}{\color{black}}

\newcommand{\I}{\mathcal{I}}
\newcommand{\defeq}{:=}
\newcommand{\norm}[1]{\left\lVert#1\right\rVert}

\newcommand{\eps}{\epsilon}

\newcommand{\E}{\mathbb{E}}
\renewcommand{\P}{\mathbb{P}}

\newcommand{\R}{\mathbb{R}}

\newcommand{\bas}[1]{\begin{align*}#1\end{align*}}
\newcommand{\ba}[1]{\begin{align}#1\end{align}}
\newcommand{\bbb}[1]{\left[#1\right]}
\newcommand{\bmb}[1]{\left | #1\right |}
\newcommand{\bcb}[1]{\left \{ #1\right \}}

\newcommand{\Var}{\var}

\newcommand{\bb}[1]{\left(#1\right)}

\usepackage{xcolor}
\newcommand{\tn}{\hat{\theta}_{\text{naive}}}
\newcommand{\tall}{\tilde{\theta}_{\text{all}}}
\newcommand{\tss}{\hat{\theta}_{\text{ss}}}
\newcommand{\tpss}{\tilde{\theta}_{\text{ss}}}

\definecolor{color1}{rgb}{0.8, 0.33, 0.0}
\definecolor{color2}{rgb}{0.0, 0.3, 0.9}

\newcommand{\G}{\mathcal{G}}
\newcommand{\proj}{\text{proj}}

\newcommand{\var}{\text{var}}
\newcommand{\Good}{\text{Good}}
\newcommand{\Bad}{\text{Bad}}
\newcommand{\conc}{\xi}
\newcommand{\cb}{\xi}
\newcommand{\hj}{Hájek }
\newcommand{\buf}[1]{L_{#1}}
\newcommand{\proxy}{\tau}
\newcommand{\X}{\mathbf{X}}
\newcommand{\F}{\mathcal{S}}
\newcommand{\dep}[3]{\text{dep}_{#1,#2}\bb{#3}}
\newcommand{\h}[3]{\hat{h}_{#2}({#3})}
\newcommand{\g}[3]{\hat{g}_{#2}({#3})}
\newcommand{\Yc}[1]{Q(#1)}
\newcommand{\Uc}[1]{Q^{\text{avg}}(#1)}

\newcommand{\A}{\mathcal{A}}

\newcommand{\wt}{\text{wt}}
\newcommand{\covar}{\text{cov}}

\DeclareMathOperator*{\argmin}{\arg\min}

\DeclareMathOperator{\Laplace}{Lap}

\DeclareMathOperator{\Bernoulli}{Ber}

\newcommand{\SubS}{\mathcal{S}}
\newcommand{\SubA}{A_n}
\newcommand{\tSubA}{\tilde{A}_n}

\title{On Differentially Private U-Statistics}
\author{
  \begin{tabular}{c@{\hskip 0.4in}c} 
 Kamalika Chaudhuri  &   Po-Ling Loh\\
  \normalsize University of California San Diego & \normalsize University of Cambridge \\ 
\normalsize\texttt{kamalika@ucsd.edu} &  \normalsize\texttt{pll28@cam.ac.uk} \\
\\
  Shourya Pandey & Purnamrita Sarkar \\
  \normalsize University of Texas at Austin &   \normalsize University of Texas at Austin\\
\normalsize\texttt{shourya@utexas.edu} & \normalsize\texttt{purna.sarkar@austin.utexas.edu}
\end{tabular}
}

\date{}
\begin{document}

\maketitle

\begin{abstract}%
We consider the problem of privately estimating a parameter $\mathbb{E}[h(X_1,\dots,X_k)]$, where $X_1$, $X_2$, $\dots$, $X_k$ are i.i.d.\ data from some distribution and $h$ is a permutation-invariant function. Without privacy constraints, standard estimators are U-statistics, which commonly arise in a wide range of problems, including nonparametric signed rank tests, symmetry testing, uniformity testing, and subgraph counts in random networks, and can be shown to be minimum variance unbiased estimators under mild conditions.
Despite the recent outpouring of interest in private mean estimation, privatizing U-statistics has received little attention. While existing private mean estimation algorithms can be applied to obtain confidence intervals, we show that they can lead to suboptimal private error, e.g., constant-factor inflation in the leading term, or even $\Theta(1/n)$ rather than $O(1/n^2)$ in degenerate settings.
To remedy this, we propose a new thresholding-based approach using \emph{local Hájek projections} to reweight different subsets of the data. This leads to nearly optimal private error for non-degenerate U-statistics and a strong indication of near-optimality for degenerate U-statistics.
\end{abstract}

\section{Introduction}
A standard task in statistical inference is to estimate a parameter of the form $\E[h(X_1,\dots,X_k)]$, where $h$ is a possibly vector-valued function and $\{X_i\}_{i=1}^n$ are i.i.d.\ draws from an unknown distribution. U-statistics are a well-established class of estimators for such parameters and can be expressed as averages of functions of the form $h(X_1,\dots, X_k)$. U-statistics arise in many areas of statistics and machine learning, encompassing diverse estimators such as the sample mean and variance, hypothesis tests such as the Mann-Whitney, Wilcoxon signed rank, and Kendall's tau tests, symmetry and uniformity testing~\citep{feuer77symmetry,diakonikolas2016collision}, goodness-of-fit tests~\cite{Vaart_1998}, counts of combinatorial objects such as the number of subgraphs in a random graph~\citep{Gilbert1961RandomPN}, ranking and clustering~\citep{clemenccon2008ranking, clemenccon2014statistical}, and subsampling~\cite{politis2012subsampling}.

Despite being a natural generalization of the sample mean, little work has been done on U-statistics under  differential privacy, 
in contrast to the rather sizable body of existing work on private mean estimation~\citep{KV18, KLSU19, CWZ19, KSSU19,biswas2020coinpress,Kamath2020PrivateME,duchi2023fast,brown2023fast,hardt2010geometry, CH12}. Ghazi et al.~\cite{Ustat2023neurips} and Bell et al.~\cite{bell2020private} focus on the setting of local differential privacy~\citep{kasiviswanathan2011can}; we are interested in privacy guarantees under the central model. Moreover, existing work on private U-statistics focuses on discrete data and relies on simple central differential privacy mechanisms (such as the global sensitivity mechanism \citep{dwork2006calibrating}), which are usually optimal in these settings.

Suitably scaled, many U-statistics converge to a limiting Gaussian distribution with variance $O(k^2/n)$; such a result is commonly used in hypothesis testing~\citep{Hoeffding1948Ustat,arcones93uprocess,Hoeffding1963ProbabilityIF}. However, there are also examples of non-degenerate U-statistics, which often arise in a variety of hypothesis tests~\cite{feuer77symmetry,Vaart_1998,diakonikolas2016collision} under the null hypothesis, where the U-statistic converges to a sum of centered chi-squared distributions~\citep{Gregory1977LargeST}. Another interesting U-statistic arises in subgraph counts in sparse random geometric graphs~\citep{Gilbert1961RandomPN}. When the probability of an edge being present tends to zero with $n$, constructing a private estimator by simply adding Laplace noise with a suitable scale may not be effective.

\textbf{Our contributions are:}

 1. We present a new algorithm for private mean estimation that achieves nearly optimal private and non-private errors for non-degenerate U-statistics with sub-Gaussian kernels.
 
2. We provide a lower bound for privately estimating non-degenerate sub-Gaussian kernels, which nearly matches the upper bound of our algorithm. We also derive a lower bound for degenerate kernels and provide evidence which suggests that the private error achieved by our algorithm in the degenerate case is nearly optimal. A summary of the utility guarantees of our algorithm and adaptations of existing private mean estimation methods is presented in Table~\ref{tab:rate_comparison}.

3. The computational complexity of our first algorithm scales as $O(\binom{n}{k})$. We generalize this algorithm to obtain an estimator based on subsampled data and provide theoretical guarantees for an algorithm with $O(n^2)$ computational complexity.

\begin{table}[!htb]
    \centering
    \begin{tabular}{|c|c|c|c|c|}
         \hline
         \textbf{Algorithm} & \multicolumn{2}{|c|}{\textbf{Sub-Gaussian, non-degenerate}} & \multicolumn{2}{|c|}{\textbf{Bounded, degenerate}} \\\hline
          & Private error & Is non-private error  & Private error & Is non-private error   \\
          \textbf{} &  & $O(\Var(U_n))$? & &$O(\Var(U_n))$?\\
          \hline
        Naive (Proposition~\ref{cor:naive})
          & $ \tilde{O}\bb{\frac{k\sqrt{\tau}}{n\epsilon}}$ &No&$\tilde{O}\bb{\frac{kC}{n\epsilon}}$&No\\[1ex]
          \hline
          All-tuples (Proposition~\ref{lem:alltuples}) & $\tilde{O}\bb{ \frac{k^{3/2}\sqrt{\proxy}}{n\epsilon}}$&Yes&$\tilde{O}\bb{ \frac{kC}{n\epsilon}}$&No\\[1ex]
          \hline
          Our algorithm &$\tilde{O} \bb{
 \frac{k \sqrt{\tau}}{n \eps}}$&Yes&$\tilde{O}\bb{ \frac{k^{3/2}C}{n^{3/2}\eps}}$&Yes\\
&Corollary~\ref{cor:nondegen}&&Corollary~\ref{cor:degen}&\\
\hline
Lower bound&$\Omega\bb{\frac{k \sqrt{\tau}}{n\eps} \sqrt{\log \frac{n\eps}{k}}}$&&$\Omega\bb{\frac{k^{3/2}C}{n^{3/2}\eps}}$&\\
 for private algorithms&Theorem~\ref{thm:lowernondegen}&& Theorem~\ref{thm:lowerdegen}&\\
          \hline
    \end{tabular}
    \vspace{4pt}
     \caption{
     We compare our application of off-the-shelf tools to Algorithm~\ref{alg:PrivateMeanDegenerate}. We only provide the leading terms in the private error. The non-private lower bound on $\E(\hat{\theta}-\E h(X_1,\dots,X_k))^2$ for all unbiased $\hat{\theta}$ is $\Var(U_n)$, which our private algorithms nearly match.}
    \label{tab:rate_comparison}
\end{table}

The paper is organized as follows: Section~\ref{SecBackground} reviews background on U-statistics and fundamental concepts in differential privacy. Section~\ref{sec:offtheshelf} presents a first set of estimators based on the CoinPress~\citep{biswas2020coinpress} algorithm for private mean estimation. Section~\ref{SecAtypical} presents a more involved algorithm using what we call \emph{local \hj projections}. A discussion of applications, including private hypothesis testing and estimation in sparse geometric graphs, is provided in Section~\ref{SecApplications}. Section~\ref{SecDiscussion} concludes the paper.


\section{Background and problem setup}
\label{SecBackground}
Let $n$ and $k$ be positive integers with $k \le n$. Let $\mathcal{D}$ be an unknown probability distribution over a set $\mathcal{X}$, and let $h:\mathcal{X}^k \to \R$ be a known function symmetric in its arguments\footnote{That is, $h(x_1, x_2, \dots, x_k) = h(x_{\sigma(1)}, x_{\sigma(2)}, \dots, x_{\sigma(k)})$ for any permutation $\sigma$. We do \emph{not} assume that the distribution $\mathcal{H}$ itself is symmetric.}. Let $\mathcal{H}$ be the distribution of $h(X_1, X_2, \dots, X_k)$, where $X_1, X_2, \dots, X_k \sim \mathcal{D}$ are i.i.d.\ random variables. We are interested in providing a $\eps$-differentially private confidence interval for the estimable parameter~\citep{halmos1946ustat} $\theta=\E[h(X_1, X_2, \dots, X_k)]$, which is the mean of the distribution $\mathcal{H}$, given access to $n$ i.i.d.\ samples from $\mathcal{D}$; we use $X_1, X_2, \dots, X_n$ to denote these $n$ samples.
We allow the kernel $h$, the degree $k$, and the estimable parameter $\theta$ to depend on $n$, but do not use the subscript $n$ for the sake of brevity, except in specific applications. 

We consider either bounded kernels, or  unbounded kernels $h$ where the distribution $\mathcal{H}$ is sub-Gaussian. We write $Y \sim \text{sub-Gaussian}(\proxy)$ if $\E[\exp(\lambda (Y-\E Y))]\leq \exp(\proxy \lambda^2/2)$ for all $\lambda \in \mathbb{R}$.

Throughout the paper, we will assume that the privacy parameter $\epsilon = O(1)$. We also use the notation $\tilde{O}(\cdot)$ in error terms, which hides polylogarithmic factors of $n/\alpha$.
A minimum variance unbiased estimator of $\theta$ is a U-statistic~\citep{Hoeffding1948Ustat, Lee1990UStatisticsTA}, which is defined in the following subsection. 

\subsection{U-Statistics}
Let $[n]$ denote $\{1,\dots, n\}$, and let $\I_{n,k}$ be the set of all $k$-element subsets of $[n]$. Denote the $n$ i.i.d.\ samples by $X_1, X_2, \dots, X_n$. For any $S \in \I_{n,k}$, let $X_S$ be the (unordered) $k$-tuple $\{X_i : i \in S\}$. The U-statistic associated with the data and the function $h$ is
\ba{\label{eq:ustat}
U_n \defeq \frac{1}{\binom{n}{k}} \sum_{\{i_1,\dots,i_k\}\in \I_{n,k}}h(X_{i_1},\dots, X_{i_k}).
}
The function $h$ is the \emph{kernel} of $U_n$ and $k$ is the \emph{degree} of $U_n$. While U-statistics can be vector-valued, we consider scalar U-statistics in this paper.
We also define conditional variances, which will be used to express the variance of $U_n$. For $c \in [k]$, we define the conditional variance
\ba{\label{eq:condvar}
\zeta_c:=\Var\bb{\E\bbb{h(X_1,\dots,X_k)|X_1,\dots,X_c}}.
}
Equivalently, $\zeta_c = \covar\bb{h(X_{S_1}), h(X_{S_2})}$ where $S_1, S_2 \in \I_{n,k}$ and $|S_1 \cap S_2| = c$. The number of such pairs of sets $S_1$ and $S_2$ is equal to $\binom{n}{k} \binom{k}{c}\binom{n-k}{k-c}$, which implies
\ba{\label{aeq:Ustatcovar}
\Var(U_n) = \binom{n}{k}^{-1} \sum_{c=1}^k \binom{k}{c} \binom{n-k}{k-c} \zeta_c.
}
Since $\E[U_n] = \theta$, $U_n$ is an unbiased estimate of $\theta$. Moreover, the variance of $U_n$ is a lower bound on the variance of any unbiased estimator of $\theta$. (cf Lee~\cite[Chapter 1, Theorem 3]{Lee1990UStatisticsTA}). We also have the following inequality from Serfling~\cite{serfling1980} (see also Appendix~\ref{lem:covar_inequality}):
\ba{\label{eq:varhierarchy}
\zeta_1 \le \frac{\zeta_2}{2} \le \frac{\zeta_3}{3} \le \dots \le \frac{\zeta_k}{k}.
}


\textbf{Infinite-order U-Statistics:} Classical U-statistics typically have small, fixed $k$. However, important estimators that appear in the contexts of subsampling~\citep{politis2012subsampling} and Breiman's random forest algorithm~\citep{song2019approximating,peng2019asymptotic} have $k$ growing with $n$. These types of U-statistics are sometimes referred to as \emph{infinite-order} U-statistics~\citep{frees1989inforder,minsker2023ustatistics}). U-statistics also frequently appear in the analysis of random geometric graphs~\citep{Gilbert1961RandomPN}. The difference between this setting and the examples above is that the conditional variances $\{\zeta_c\}$ vanish with $n$ in the sparse setting. (See Section~\ref{sec:sparsegraph}.)

\textbf{Degenerate U-statistics:}
A U-statistic is \emph{degenerate} of \emph{order} $\ell \le k-1$ if $\zeta_i=0$ for all $i\in [\ell]$ and $\zeta_{\ell+1} > 0$ (if $\zeta_k = 0$, the distribution is almost surely constant). Degenerate U-statistics arise in hypothesis tests such as Cramer-Von Mises and Pearson tests of goodness of fit~\citep{Gregory1977LargeST,Anderson1952AsymptoticTO}, ~\citep[Chapter 5]{Shorack2009EmpiricalPW} and tests for unformity~\cite{diakonikolas2016collision}. They also appear in tests for model misspecification in econometrics~\citep{li2020degenerate, linton2014}. For more examples of degenerate U-statistics, see~\cite{Wet1987DegenerateUA, weber1981incomplete, Ho2006TwostageUF}.

\subsection{Differential privacy}
Differential privacy (DP)~\citep{dwork2006calibrating} addresses the issue of privacy by requiring that the output of the algorithm be minimally affected by changing a single data point.
A randomized algorithm $M$, that takes as input a dataset $D \in \mathcal{X}^n$ and outputs an element of its range space $\mathcal{S}$, satisfies $\epsilon$-differential privacy if for any pair of adjacent datasets $D$ and $D'$ and any measurable subset $S \subseteq \mathcal{S}$ , the inequality $\Pr(M(D) \in S) \leq e^{\epsilon} \Pr(M(D') \in S)$ holds. 
Here, two datasets $D$ and $D'$ are adjacent if they differ in exactly one index.

Differentially private algorithms satisfy the following composition theorems, using which we can design sophisticated private algorithms by composing multiple smaller private algorithms.

\begin{lemma}[Basic composition]
\label{lem:basic_comp}
Let $\A_i : \mathcal{X}^n \times \prod_{j=1}^{i-1} \mathcal{Y}_i \to \mathcal{Y}_i$ for $i \in [k]$ be $k$ randomized algorithms such that for any $i \in [k]$ and any $(y_1, y_2, \dots, y_{i-1}) \in \prod_{j=1}^{i-1} \mathcal{Y}_j$, the algorithm $\A_i(\cdot, y_1, y_2, \dots, y_{i-1})$ is $\eps_i$-differentially private. Then, the algorithm $\mathcal{A}: \mathcal{X}^n \to \prod_{j=1}^k \mathcal{Y}_j$ that outputs the $k$-tuple
\bas{
\mathcal{A}(D) = (y_1, y_2, \dots, y_k),
}
where $y_i = \mathcal{A}_i(D, y_1, \dots, y_{i-1})$ for all $i \in [k]$, is $\sum_{i=1}^k \eps_i$-differentially private. 
\end{lemma}

\begin{lemma}[Parallel composition]
\label{lem:parallel_comp}
Let $\A_i : \mathcal{X}^n \to \mathcal{S}_i$ for $i \in [k]$ be $k$ $\epsilon$-differentially private algorithms. Then, the algorithm $\A: \mathcal{X}^{kn} \to \mathcal{S}_1 \times \dots \times \mathcal{S}_k$ that outputs the $k$-tuple
\bas{
\A\bb{X_1, \dots, X_{kn}} = \bb{\A_1\bb{X_1, \dots, X_n}, \A_2\bb{X_{n+1}, \dots, X_{2n}}, \dots, \A_k\bb{X_{(k-1)n+1}, \dots, X_{kn}}}
}
is $\epsilon$-differentially private.
\end{lemma}

\textbf{Basic DP algorithms.}
The global sensitivity of a function $f:\mathcal{X}^n \to \mathcal{S}$ is
\ba{\label{def:global_sensitivity}
GS(f) = \max_{|D \Delta D'| = 1} | f(D) - f(D')|,
}
where $D \Delta D' \defeq \bmb{\bcb{i : D_i \neq D'_i}}$. A fundamental result in differential privacy is that a private estimate of $f$ can be obtained by adding noise calibrated to its global sensitivity.
\begin{lemma}[Global sensitivity mechanism~\citep{dwork2006calibrating}]
\label{lem:laplace_mechanism}
Let $f:\mathcal{X}^n \to \mathcal{S}$ be a function and let $\epsilon > 0$ be the privacy parameter. Then the algorithm $\mathcal{A}(D) = f(D) + \Laplace\bb{{GS(f)}/{\epsilon}}$ is $\epsilon$-differentially private.\footnote{The Laplace Distribution $\Laplace\bb{b}$ with parameter $b > 0$ has distribution $\ell(z) =  \frac{1}{2b}\exp\bb{-\frac{|z|}{b}}.$} 
\end{lemma} 
The global sensitivity of a function is the largest difference in the function value between adjacent datasets, and may be high on atypical datasets. 
To account for the small sensitivity on ``typical'' datasets, the notion of {\em{local sensitivity}} is useful. 
The local sensitivity of a function $f:\mathcal{X}^n \to \mathcal{S}$ at $D$ is
\ba{\label{def:local_sensitivity}
LS(f, D) = \max_{|D \Delta D'| = 1} |f(D) - f(D')|.
}
Adding noise proportional to the local sensitivity does \emph{not} ensure differential privacy, because variation in the magnitude of noise itself may leak information. Instead, Nissim et al.~\cite{nissim2007smooth} proposed the notion of a smooth upper bound on $LS(f,D)$. 

A function $SS(f,\cdot)$ is said to be an $\eps$-smooth upper bound on the local sensitivity of $f$ if (i) $SS(f,D) \ge LS(f,D)$ for all $D$, and (ii) $SS(f,D) \le e^\epsilon SS(f,D')$ for all $|D \Delta D'| = 1$.
Intuitively, the first condition ensures that enough noise is added, and the second condition ensures that the noise is smooth and does not leak information about the data.
\begin{lemma}[Smoothed sensitivity mechanism~\citep{nissim2007smooth}]
\label{lem:smoothed_mechanism}
Let $f:\mathcal{X}^n \to \mathcal{S}$ be a function, $\epsilon > 0$ be the privacy parameter, and $SS(f,\cdot)$ be an $\eps$-smooth upper bound on $LS(f, \cdot)$. Then the algorithm $\mathcal{A}(D) = f(D) + \frac{10 \cdot SS(f,D)}{\epsilon}Z$, where $Z$ has density $h(z) \propto \frac{1}{1+z^4}$, is $\eps$-differentially private.
\end{lemma}

\subsection{Private mean estimation}
A fundamental task in private statistical inference is to privately estimate the mean based on a set of i.i.d.\ observations. One way to do this is via the global sensitivity method, wherein the standard deviation of the noise scales with the ratio between the range of the distribution and the size of the dataset. In the fairly realistic case where the range is large or unbounded, this leads to highly noisy estimation even in the setting where \textit{typical} samples are small in size.

To remedy this effect, a line of work~\citep{KV18, KLSU19, CWZ19, KSSU19, duchi2023fast, brown2023fast} has proposed better private mean estimators for (sub)-Gaussian vectors. Our work will build on one such method: CoinPress~\citep{biswas2020coinpress}. The idea is to iteratively refine an estimate for the parameters until one obtains a small range containing most of the data with high probability; noise is then added proportional to this smaller range. Note that some dependence on the range of the mean is inevitable for estimation with pure differential privacy~\citep{hardt2010geometry, CH12}.

\section{Lower bounds and application of off-the-shelf tools}
\label{sec:offtheshelf}

We begin with some initial methods for private U-statistic adaptations based on the CoinPress algorithm~\cite{biswas2020coinpress}. We also present a lower bound concerning the privacy terms in the minimax estimation error.



\subsection{Adaptations of the CoinPress algorithm for private estimation}

In this section, we operate under the following setting:
\paragraph{Setting.}Let $\X = \{X_i\}_{i \in [n]}$ be i.i.d.\ draws from $\mathcal{D}$. Assume the distribution of $h(X_1, \dots, X_k)$ is $\text{sub-Gaussian}(\proxy)$ for some known variance proxy $\proxy$,
with unknown mean $\theta \in (-R,R)$ for some known parameter $R$, and unknown variance $\zeta_k$.

A natural approach to this problem is to view it as a standard private mean estimation task: split the data into $n/k$ equally-sized chunks, apply the function $h$ to each chunk, and run any existing private mean estimation algorithm to these $n/k$ values. 
\begin{definition}\label{def:naive}
    Consider the following estimator: divide the $n$ data points into $n/k$ disjoint chunks, compute $h(X_S)$ on each of these chunks, and apply the CoinPress algorithm~\cite{biswas2020coinpress} to obtain a private estimate of the mean $\theta$. We will call this naive estimator $\hat{\theta}_{\text{naive}}$.
\end{definition}


The following proposition records the guarantee of the naive estimator $\hat{\theta}_{\text{naive}}$, the proof of the which can be found in Appendix~\ref{AppCorNaive}:

\begin{proposition}\label{cor:naive}
The naive estimator $\hat{\theta}_{\text{naive}}$ satisfies
    \bas{
|\tn-\theta|\leq O\bb{\sqrt{\frac{k\zeta_k}{n}}}+\tilde{O}\bb{\frac{k\sqrt{\proxy}}{n\epsilon}},
    }
with probability at least $0.9$, as long as $n = \tilde{\Omega}\bb{\frac{k }{\epsilon}\log \frac{R}{\sqrt{\tau}} }$. The estimate $\tn$ is $\epsilon$-differentially private and the algorithm runs in time $\tilde{O}\bb{n + \frac{n}{k} \log \frac{R}{\sqrt{\tau}}}$.
\end{proposition}

\begin{remark}
\label{RemSubOpt}
The optimal non-private unbiased estimator is the U-statistic defined in Eq~\eqref{eq:ustat} Lee~\citep[Chapter 1, Theorem 3]{Lee1990UStatisticsTA}.
Indeed, using Lemma~\ref{lem:Ustatvar}, we can see that the variance of a non-degenerate U-statistic is $k^2\zeta_1/n+O(k^2\zeta_k/n^2)$, which is smaller than the non-private part of the deviation provided in Proposition~\ref{cor:naive}, due to equation~\eqref{eq:varhierarchy}.
As an illustrative example, consider the problem of uniformity testing~\cite{diakonikolas2016collision} (see Section~\ref{SecApplications} for more details), where one has access to $n$ data points sampled i.i.d.\ from a discrete distribution with $m$ atoms. One often uses the U-statistic $\hat{U}_n=\sum_{i\neq j} 1(X_i=X_j)/{n\choose 2}$ for hypothesis testing. Consider a distribution $\{p_i\}_{i=1}^m$ such that $p_i=(1+a)/m$ for $i\in [1,m/2]$ and $p_i=(1-a)/m$ for $i\in (m/2,m]$, where $p_i=P(X_1=i)$. We show in Lemma~\ref{lem:unif_variance} that $$\zeta_1=\sum_{i<j}p_ip_j(p_i-p_j)^2=O\left(\frac{(1-a^2)a^2}{m^2}\right).$$ Note that $\sum_i p_i^2=1/m (1+a^2)$, so $$\zeta_2=\frac{1+a^2}{m} - \frac{(1+a^2)^2}{m^2} \geq \frac{1}{m}\left(1 - \frac{1}{m}\right),$$ which is $\Omega(m\zeta_1)$ even when $a\in(0,1)$. When $a\in\{0,1\}$, this is a degenerate U-statistic and $\zeta_1=0$. Thus, for large $m$, the terms $\zeta_1$ and $\zeta_2$ are indeed of different order.
\end{remark}

To reduce the non-private variance term, we need to compute the kernel function $h$ on overlapping blocks, as in the definition of a U-statistic.

In Algorithms~\ref{alg:PrivateMeanArxiv} and~\ref{alg:OneStepArxiv}, we present a generalization of the CoinPress algorithm~\citep{biswas2020coinpress}, originally proposed for private mean estimation of i.i.d.\ observations, which is then used to obtain a private estimate of $\theta$ with the non-private error term matching $\sqrt{\Var(U_n)}$. Propositions~\ref{lem:alltuples} and~\ref{lem:ss} establish privacy and utility of this algorithm depending on the family $\fm$ chosen: either $\fm = \I_{n,k}$, in which case we consider all $\binom{n}{k}$ tuples, or $\fm$ is a random subsample of $\I_{n,k}$. 

\begin{algorithm}[htb]
\caption{\label{alg:PrivateMeanArxiv} \textbf{{U-StatMean}}$\bb{n, k, \left \{h(X_S) \right\}_{S \in \fm}, \eps, R, \tau}$}
\begin{algorithmic}[1]
\STATE $M\gets |\fm|$
\STATE $\gamma \gets 0.01$
\STATE $t \gets \log\bb{\frac{R}{ \sqrt{2k \tau \log 4n/\gamma}}}$  
\STATE $[l_0, r_0] \gets [-R,R]$
\STATE $Y_{S} \gets h(X_{S})$ for $S \in \fm$
\FOR{$i = 1,2,...,t$}
\STATE $\{Y_{S}\}_{S \in \I_{n,k}}, [l_i, r_i] \gets \text{U-StatOneStep}\bb{n, k, \{Y_{S}\}, [l_{i-1}, r_{i-1}], \frac{\eps}{2t}, \frac{\gamma}{t}}$ 
\ENDFOR
\STATE $\{Y_{S}\}_{S \in \I_{n,k}}, [l_{t+1}, r_{t+1}] \gets \text{U-StatOneStep}\bb{n, k, \{Y_{S}\}, [l_t, r_t], \eps/2, \gamma/2}$
\RETURN $(l_{t+1}+r_{t+1})/{2}$
\end{algorithmic}
\end{algorithm}
\begin{algorithm}[htb]
\caption{\label{alg:OneStepArxiv} \textbf{U-StatOneStep}$\bb{n, k, \left\{Y_S \right\}_{S \in \fm}, [l,r], \eps, \beta}$}
\begin{algorithmic}[1]
\STATE $M\gets |\fm|$
\STATE $\kappa\gets \max_i\frac{|\{S\in \mathcal{S}:i\in S\}|}{M}$
\STATE $Y_S \gets \proj_{l-\sqrt{2k\tau \log 4n /\beta}, r+\sqrt{2k\tau \log 4n/\beta}}\bb{Y_S}$ for $S \in \fm$
\STATE $\Delta \gets \kappa\bb{r-l+2\sqrt{2k\tau \log 4n/\beta}}$
\STATE $Z \gets \frac{1}{M} \sum_{j=1}^{M} Y_j + W$, where $W \sim \text{Lap}\bb{\frac{\Delta}{\eps}}$
\STATE $[l,r] \gets \bbb{Z - \sqrt{\frac{16k \tau}{\min(M,n)}}\log \frac{4n}{\beta} - \frac{\Delta}{\eps} \log \frac{6}{\beta}, Z +  \sqrt{\frac{16k \tau}{\min(M,n)}} \log \frac{4n}{\beta} + \frac{\Delta}{\eps} \log \frac{6}{\beta}}$
\RETURN $\{Y_S\}_{S \in \I_{n,k}}, [l,r]$
\end{algorithmic}
\end{algorithm}
As stated, Algorithm~\ref{alg:PrivateMeanArxiv} has a constant success probability of $0.75$. To boost the probability of our stated error bounds from a constant error probability (e.g., $0.75$) to an arbitrarily high probability $1-\alpha$, for all algorithms in this paper, we employ the well-known median-of-means wrapper, which we incorporate into all of our theoretical results. For completeness, we state it here, with more details in Lemma~\ref{lem:mother_MoM}:

\begin{wrapper}[MedianOfMeans($n$, $k$, Algorithm $\mathcal{A}$, Parameters $\Lambda$, Failure probability $\bp$, Family type $f\in\{\text{all},\text{ss}\}$)]
\label{alg:MoM} 
    Divide $[n]$ into $q=\qa$ independent chunks $I_i,i\in[q]$. For each $i\in [q]$, 
    run Algorithm $\mathcal{A}$ with subset family $\SubS_i:=\fm_f(I_i)$, Dataset $\{h(X_S)\}_{S\in\SubS_i}$, and other parameters $\Lambda$ for $\mathcal{A}$ to output $\hat{\theta}_i,i\in [q]$. Return $\tilde{\theta}=\textup{median}(\hat{\theta}_1,\dots,\hat{\theta}_q)$.
\end{wrapper}

In the above wrapper,
$\fm_f(D_i)$ simply creates the appropriate family of subsets for the dataset $D_i$. For example, if $D_i=\{X_1,\dots X_{\na}\}$, $f=\text{all}$, then $\fm_{\text{all}}(D_i)$ is $\{h(X_S)\}_{S\in \mathcal{I}_{\na,k}}$. If $f=\text{ss}$, then $\fm_{\text{ss}}(D_i)$ is $\{h(X_S)\}_{S\in \fm_i}$, where $\fm_i$ is the set of $M$ subsampled subsets of $D_i$.
Here, $\na$ is defined as
\ba{\label{eq:nalpha}
\na=\frac{n}{\qa}.
} 

Before stating the first result for Algorithm~\ref{alg:PrivateMeanArxiv}, we formally define the ``all-tuples'' and the ``subsampled'' families $\fm_{\text{all}}$ and $\fm_{\text{ss}}$.
\begin{definition}[All-tuples family]\label{def:alltuples}
Let $M=\binom{n}{k}$ and let $\fm_{\text{all}} = \{S_1, S_2, \dots, S_M\} = \I_{n,k}$ be the set of all $k$-element subsets of $[n]$. Call $\fm_{\text{all}}$ the ``all-tuples'' family.
\end{definition}

\begin{definition}[Subsampled Family]\label{def:ss}
Draw $M$ i.i.d.\ samples $S_1,\dots, S_M$ from the uniform distribution over the elements of $\I_{n,k}$, and let $\fm_{\text{ss}} \defeq \bcb{S_1, \dots, S_M}$. Call $\fm_{\text{all}}$ the ``subsampled'' family.
\end{definition}



\begin{proposition}\label{lem:alltuples}
Suppose $\theta \in [-R,R]$. Let $\fm_{\text{all}}$ be the all-tuples family in Definition~\ref{def:alltuples}. Then Wrapper~\ref{alg:MoM}, with $f = \text{all}$, failure probability $\bp$, and $\mathcal{A}=\text{U-StatMean}$ (Algorithm~\ref{alg:PrivateMeanArxiv})
returns an estimate $\tall$ of the mean $\theta$ such that, with probability at least $1-\bp$,
\ba{ \label{eq:allpairs_error} |\tall-\theta|
&\leq O\bb{\sqrt{\Var(U_{\na})}} + \tilde{O}\bb{\frac{k^{3/2}\sqrt{\proxy}}{\na\epsilon}},
}
as long as $\na = \tilde{\Omega}\bb{\frac{k}{\epsilon}\log \frac{R}{\sqrt{k\tau}} }$. Moreover, the algorithm is $\eps$-differentially private and runs in time $\tilde{O}\bb{\log(1/\alpha)\bb{k + \log \frac{R}{\sqrt{k\tau}}}\binom{\na}{k}}$. 
\end{proposition}

The proof of Proposition~\ref{lem:alltuples} can be found in Appendix~\ref{AppLemAlltuples}.


\begin{remark}
\label{rem:coinpress}
The error term $\tilde{O}\bb{\sqrt{\var(U_n)}}$ in equation~\eqref{eq:allpairs_error} is within logarithmic factors of the error of any unbiased estimator of $\theta$, and the \textit{private} error term is a $\sqrt{k}$ factor worse than the lower bound $\tilde{\Omega}\bb{{k\sqrt{\tau}}/{n\eps}}$. Moreover,
we need $k^2/n = \tilde{O}(1)$ for the private error to be asymptotically smaller than the non-private error when $\epsilon$ is constant. Note, however, that existing concentration~\citep{Hoeffding1948Ustat,arcones93uprocess} or convergence in probability~\citep{Vaart_1998,minsker2023ustatistics} results for U statistics only require $k=o(n)$.
\end{remark}
\begin{remark}[Degenerate and sparse settings]
While Proposition~\ref{lem:alltuples} improves over the naive estimator, the private error can overwhelm the non-private error for degenerate and sparse U-statistics (see Section~\ref{SecApplications}). We also show that this estimator can lead to suboptimal sample complexity for uniformity testing in a neighborhood of the null hypothesis.
\end{remark}


Next, we state the result for the subsampled family $\fm_{\text{ss}}$.  Unlike the all-pairs family $\fm_{\text{all}}$, the subsampled $\fm_{\text{ss}}$ in Definition~\ref{def:ss} is randomized.  Define $ \tss = \sum_{j=1}^M h(X_{S_{j}})/M$. Recall from our discussion before (cf.\ Theorem~\ref{prop:mother}) that we want each of the $h(X_{S_j})$'s, as well as $\tss$, to concentrate around $\theta$, and we also want $\dep{n}{k}{\fm_{\text{ss}}}$ to be small. As we show, the former concentration holds as in the all-tuples case, and the latter holds with high probability. 


\begin{proposition}\label{lem:ss}
Let $M_n = \Omega((n/k)\log n)$.
Then Wrapper~\ref{alg:MoM}, with $f=\text{ss}$, failure probability $\bp$, algorithm $\mathcal{A}=\text{U-StatMean}$ (Algorithm~\ref{alg:PrivateMeanArxiv}), $\fm(I_i)$ a set of $M_{\na}$ i.i.d.\ subsets of size $k$ picked with replacement from $I_i$,
returns an estimate $\tpss$ such that, with probability at least $1-\bp$,

\bas{
|\tpss-\theta|\leq 
O\bb{\sqrt{\Var(U_{\na})}}+\tilde{O}\bb{\sqrt{\frac{\zeta_k}{M_{\na}}} +\frac{k^{3/2}\sqrt{\proxy}}{\na\epsilon}}
,
}

as long as 
$\na = \tilde{\Omega}\bb{\frac{k}{\epsilon}\bb{\log \frac{R}{\sqrt{k\tau}}}}$. Moreover, the estimator $\tpss$ is $\epsilon$-differentially private and runs in time $\tilde{O}\bb{\log(1/\alpha)\bb{k + \log \frac{R}{\sqrt{k\tau}}}M_{\na}}$.
\end{proposition}

The proof of Proposition~\ref{lem:ss} is contained in Appendix~\ref{AppLemSS}.

\subsection{Lower bound for non-degenerate kernels}

The next result shows a nearly optimal dependence on $n$ and $\epsilon$ in the bounds for $\tn$ and $\tall$. Furthermore, the dependence on $k$ for the error bound is optimal in the case of $\tn$ (but not in the case of $\tall$).

\begin{theorem}[Lower bound for non-degenerate kernels]\label{thm:lowernondegen}
Let $n$ and $k$ be positive integers with $k < n/2$ and let $\eps = \Omega(k/n)$.
Let $\fm$ be the set of all sub-Gaussian distributions over $\R$ with variance proxy $1$, and let $\tilde{\mu}$ be the output of any $\eps$-differentially private algorithm applied to $n$ i.i.d.\ observations from $\mathcal{D}$.
Then
$\sup_{h, \mathcal{D}: \mathcal{H} \in \fm} \mathbb{E} \bmb{\tilde{\mu}(X_1, \dots, X_n)-\E[h(X_1,\dots,X_k)]} = \Omega\bb{\frac{k}{n\eps} \sqrt{\log \frac{n\eps}{k}}}.$
\end{theorem}
We should note that another lower bound on $\mathbb{E} \bmb{\tilde{\mu}(X_1, \dots, X_n)-\E[h(X_1,\dots,X_k)]}$ is the error incurred by the best non-private estimator. Among all unbiased estimators, $U_n$ is the best non-private estimator~\cite{Hoeffding1948Ustat, Lee1990UStatisticsTA}. The most widely used non-private estimators are U- and V-statistics, which share similar asymptotic properties~\cite{Vaart_1998}.  The above lower bound also has a log factor that arises from an optimal choice of the sub-Gaussian proxy for Bernoulli random variables~\cite{arbel2020strict}. This shows that the logarithmic factors in the upper bounds in Propositions~\ref{cor:naive},~\ref{lem:alltuples}, and~\ref{lem:ss} are unavoidable.
The proofs are deferred to Appendix~\ref{AppThmLower}.

\section{Main results}
\label{SecAtypical}
In Section~\ref{sec:offtheshelf}, we showed that off-the-shelf private mean estimation tools applied to U-statistics either achieve a sub-optimal non-private error (see Remark~\ref{rem:coinpress}) or a sub-optimal private error. If the U-statistic is degenerate of order $1$, the non-private and private errors (assuming $\epsilon = \Theta(1)$) are $\tilde{\Theta}(1/n)$. This section presents our main algorithm, which achieves nearly optimal private error for sub-Gaussian non-degenerate kernels.  Our algorithm can be viewed as a generalization of the algorithm proposed in Ullman and Sealfon~\cite{ullman2019er} for privately estimating the edge density of an Erd\H{o}s-Rényi graph. We provide strong evidence that, for bounded degenerate kernels, we achieve nearly optimal non-private error. All proofs are in Section~\ref{sec:app-mainthm}.
\subsection{Key intuition}
Our key insight is to leverage the Hájek projection~\cite{Vaart_1998,Lee1990UStatisticsTA}, which gives the best representation of a U-statistic as a linear function of the form $\sum_{i=1}^n f(X_i)$:
\bas{
\hat{S}_n\stackrel{(i)}{=}\sum_{i=1}^n \E[T_n|X_i]-(n-1)\E [T_n] \stackrel{(ii)}{=}\frac{k}{n}\sum_{i=1}^n \E[h(X_S)|X_i]-(n-1)\theta.
}
Equality (i) gives the form of the \hj  projection for a general statistic $T_n$, whereas (ii) gives the form when $T_n$ is a U-statistic. Let $\I^{(i)}_{n,k}=\{S\in \I_{n,k}: i\in S\}$. In practice, one uses the estimates
\ba{\label{def:local_hajek}
\widehat{\E}[h(X_S)|X_i]\defeq \frac{1}{\binom{n-1}{k-1}}\sum_{S\in\I^{(i)}_{n,k}}h(X_S),
}
which we call \emph{local \hj projections}, with some abuse of notation. When the dataset is clear from context, we write $\h{\I_{n,k}}{X}{i}$, or simply $\h{}{X}{i}$, for $\widehat{\E}[h(X_S)|X_i]$.

\subsection{Proposed algorithm}
\label{SecDegen}
Consider a family of subsets $\mathcal{S} \subseteq \I_{n,k}$ of size $M$. Let $\mathcal{S}_i=\{S\in \mathcal{S}:i\in S\}$ and $M_i=|\mathcal{S}_i|$, and suppose $M_i \neq 0$ for all $i \in [n]$. Assume also that $\mathcal{S}$ satisfies the inequalities
\ba{
\label{EqnSubsampleSize}
\frac{M_i}{M} \le \frac{3k}{n}, \qquad \frac{M_{ij}}{M_i} \le \frac{3k}{n},
}
for any two distinct indices $i \neq j$ in $[n]$ (one such family is $\mathcal{S} = \I_{n,k}$, for which $M_i/M = k/n$ and $M_{ij}/M_i = (k-1)/(n-1) \le k/n$, but there are other such families). Define
\ba{\label{eq:Ustat}
\SubA(\SubS) \defeq \frac{1}{M} \sum_{S \in \SubS} h(X_S), \quad \text{ and } \quad \h{\SubS}{X}{i} \defeq \frac{1}{M_i}\sum_{S \in \SubS_i} h(X_S), \quad \forall i \in [n].
}
Note that $U_n$ in Eq~\eqref{eq:ustat} and $\widehat{\E}[h(X_S)|X_i]$ in equation~\eqref{def:local_hajek} are simply $A_n(\I_{n,k})$ and $\h{\I_{n,k}}{X}{i}$, respectively.

A standard sub-procedure in private mean estimation algorithms is to project or clip the data to an appropriate interval~\citep{biswas2020coinpress, Kamath2020PrivateME} in such a way that the sensitivity of the overall estimate can be bounded. In similar spirits, we use the concentration of the local \hj projections to define an interval such that each $i$ can be classified as ``good" or ``bad'' based on the distance between $\h{\SubS}{}{i}$ and the interval. The final estimator is devised such that the contribution of the bad indices to the estimator is low and the estimator has low sensitivity.

Let $\cb$ and $C$ be parameters to be chosen later; they will be chosen in such a way that with high probability, (i) $|\h{}{X}{i}-\theta| \le \cb$ for all $i$, and (ii) each $h(X_S)$ is at most $C$ away from $\theta$.
Define
\ba{\label{eq:buf_definition}
\buf{\SubS} \defeq \argmin_{t \in \mathbb{N}_{> 0}} \bb{ \bmb{\bcb{i: \left |\h{\SubS}{X}{i} - \SubA(\SubS) \right | > \cb +\frac{6kCt}{n}}} \le t }.
}
In other words, $\buf{\SubS}$ is the smallest positive integer $t$ such that at most $t$ indices $i \in [n]$ satisfy $|\h{\SubS}{X}{i} - \SubA(\SubS) | > \cb +\frac{6kCt}{n}$ (such an integer $t$ always exists because $t = n$ works).
Define
\ba{\label{eq:goodset}
\Good(\SubS) \defeq \bcb{i : \left |\h{\SubS}{X}{i}-\SubA(\SubS) \right | \le \cb +\frac{6kC\buf{\SubS}}{n}}, \qquad \Bad(\SubS) \defeq [n] \setminus \Good.
}
For each index $i \in [n]$, define the weight of $i$ with respect to $\SubS$ as
\ba{\label{eq:weight}
\wt_{\SubS}(i) \defeq \max\bb{0,1-\frac{\eps n}{6Ck} \cdot \text{dist}\bb{\h{\SubS}{X}{i}-\SubA, \bbb{-\cb-\frac{6kC\buf{\SubS}}{n}, \cb+\frac{6kC\buf{\SubS}}{n}}}}.
}
Here, $\eps$ is the privacy parameter and $\text{dist}(x, I)$ is the distance between $x$ and the interval $I$. 

Based on whether a datapoint is good or bad, we will define a weight scheme that reweights the $h(X_S)$ in equation~\eqref{eq:ustat}. For each $S \in \SubS$, let
\begin{equation*}
\wt_{\SubS}(S) \defeq \min_{i \in S} \wt_{\SubS}(i), \quad \text{and} \quad
g_{\SubS}(X_S) \defeq h(X_S)\wt_{\SubS}(S) + \SubA(\SubS) \bb{1-\wt_{\SubS}(S)}.
\end{equation*}
In particular, if $\wt_{\SubS}(S) = 1$, then $g_{\SubS}(X_S) = h(X_S)$; and if $\wt_{\SubS}(S) = 0$, then $g_{\SubS}(X_S) = \SubA(\SubS)$. Finally, define the quantities
\ba{\label{eq:fg}
\tSubA(\SubS) \defeq \frac{1}{M} \sum_{S \in \SubS} g_{\SubS}(X_S), \qquad
\g{\SubS}{X}{i} \defeq \frac{1}{M_i}\sum_{S \in \SubS_i} g_{\SubS}(X_S) ~~ \forall i \in [n].
}
To simplify notation, we will drop the argument $\SubS$ from $L$, $A_n$, $\tilde{A}_n$, $\hat{h}$, $\hat{g}$, $\Good$, and $\Bad$.

\begin{algorithm}[htb]
\caption{\label{alg:PrivateMeanDegenerate} 
\textbf{PrivateMeanLocalHájek}$\bb{n, k,  \{h(X_S), S\in\SubS\}, \epsilon, C,\cb, \SubS}$}
\begin{algorithmic}[1]
\STATE $M\gets \left|\SubS\right|$
\STATE $\SubS_i\gets \{S\in \SubS: i\in S\}$
\STATE $M_i\gets \left|\SubS_i\right|$
\IF{there exist indices $i \neq j$ such that $M_{i} = 0$ or $M_i/M > 3k/n$ or $M_{ij}/M_i > 3k/n,$}
\RETURN $\perp$
\ENDIF
\STATE $\SubA \gets \sum_{S \in \SubS} h(X_S)/M$
\FOR{$i = 1,2, \dots, n$}
\STATE $\h{}{X}{i} \gets \sum_{S \in \SubS_i} h(X_S)/M_i$
\ENDFOR
\STATE Let $\buf{}$ be the smallest positive integer such that $\bmb{\bcb{i : \bmb{\h{}{X}{i}-\SubA} > \cb + \frac{6kC \buf{}}{n}}} \le \buf{}$
\STATE $\Good(\SubS) \gets \bcb{i : \bmb{\h{}{X}{i}-\SubA} \le \cb + \frac{6kC \buf{}}{n}}; \Bad(\SubS) \gets [n]\setminus \Good(\SubS)$
\FOR{$i = 1,2,\dots, n$}
\STATE $\wt(i) \gets \max\bb{0,1-\frac{\eps}{6Ck/n}\text{dist}\bb{\h{}{X}{i}-\SubA,\bbb{-\cb-\frac{6kC\buf{}}{n}, \cb+\frac{6kC\buf{}}{n}}}}$
\ENDFOR
\FOR{$S \in \SubS$}
\STATE $g(X_S) \gets h(X_S) \min_{i \in S} \wt(i)
+ \SubA \bb{1-\min_{i \in S} \wt(i)}$
\ENDFOR
\STATE $\tSubA \gets \sum_{S \in \SubS} g(X_S)/M$
\STATE $\scriptstyle S(\SubS) \gets \max_{0 \le \ell \le n} e^{-\epsilon \ell} \bb{\frac{k}{n} \bb{\conc + \frac{kC(\buf{}+\ell)}{n}}\bb{1+\eps (\buf{}+\ell)} + \frac{k^2 C (\buf{}+\ell)^2 \min(k,(\buf{}+\ell))}{n^2}\bb{\eps+\frac{k}{n}} + \frac{k^2C}{n^2\eps}}$
\STATE Draw $Z$ from distribution with density $h(z) \propto 1/(1+|z|^4)$ 
\RETURN  $\tSubA + S(\SubS)/\epsilon \cdot Z$
\end{algorithmic}
\end{algorithm}

\begin{theorem}\label{thm:mainthm}
Algorithm \ref{alg:PrivateMeanDegenerate} is $O(\epsilon)$-differentially private for any $\xi$. Moreover, suppose $h$ is bounded with additive range $C$,\footnote{More precisely, suppose $\sup_x h(x) - \inf_x h(x) \le C.$} and with probability at least $0.99$, we have $\max_i|\h{\SubS}{}{i}-\SubA| \le \conc$. Run Wrapper~\ref{alg:MoM} with  $f=\text{all}$, and $\mathcal{A}=\text{PrivateMeanLocalHajek}$ (Algorithm~\ref{alg:PrivateMeanDegenerate}) to output $\tilde{\theta}$. With probability at least $1-\alpha$, we have
\bas{|\tilde{\theta} - \theta| = O\bb{\sqrt{\Var(U_{\na})} + \frac{k\cb}{\na\eps} + \bb{\frac{k^2}{\na^2 \eps^2}+\frac{k^3}{\na^3 \eps^{3}}}C}.}
\end{theorem}

The proof of Theorem~\ref{thm:mainthm} is contained in Appendix~\ref{SecAppMainThm}, as a variant of the more complicated proof of Theorem~\ref{thm:mainthm_subsample}.

\textbf{Idea behind the algorithm:} If all $\h{}{X}{i}$'s are within $\conc$ of the empirical mean $\SubA$, then $\Bad = \varnothing$ and $\tSubA = \SubA$. Otherwise, for any set $S$ containing a bad index, we replace $h(X_S)$ by a weighted combination of $h(X_S)$ and $\SubA$. This averaging-out of the bad indices allows a bound on the local sensitivity of $\tSubA$. We then provide a smooth upper bound on the local sensitivity characterized by the quantity $\buf{}$, which can be viewed as an indicator of how well-concentrated the data is. The choice of $\conc$ will be such that $\buf{} = 1$ with high probability and that the smooth sensitivity of $\tSubA$ at $\X$ is small. This ensures that a smaller amount of noise is added to $\tSubA$ while still preserving privacy.

\textbf{Connection to Ullman and Sealfon~\cite{ullman2019er}:}
An idea in Ullman and Sealfon~\cite{ullman2019er} is to use the strong concentration of the degrees of an Erd\H{o}s Renyi graph. This idea can be loosely generalized to a broader setting of U-statistics according to the following mapping: consider a hypergraph with $n$ nodes and ${n\choose k}$ edges, where the $i^{th}$ node corresponds to index $i$. An edge corresponds to a $k$-tuple of data points $S\in\I_{n,k}$, and the edge weight is given by $h(X_S)$.  The natural counterpart of a degree in a graph becomes a local \hj projection, defined as in equation~\eqref{def:local_hajek}. In degenerate cases and cases where $k^2\zeta_1 << k\zeta_k$, the local \hj projections are tightly concentrated around the mean $\theta$. We exploit this fact and reweight the edges ($k$-tuples) in such a way that the local sensitivity of the reweighted U-statistic is small.

\subsection{Application to different types of kernels}
Algorithm~\ref{alg:PrivateMeanDegenerate} can be extended from bounded kernels to $\text{sub-Gaussian}(\tau)$ kernels. First, split the samples into two roughly equal halves. The first half of the samples will be used to obtain a coarse estimate of the mean $\theta$. For this, 
we can use any existing private mean estimation algorithm (e.g., the naive estimator in Lemma~\ref{cor:naive}) to obtain an $\eps/2$-differentially private estimate $\tilde{\theta}_{\text{coarse}}$ such that with probability at least $1-\alpha$,
\bas{
|\tilde{\theta}_{\text{coarse}}-\theta| = \tilde{O}\bb{\sqrt{\frac{k\zeta_k}{n}} + \frac{k\sqrt{\tau}}{n\eps}}.
}
By a union bound, with probability at least $1-\alpha$, the kernel $h(X_S)$ is within $4\sqrt{\tau \log \left(\binom{n}{k}/\alpha\right)}$ of $\theta$ for all $S \in \I_{n,k}$, and therefore also within $c\sqrt{k \tau \log (n/\alpha)}$ of the coarse estimate $\tilde{\theta}_{\text{coarse}}$, for some universal constant $c$, as long as $\eps = \tilde{\Omega}(\sqrt{k}/n)$.  By a union bound, with probability at least $1-\alpha$, $h(X_S) \le 4\sqrt{k \tau \log (n/\alpha)}$ for all $S \in \I_{n,k}$.

Define the projected function $\tilde{h}(X_1, X_2, \dots, X_k)$ to be the value $h(X_1, X_2, \dots, X_k)$ projected to the interval 
$\left[\tilde{\theta}_{\text{coarse}}-c\sqrt{k \tau \log (n/\alpha)}, \tilde{\theta}_{\text{coarse}}+c\sqrt{k \tau \log (n/\alpha)}\right]$.
The final estimate of the mean $\theta$ is obtained by applying Algorithm~\ref{alg:PrivateMeanDegenerate} to the other half of the samples, the function $\tilde{h}$, and the privacy parameter $\eps/2$. The following lemma shows that $\sqrt{2\tau \log (2n/\alpha)}$ is a valid choice of the concentration parameter $\xi$ for sub-Gaussian, non-degenerate kernels. 
\begin{lemma}\label{lem:local_hajek_conc}
If $\mathcal{H}$ is $\text{sub-Gaussian}(\proxy)$, the local Hájek projections $\h{}{}{i}$ are also $\text{sub-Gaussian}(\proxy)$. In particular, with probability at least $1-\alpha$, we have
\bas{
\max_{1 \le i \le n} |\h{}{}{i}-\theta| \le \sqrt{2\tau \log \frac{2n}{\alpha}}.
}
\end{lemma}
Combining these parameters with Theorem~\ref{thm:mainthm} gives us the following result:

\begin{corollary}[Non-degenerate sub-Gaussian kernels]\label{cor:nondegen}
Suppose $h$ is sub-Gaussian$(\tau)$. Split the samples into two halves and compute a private estimate of the mean by applying the naive estimator (Definition~\ref{def:naive}) on the first half of the samples with privacy parameter $\eps/2$ to obtain $\tilde{\theta}_{\text{coarse}}$, where $\eps = \tilde{\Omega}(\sqrt{k}/n)$. Define $\tilde{h}$ to be the function $h$ projected to the interval $\tilde{\theta}_{\text{coarse}} \pm O(\sqrt{k\tau \log (2n/\alpha)})$. On the remaining half, run Wrapper~\ref{alg:MoM} with $f=\text{all}$, failure probability $\alpha/2$, algorithm $\mathcal{A}$ = PrivateMeanLocalHájek (Algorithm~\ref{alg:PrivateMeanDegenerate}) with $C = \sqrt{2k\tau \log(n/2\alpha)}$, $\conc = \sqrt{2\tau \log(n/2\alpha)}$, the function $\tilde{h}$, and privacy parameter $\eps/2$ to output $\tilde{\theta}$. With probability at least $1-\alpha$, we have
\bas{
|\tilde{\theta}-\theta| = O\bb{\sqrt{\Var(U_{\na})}} + 
\tilde{O} \bb{\frac{k \sqrt{\tau}}{\na \eps} + \frac{k^{2.5} \sqrt{\tau}}{\na^2 \eps^2} + \frac{k^{3.5} \sqrt{\tau}}{\na^3 \eps^3}}.
}
\end{corollary}
(In fact, Corollary~\ref{cor:nondegen} can also be applied to degenerate kernels, but the non-private error term is suboptimal.)
From our lower bound on non-degenerate kernels in Theorem~\ref{thm:lowernondegen}, we see that the private error in Corollary~\ref{cor:nondegen} is nearly optimal as long as $\eps = \tilde{\Omega}(k^{3/2}/n)$. In contrast, the private error in Proposition~\ref{lem:alltuples} is suboptimal in $k$.

Many degenerate U-statistics (e.g., all the degenerate ones in Section~\ref{SecApplications}) have bounded kernels. For these, we see that the local \hj projections concentrate strongly around the U-statistic.

\begin{lemma}\label{lem:degen_conc_hajek}
Suppose $\mathcal{H}$ is bounded, with additive range $C$. Let $i \in [n]$ be an arbitrary index and $S_i\in\I_{n,k}$ be a set containing $i$, and suppose $x_i \in \mathbb{R}$ is some element in the support of $\mathcal{D}$. With probability at least $1-\frac{\beta}{n}$, conditioned on $X_i=x_i$, we have 
\ba{
\label{eq:hajekconc}
\bmb{\widehat{\E}[h(X_S)|X_i=x_i]-\E\bbb{h(X_{S_i})|X_i=x_i}} \le 
2\sigma_i \sqrt{\frac{k}{n}\log\left(\frac{2n}{\beta}\right)}  + \frac{8Ck}{3n}\log \left(\frac{2n}{\beta}\right),
}
where $\widehat{\E}[h(X_S)|X_i=x_i]=\frac{\sum_{S\in \SubS_i}h(X_S)}{{n-1\choose k-1}}$, and $\sigma_i^2 = \Var\bb{h(X_{S_i})|X_i=x_i}$.
\end{lemma}
For bounded kernels with additive range $C$, we have $\sigma_i\leq \frac{C}{2}$ by Popoviciu's inequality~\cite{popoviciu1965certaines}. Moreover, for degenerate kernels, $\zeta_1 = 0$, that is, the conditional expectation $\E\bbb{h(X_{S_i})|X_i=x_i}$ is equal to $\theta$ for all $x_i$, because the variance of this conditional expectation is $\zeta_1$.

Therefore, by Lemma~\ref{lem:degen_conc_hajek}, the choice of $\conc=O\left(C\sqrt{\frac{k}{n}} \log\left(\frac{n}{\alpha}\right)\right)$ satisfies the requirement that the local Hájek projections are within $\conc$ of $\theta$ with probability at least $1-\alpha$. Based on this, we now present our result for degenerate bounded kernels:
\begin{corollary}[Degenerate bounded kernels]\label{cor:degen}
Suppose $h$ is bounded with additive range $C$, and $\zeta_1 = 0$. Run Wrapper~\ref{alg:MoM} with $f=\text{all}$, failure probability $\alpha$, and algorithm $\mathcal{A}$ = PrivateMeanLocalHájek (Algorithm~\ref{alg:PrivateMeanDegenerate}) with $\conc =O(C\sqrt{k/n} \log(n/\alpha))$, to output $\tilde{\theta}$.
Then, with probability $1-\alpha$,
\bas{\bmb{\tilde{\theta} - \theta} = O\bb{\sqrt{\Var(U_{n_\alpha})}} + \tilde{O}\bb{ \frac{k^{1.5}}{\na^{1.5}\eps}C + \frac{k^{2}}{\na^{2}\eps^2}C + \frac{k^3}{\na^3 \eps^{3}}C},}
with probability at least $1-\alpha$, as long as $\eps = \tilde{\Omega}(\sqrt{k}/\na)$.
\end{corollary}

Note that the private error is as smaller order than the non-private error as long as $\zeta_2\gg \frac{\sqrt{k/n}}{\epsilon}$.
Obtaining a result for sub-Gaussian degenerate kernels poses difficulties in bounding the concentration parameter $\xi$. However, for bounded kernels, we see that the above result obtains better private error than the application of off-the-shelf methods (Lemma~\ref{lem:alltuples}). In the next subsection, we provide a lower bound for degenerate bounded kernels which, together with Corollary~\ref{cor:degen}, strongly indicates that our algorithm achieves optimal private error for degenerate kernels. 

\subsection{Lower bound}

To compute a lower bound on the private error, we will construct a deterministic dataset and a kernel function such that the local Hájek projections are $1/\sqrt{n}$-concentrated around the corresponding U-statistic. Note that this is one way of characterizing a degenerate U-statistic. 

\begin{theorem}\label{thm:lowerdegen}
For any $n,k \in \mathbb{N}$ with $k \le n$,  $\eps = \Omega((k/n)^{1-1/(2k-2)})$, and $\eps$-differentially private algorithm $\mathcal{A}:\mathcal{X}^n \to \mathbb{R}$, there exists a function $h:\mathcal{X}^k \to \{0,1\}$ and dataset $D$ such that $|\h{}{}{i}-U_n| \le \sqrt{k/n}$ (where $\h{}{}{i}$ and $U_n$ are computed on $D$) for every $i \in [n]$,
and $\E |\mathcal{A}(D) - U_n| = \Omega(k^{3/2}/n^{3/2}\eps)$, where the expectation is taken over the randomness of $\mathcal{A}$.
\end{theorem}
\begin{remark}
The above lower bound is in some sense informal because we created a deterministic dataset and $h$ that mimics the property of a degenerate U-statistic---namely that the \hj projections concentrate around $U_n$ at a rate $\sqrt{k/n}$. However, it gives us a strong reason to believe that the private error cannot be smaller than $O\left(\frac{k^{3/2}}{n^{3/2}\epsilon}\right)$ for degenerate U-statistics of order $k$. Note that for bounded kernels, Corollary~\ref{cor:degen} does achieve this bound, as opposed to Proposition~\ref{lem:alltuples}.
\end{remark}
The proof is deferred to Appendix~\ref{AppThmLowerDeg}.


\subsection{Subsampling estimator}

We now focus on subsampled U-statistics. Previous work has shown how to use random subsampling to obtain computationally efficient approximations of U-statistics~\citep{janson1984asymptotic, politis2012subsampling, kato2019randomized}, where the sum is replaced with an average of samples (drawn with or without replacement) from $\mathcal{I}_{n,k}$. 

Recall Definition~\ref{def:ss}. Let $\SubS := \{S_1,\dots,S_M\}$ denote the subsampled set of subsets, let $\SubS_i := \{S\in \SubS:i\in S\}$, and let $M_i:=|\SubS_i|$.
The proof of Theorem~\ref{thm:mainthm} with $\SubS=\I_{n,k}$ (cf.\ Appendix~\ref{sec:app-mainthm}) uses the property of $\I_{n,k}$ that $\frac{M_i}{M} = \frac{k}{n}$ and $\frac{M_{ij}}{M_i} = \frac{k-1}{n-1}$, so the inequalities~\eqref{EqnSubsampleSize} certainly hold. Indeed, we can show that for subsampled data (cf.\ Lemma~\ref{lem:ss_mi_m}), the following inequalities hold with probability at least $1-\alpha$, provided $M = \Omega(n^2/k^2 \log (n/\alpha))$:
\ba{\label{eq:subsample_conc}
\frac{M_i}{M} \le \frac{3k}{n}, \qquad \text{ and } \qquad \frac{M_{ij}}{M_i} \le \frac{3k}{n}.
}
Algorithmically, we check if the bounds~\eqref{eq:subsample_conc} hold for $\SubS$, and output $\perp$ if not. Privacy is not compromised because the check only depends on $\SubS$ and is agnostic to the data. 



\begin{theorem}\label{thm:mainthm_subsample}
Let $M_{n} = \Omega\left((n^2/k^2)\log n \right)$. Then Algorithm \ref{alg:PrivateMeanDegenerate}, modified to output $\perp$ if the bounds~\eqref{eq:subsample_conc} do not hold, is $\epsilon$-differentially private. Moreover, suppose that with probability at least $0.99$, we have $\max_i|\h{\SubS}{}{i}-\SubA| \le \conc$ and $|h(S)-\theta| \le C$ for all $S \in \I_{n,k}$. Run Wrapper~\ref{alg:MoM} with $f=\text{ss}$, failure probability $\alpha$, and $\mathcal{A}=\text{PrivateMeanLocalHajek}$ (Algorithm~\ref{alg:PrivateMeanDegenerate})  to output $\tilde{\theta}$. With probability at least  $1-\alpha$, we have
\bas{\bmb{\mathcal{A}(\X) - \theta} = O\bb{\sqrt{\Var(U_{\na})} + \sqrt{\frac{\zeta_k}{M_{\na}}} + \frac{k\conc}{\na\eps} + \bb{\frac{k^2 C}{\na^2 \eps^{2}}+\frac{k^3 C}{\na^3 \eps^{3}}}\min\bb{k, \frac{1}{\eps}}}.}
\end{theorem}

The proof of Theorem~\ref{thm:mainthm_subsample} is contained in Appendix~\ref{AppThmMain}.


\begin{remark}
If the kernel is non-degenerate and the number of times we subsample (for each run of the algorithm) is $\tilde{\Omega}\bb{\na^2/k^2}$, then Theorem~\ref{thm:mainthm_subsample} nearly achieves the same error as Algorithm~\ref{alg:PrivateMeanDegenerate} with $\SubS=\I_{n,k}$ with a better computational complexity for $k \ge 3$. The lower-order terms have an additional $\min(k,1/\epsilon)$ factor, which can be removed with $\Omega(n^3)$ subsamples. 
\end{remark}

\vspace{-.5em}
\section{Applications}
\label{SecApplications}
\vspace{-.5em}
We now discuss several applications illustrating the versatility of our algorithmic framework, and highlighting the differences between the guarantees of various algorithms we have proposed in this paper.

\textbf{1. Uniformity testing:} 
In uniformity testing, one tests whether a discrete distribution is uniform over its support $[m]$ or sufficiently far from uniform~\cite{diakonikolas2016collision}. The test statistic $\sum_{i < j} \mathbbm{1}(X_i=X_j)/{n\choose 2}$ estimates $\E[\mathbbm{1}(X_1=X_2)] \defeq P(X_1=X_2)$. It is a U-statistic and is degenerate under the uniform distribution.

In this section, suppose we have $n$ i.i.d.\ samples $X_1, X_2, \dots, X_n$ from a discrete distribution with support $[m]$, characterized by the probability masses $p_1, p_2, \dots, p_m$ on the atoms. 
We want to test closeness to the uniform distribution. Formally, given some error tolerance $\delta > 0$, we wish to distinguish between the cases $\left\{p: \sum_i\left(p_i-\frac{1}{m}\right)^2 \le \frac{\delta^2}{2m}\right\}$ (approximate uniformity) and $\left\{p: \sum_i\left(p_i-\frac{1}{m}\right)^2> \frac{\delta^2}{m}\right\}$. We will restrict ourselves to the class of distributions where $p_i \le \frac{2}{m}$ for all $i$, so we can write  $p_i=(1+a_i)/m$, for $a_i\in [-1,1]$ such that $\sum_i a_i=0$. 
(the upper bound of 1 on the $a_i$'s is somewhat arbitrary, and any other choice of constant works.) Thus, we have $\delta < 1$.

Without the constraint of privacy, Diakonikolas et al.~\cite{diakonikolas2016collision} perform this test by rejecting the uniformity hypothesis whenever $U_n \defeq \sum_{i < j} \mathbbm{1}(X_i=X_j)/{n\choose 2} > (1+3\delta^2/4)/{m}$, and show that this test succeeds with probability $0.9$ as long as $n = \Omega\bb{m^{1/2}/\delta^2}$. We derive guarantees for a private version of this test. As detailed in Algorithm~\ref{alg:PrivateUniform}, instead of using the statistic $U_n$, we use the private estimate $\tilde{U}_n$ of $U_n$ from Algorithm~\ref{alg:PrivateMeanDegenerate} and perform the same comparison as above.

\begin{algorithm}[htb]
\caption{\label{alg:PrivateUniform} \textbf{PrivateUniformityTest}$\bb{n, m, \X = \left \{X_i \right\}_{i \in [n]}, \epsilon}$}
\begin{algorithmic}[1]
\STATE $C \gets 1$, $\gamma\gets 0.01$
\STATE $\conc\gets \frac{6}{m} + \frac{8 \log (4n/\gamma)}{n}$
\STATE $\SubS\gets \{(i,j):1\leq i< j\leq n\}$
\STATE $\tilde{\theta}\gets$ \textbf{PrivateMeanHájek}$\bb{n, 2,   \{1(X_i=X_j), (i,j)\in\SubS\}, \epsilon, \alpha, C,\conc, \SubS}$
\IF{$\tilde{\theta}\geq \frac{1+3\delta^2/4}{m}$}  
\STATE $\text{DEC} \gets 1$ \COMMENT{\textit{Reject approximate uniformity}}
\ELSE 
\STATE \text{DEC} $\gets 0$ \COMMENT{\textit{Accept approximate uniformity}}
\ENDIF
\RETURN \text{DEC}
\end{algorithmic}
\end{algorithm}



\rd 
\bk

\begin{theorem}\label{thm:uniformity_testing}
Consider the parametric family $p_i=(1+a_i)/m$ with $a_i\in[-1,1]$, $\sum_{i=1}^m a_i=0$. Let $\{X_j\}_{j=1}^n$ be i.i.d.\ multinomial random variables such that $P(X_1=i)=p_i$, for all $i\in [m]$.  The Wrapper~\ref{alg:MoM} with $f=\text{all}$, failure probability $\alpha$, and algorithm $\mathcal{A}$ = PrivateUniformityTest (Algorithm~\ref{alg:PrivateUniform})   distinguishes between $\frac{\|a\|^2}{m^2} \geq \frac{\delta^2}{m}$ from $\frac{\|a\|^2}{m^2}< \frac{\delta^2}{2m}$ with probability at least $1-\alpha$, as long as $\na = {\Omega}\bb{\frac{m^{1/2}}{ \delta^2} + \frac{m^{1/2}}{(\delta \eps)} + \frac{m^{1/2} \log(m/\delta \eps)}{\delta \eps^{1/2}} +  \frac{m^{1/3}}{\delta^{2/3} \eps} + \frac{1}{\delta^2}}$. 
\end{theorem}

\begin{remark}
First, note that the non-private term is the same as that in ~\cite{diakonikolas2016collision} Theorem 1. Suppose we use Algorithm~\ref{alg:PrivateMeanArxiv} with the all-tuples family instead of Algorithm~\ref{alg:PrivateMeanDegenerate} for the hypothesis test. Then the private error is $\tilde{O}(\frac{1}{n\eps})$ because the sub-Gaussian parameter is $\tilde{O}(1)$ (see Proposition~\ref{lem:alltuples}). This private error term is $O(\delta^2/m)$ only when $n = \Omega(\frac{m}{\delta^2\eps})$. 
In comparison, Algorithm~\ref{alg:PrivateGraph} has error $O(\delta^2/m)$ if $n = \Omega\bb{\frac{\sqrt{m}}{\delta\min(\delta,\epsilon)}}$, a quadratically better dependence on $m$ than the bound achieved by Algorithm~\ref{alg:PrivateMeanArxiv}. 
\end{remark}
The proof of Theorem~\ref{thm:uniformity_testing} appears in the Appendix Section~\ref{sec:uniformity}.

\textbf{2. Sparse graph statistics:}
\label{sec:sparsegraph}

The geometric random graph (see Gilbert~\citep{Gilbert1961RandomPN}) has adjacency matrix $A\in \{0,1\}^{n\times n}$ with entries $A_{ij}=\mathbbm{1}(\|X_i-X_j\|_2\leq r_n)$, $1<i<j<n$, where the $X_i$'s are suitably distributed in some $d$-dimensional space (note that the latent positions $\{X_i\}_{i\in [n]}$ are unobserved). Typically, we only observe the graph and do not know the underlying distribution or the radius $r_n$. This is why estimates of the network moments are of interest since they reveal information about the underlying unknown distribution and parameters. 

For concreteness, assume that $X_i\in \partial \mathbb{S}^{3}$ for all $i \in [n]$, where $\mathbb{S}^{3}$ is the three-dimensional sphere of radius $1$, $\partial \mathbb{S}^{3}$ is the surface of this ball, and $r_n$ is the radius that governs the sparsity of the graph. We want to estimate the triangle density of this graph.

Assuming such a distribution in contrast to one supported on the unit Euclidean ball avoids boundary discrepancies, making the analysis easier. To count the triangle density, we take $g(x,y,z)=\mathbbm{1}(\text{$x,y,z$ are pairwise at most $r_n$ apart})$. The parameter of interest is
\ba{\label{eq:expectation}
\theta_n=\E\bbb{g(X_1,X_2,X_3)},
}
and the corresponding U-statistic is
\bas{
U_n=\sum_{1\leq i<j<k\leq n} \frac{g(X_i,X_j,X_k)}{{n\choose 3}}=\sum_{1\leq i<j<k\leq n} \frac{A_{ij}A_{jk}A_{ik}}{{n\choose 3}}.
}
Recall from Eq~\eqref{aeq:Ustatcovar} that the variance of $U_n$ can be expressed as a linear combination of the conditional variances $\zeta_1, \zeta_2, $ and $\zeta_3$. It is clear $\E \bbb{g(X_i,X_j,X_k)|X_i}$ does not depend on $X_i$, that is 
\bas{
\zeta_1 =\var \bb{\E \bbb{g(X_i,X_j,X_k)|X_i}}=0.
}
Therefore, the U-statistic is degenerate.
For $\zeta_2$, the conditional expectation $\E \bbb{g(X_i,X_j,X_k)|X_i,X_j}$ depends on the distance between $X_i$ and $X_j$. If $X_i = X_j$, then $\E \bbb{g(X_i,X_j,X_k)|X_i, X_j}$ equals the probability that $X_k$ is in the spherical cap $\{y:\|x-y\|\leq r_n, y\in \partial \mathbb{S}^{3}\}$, the area of which equals $\pi r_n^2$. On the other hand, if $X_i$ and $X_j$ are exactly $r_n$ apart, then the area (and thus the probability) is smaller (see Figure~\ref{fig:intersect}).

\begin{figure}
    \centering
   \includegraphics[width=0.3\textwidth]{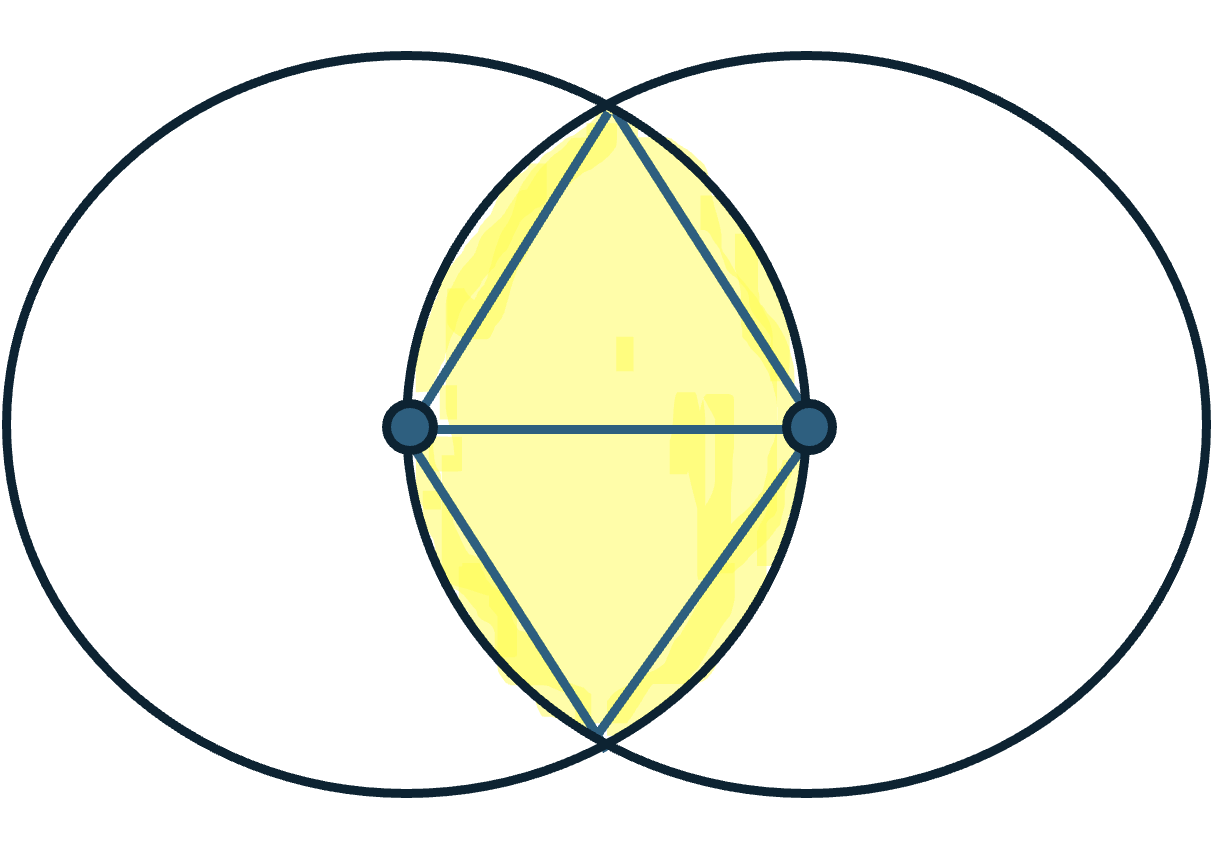}
    \caption{Conditional probability of a triangle given two vertices are $r_n$ distance away.}
    \label{fig:intersect}
\end{figure}

A lower bound on the conditional expectation is achieved when $\|X_i-X_j\|=r_n$. This is proportional to the area of the intersection of two circles of radius $r_n$, whose centers are $r_n$ away from each other (see the shaded region in Figure~\ref{fig:intersect}). We can further lower-bound this area by $\sqrt{3} r_n^2/2$, which is twice the area of an equilateral triangle with each side of length $r_n$. Since the surface area of a sphere of radius $1$ is $4\pi$, we have
\ba{\label{eq:condexp}
\E \bbb{g(X_i,X_j,X_k)|X_i,X_j}\in \bbb{\frac{\sqrt{3}}{8\pi} r_n^2, \frac{1}{4}r_n^2} \mathbbm{1}(\|X_i-X_j\|\leq r_n).
}

From this, we will determine $\E[g(X_i,X_j,X_k)|X_i=x]$ up to constants.  Integrating Eq~\ref{eq:condexp}, we see that there exists a universal constant $c_1$ such that for any $x \in \partial \mathbb{S}^{3}$ and any three distinct indices $i,j,k$ in $[n]$,
\bas{
\E \bbb{g(X_i,X_j,X_k)|X_i=x} =c_1 r_n^4.
}
Next, we compute $\zeta_2$.
Since the variance of any distribution on $[a,b]$ is bounded by $(b-a)^2/4$ (Popoviciu's inequality~\cite{popoviciu1935approximation}), we have
\bas{
\zeta_2 \le \frac{(r_n^2/4)^2}{4} = \frac{r_n^4}{64}, ~~\text{ and }~~ \zeta_3 = \Theta(\E[g(X_i,X_j,X_k)]) = \Theta(r_n^4).
}
The second equality follows because $g(X_1,X_2,X_3)$ is a Bernoulli random variable with expectation $\Theta(r_n^4)$ (see equation~\eqref{eq:expectation}). Using equation~\eqref{aeq:Ustatcovar}, it follows that, for large enough $n$, we have
\ba{\label{eq:sparsevar}
\var(U_n)\leq \frac{18}{(n-1)(n-2)} \cdot \frac{r_n^4}{64}+\frac{\pi}{2} \cdot \frac{r_n^4}{{n\choose 3}}\leq \frac{r_n^4}{n^2}.
}
To apply Algorithm~\ref{alg:PrivateMeanDegenerate}, we also need the concentration parameter $\xi$, which, from Lemma~\ref{lem:degen_conc_hajek}, depends on 
\ba{\label{eq:sigma-sparsegraph}
\max_{1 \le i \le n}\sigma_i^2=\max_{1 \le i \le n}\var(g(X_i,X_j,X_k)|X_i) \leq \frac{\pi r_n^4}{4}.
} However, we do not have access to $r_n$. 
Therefore, we first obtain a crude estimate of $r_n$ from the edge density $P(\|X_i-X_j\|\leq r_n)$, which is equal to $ (\pi r_n^2)/(4\pi) = r_n^2/4$, the ratio of the area of a spherical cap  $\{y:\|x-y\|\leq r_n, y\in \partial \mathbb{S}^{3}\}$ and the surface area of $\mathbb{S}^3$. Step~\ref{step:crude_estimate} in Algorithm~\ref{alg:PrivateGraph} computes a private estimate of an upper bound on $\max_i \sigma_i^2$ as $\nu^2$, where $\nu := U'_n+\frac{2}{n\epsilon} Z$, with $U'_n=\sum_{i<j} A_{ij}/{n\choose 2}$ and $Z$ a standard Laplace random variable.

\begin{algorithm}
\caption{\label{alg:PrivateGraph} \textbf{PrivateNetworkTriangle}$\bb{n, m, \{A_{ij}\}_{1\leq i<j\leq n}, \epsilon}$}
\begin{algorithmic}[1]
\STATE $C \gets 1$, $\gamma\gets 0.01$
\STATE $U_n'\gets \frac{1}{{n\choose 2}}\sum_{i<j}A_{ij}$
\STATE $\nu\gets U_n'+\frac{1}{n\epsilon}Z$, where $Z$ is a standard Laplace random variable \label{step:crude_estimate}
\IF{$\nu < 0$}
\RETURN $\perp$
\ENDIF
\STATE $\conc\gets 18\nu \sqrt{\frac{2}{n}\log\left(\frac{2n}{\gamma}\right)}  + \frac{16}{3n}\log \left(\frac{2n}{\gamma}\right) + \frac{9\nu}{n} \sqrt{\frac{2}{\gamma}}$
\STATE $\SubS\gets \{(i,j,k):1\leq i< j<k\leq n\}$
\STATE $\tilde{\theta}\gets$ \textbf{PrivateMeanHájek}$\bb{n, 2,   \{1(A_{ij}=A_{ik}=A_{jk}=1), (i,j,k)\in\SubS\}, \epsilon, C,\conc, \SubS}$
\RETURN $\tilde{\theta}$
\end{algorithmic}
\end{algorithm}

\begin{theorem}\label{cor:triangle}
Let $r_n=\tilde{\Omega} (n^{-1/2})$, $\eps = \Omega\bb{\frac{1}{nr_n^2}}$. Let $\{X_1,\dots, X_n\}$ be i.i.d.\ latent positions such that $X_i$ is distributed uniformly on $\partial \mathbb{S}^{3}$. Let the observed geometric network have adjacency matrix $\{A_{ij},1\leq i<j\leq n\}$ where $A_{ij}=1(\|X_i-X_j\|\leq r_n)$. 
Let $\tilde{\theta}$ be the output of Wrapper~\ref{alg:MoM} with $f=\text{all}$, failure probability $\alpha$, and algorithm $\mathcal{A}$ = PrivateNetworkTriangle (Algorithm~\ref{alg:PrivateGraph}). 
We have
 \bas{
 |\tilde{\theta}-\theta| = \tilde{O}\bb{\frac{r_n^2}{\na} + \frac{1}{\na^2\eps^2}},
 }
 with probability at least $1-\alpha$. Moreover, $\tilde{\theta}$ is $O(\eps)$-differentially private.
\end{theorem}
The proof of Theorem~\ref{cor:triangle} is deferred to Appendix~\ref{sec:triangle}.

\begin{remark}[Comparison with the all-tuples estimator]
By Proposition~\ref{lem:alltuples}, the all-tuples estimator (Algorithm~\ref{alg:PrivateMean}) satisfies 
\bas{|\tall-\theta| = \tilde{O}\bb{\frac{r_n^2}{n} + \frac{\sqrt{\proxy}}{n\epsilon}},
    }
with probability at least $1-\alpha$, where $\tau$ is the variance proxy of the distribution. Since $\tau=\Omega\bb{\log 1/r_n}$~\cite{arbel2020strict}, the private error overpowers the main variance term in sparse settings where $r_n = o(1)$. In comparison, when $r_n=\tilde{\Omega}(1/\eps\sqrt{n} + 1/\eps^{3/2}n)$, the private error from Algorithm~\ref{alg:PrivateGraph} (see Theorem~\ref{cor:triangle}) is of smaller order than the non-private error, for constant $\epsilon$.
\end{remark}

\textbf{3. Other examples}

We briefly mention a few additional examples that fit readily into our algorithmic framework, where we omit the details of specific corollaries that may be stated in such settings.

~~~\textbf{3a. Goodness-of-fit testing:}
The Cramer-Von Mises statistic for testing the hypothesis that the cumulative distribution function of a random variable is equal to a function $F_0$ is given by
\begin{equation*}
\frac{1}{n} \sum_{i=1}^n \sum_{j=1}^n \int \left(1\{X_i \le x\} - F_0(x)\right) \left(1\{X_j \le x\} - F_0(x)\right) dF_0(x).
\end{equation*}
Under the null $H_0: X \sim F_0$, the distribution of the statistic is degenerate~\citep{Vaart_1998}. Thus, our techniques from Section~\ref{SecDegen} provide a method for private goodness-of-fit testing based on the Cramer-Von Mises statistic. Private goodness-of-fit testing has so far mostly been studied in the setting of discrete data~\citep{gaboardi2016differentially, acharya2018differentially, aliakbarpour2019private}. For continuous distributions, we are only aware of work that analyzes the local DP framework~\citep{dubois2019goodness, lam2022minimax, butucea2023interactive}, which is therefore not directly comparable to our proposed approach.

~~~\textbf{3b. Pearson's chi-squared test:} The chi-squared goodness of fit test is widely used to test if a discrete random variable comes from a given distribution. The corresponding statistic (which can be written as a U-statistic plus a smaller order term) is degenerate~\cite{Wet1987DegenerateUA}.

~~~\textbf{3c. Symmetry testing:} Testing the symmetry of the underlying distribution of i.i.d.\ $X_1,\dots, X_n$ is often used in paired tests.
Feuerverger and Mureika~\cite{feuer77symmetry} use the test statistic $\sum_{i,j} (g(X_i-X_j)-g(X_i+X_j))/n^2$ (which is a U-statistic plus a lower-order term), where $g$ is the characteristic function of some distribution symmetric around $0$. When the distribution of $X_i$ is symmetric, this is degenerate.


\section{Discussion}

\label{SecDiscussion}
We have considered the problem of estimating $\theta:=\E h(X_1,\dots,X_k)$ for a broad class of kernel functions $h$. The best non-private unbiased estimator is a U-statistic, which is widely used in estimation and hypothesis testing. While existing private mean estimation algorithms can be used for this setting, they can be suboptimal for large $k$ or for non-degenerate U-statistics, which have $O(1/n)$ limiting variance. We provide lower bounds for both degenerate and non-degenerate settings. We analyze bounded degenerate kernels motivated by typical applications with degenerate U-statistics. Extending this to the sub-Gaussian setting is part of future work.
We propose an algorithm that matches our lower bounds for sub-Gaussian non-degenerate kernels and bounded degenerate kernels. We also provide applications of our theory to private hypothesis testing and estimation in sparse graphs. 

\bibliography{refs}

\begin{thebibliography}{10}

\bibitem{acharya2018differentially}
J.~Acharya, Z.~Sun, and H.~Zhang.
\newblock Differentially private testing of identity and closeness of discrete distributions.
\newblock {\em Advances in Neural Information Processing Systems}, 31, 2018.

\bibitem{aliakbarpour2019private}
M.~Aliakbarpour, I.~Diakonikolas, D.~Kane, and R.~Rubinfeld.
\newblock Private testing of distributions via sample permutations.
\newblock {\em Advances in Neural Information Processing Systems}, 32, 2019.

\bibitem{Anderson1952AsymptoticTO}
T.~W. Anderson and D.~Darling.
\newblock Asymptotic theory of certain "goodness of fit" criteria based on stochastic processes.
\newblock {\em Annals of Mathematical Statistics}, 23:193--212, 1952.

\bibitem{arbel2020strict}
J.~Arbel, O.~Marchal, and H.~D. Nguyen.
\newblock On strict sub-{G}aussianity, optimal proxy variance and symmetry for bounded random variables.
\newblock {\em ESAIM: Probability and Statistics}, 24:39--55, 2020.

\bibitem{arcones93uprocess}
M.~A. Arcones and E.~Gine.
\newblock Limit theorems for {$U$}-processes.
\newblock {\em The Annals of Probability}, 21(3):1494 -- 1542, 1993.

\bibitem{bell2020private}
J.~Bell, A.~Bellet, A.~Gasc{\'o}n, and T.~Kulkarni.
\newblock Private protocols for {U}-statistics in the local model and beyond.
\newblock In {\em International Conference on Artificial Intelligence and Statistics}, pages 1573--1583. PMLR, 2020.

\bibitem{bernstein1924modification}
S.~Bernstein.
\newblock On a modification of {C}hebyshev’s inequality and of the error formula of laplace.
\newblock {\em Ann. Sci. Inst. Sav. Ukraine, Sect. Math}, 1(4):38--49, 1924.

\bibitem{biswas2020coinpress}
S.~Biswas, Y.~Dong, G.~Kamath, and J.~Ullman.
\newblock Coin{P}ress: {P}ractical private mean and covariance estimation.
\newblock {\em Advances in Neural Information Processing Systems}, 33:14475--14485, 2020.

\bibitem{brown2023fast}
G.~Brown, S.~Hopkins, and A.~Smith.
\newblock Fast, sample-efficient, affine-invariant private mean and covariance estimation for sub-{G}aussian distributions.
\newblock In {\em The Thirty Sixth Annual Conference on Learning Theory}, pages 5578--5579. PMLR, 2023.

\bibitem{butucea2023interactive}
C.~Butucea, A.~Rohde, and L.~Steinberger.
\newblock Interactive versus noninteractive locally differentially private estimation: {T}wo elbows for the quadratic functional.
\newblock {\em The Annals of Statistics}, 51(2):464--486, 2023.

\bibitem{CWZ19}
T.~T. Cai, Y.~Wang, and L.~Zhang.
\newblock The cost of privacy: {O}ptimal rates of convergence for parameter estimation with differential privacy.
\newblock {\em The Annals of Statistics}, 49(5):2825--2850, 2021.

\bibitem{CH12}
K.~Chaudhuri and D.~Hsu.
\newblock Convergence rates for differentially private statistical estimation.
\newblock In {\em Proceedings of the International Conference on Machine Learning}, volume 2012, page 1327. NIH Public Access, 2012.

\bibitem{kato2019randomized}
X.~Chen and K.~Kato.
\newblock Randomized incomplete {$U$}-statistics in high dimensions.
\newblock {\em The Annals of Statistics}, 47(6):3127 -- 3156, 2019.

\bibitem{clemenccon2014statistical}
S.~Cl{\'e}men{\c{c}}on.
\newblock A statistical view of clustering performance through the theory of {U}-processes.
\newblock {\em Journal of Multivariate Analysis}, 124:42--56, 2014.

\bibitem{clemenccon2008ranking}
S.~Cl{\'e}men{\c{c}}on, G.~Lugosi, and N.~Vayatis.
\newblock Ranking and empirical minimization of {U}-statistics.
\newblock {\em The Annals of Statistics}, 36(2):844--874, 2008.

\bibitem{Wet1987DegenerateUA}
T.~de~Wet.
\newblock Degenerate {U}- and {V}-statistics.
\newblock {\em South African Statistical Journal}, 21:99--129, 1987.

\bibitem{diakonikolas2016collision}
I.~Diakonikolas, T.~Gouleakis, J.~Peebles, and E.~Price.
\newblock Collision-based testers are optimal for uniformity and closeness.
\newblock {\em arXiv preprint arXiv:1611.03579}, 2016.

\bibitem{dubois2019goodness}
A.~Dubois, T.~B. Berrett, and C.~Butucea.
\newblock Goodness-of-fit testing for {H}{\"o}lder continuous densities under local differential privacy.
\newblock In {\em Foundations of Modern Statistics}, pages 53--119. Springer, 2019.

\bibitem{duchi2023fast}
J.~Duchi, S.~Haque, and R.~Kuditipudi.
\newblock A fast algorithm for adaptive private mean estimation.
\newblock {\em arXiv preprint arXiv:2301.07078}, 2023.

\bibitem{dwork2006calibrating}
C.~Dwork, F.~McSherry, K.~Nissim, and A.~Smith.
\newblock Calibrating noise to sensitivity in private data analysis.
\newblock In {\em Theory of Cryptography: Third Theory of Cryptography Conference, TCC 2006, New York, NY, USA, March 4-7, 2006. Proceedings 3}, pages 265--284. Springer, 2006.

\bibitem{feuer77symmetry}
A.~Feuerverger and R.~A. Mureika.
\newblock The empirical characteristic function and its applications.
\newblock {\em The Annals of Statistics}, 5(1):88 -- 97, 1977.

\bibitem{frees1989inforder}
E.~W. Frees.
\newblock Infinite order {U}-statistics.
\newblock {\em Scandinavian Journal of Statistics}, 16(1):29--45, 1989.

\bibitem{gaboardi2016differentially}
M.~Gaboardi, H.~Lim, R.~Rogers, and S.~Vadhan.
\newblock Differentially private chi-squared hypothesis testing: {G}oodness of fit and independence testing.
\newblock In {\em International Conference on Machine Learning}, pages 2111--2120. PMLR, 2016.

\bibitem{Ustat2023neurips}
B.~Ghazi, P.~Kamath, R.~Kumar, P.~Manurangsi, and A.~Sealfon.
\newblock On computing pairwise statistics with local differential privacy.
\newblock {\em Advances in Neural Information Processing Systems}, 33:14475--14485, 2020.

\bibitem{Gilbert1961RandomPN}
E.~N. Gilbert.
\newblock Random plane networks.
\newblock {\em Journal of the Society for Industrial and Applied Mathematics}, 9:533--543, 1961.

\bibitem{Gregory1977LargeST}
G.~Gregory.
\newblock Large sample theory for {$U$}-statistics and tests of fit.
\newblock {\em Annals of Statistics}, 5:110--123, 1977.

\bibitem{halmos1946ustat}
P.~R. Halmos.
\newblock The theory of unbiased estimation.
\newblock {\em The Annals of Mathematical Statistics}, 17(1):34 -- 43, 1946.

\bibitem{hardt2010geometry}
M.~Hardt and K.~Talwar.
\newblock On the geometry of differential privacy.
\newblock In {\em Proceedings of the Forty-Second ACM Symposium on Theory of Computing}, pages 705--714, 2010.

\bibitem{Ho2006TwostageUF}
H.-C. Ho and G.~S. Shieh.
\newblock Two‐stage {U}‐statistics for hypothesis testing.
\newblock {\em Scandinavian Journal of Statistics}, 33, 2006.

\bibitem{Hoeffding1948Ustat}
W.~Hoeffding.
\newblock A class of statistics with asymptotically normal distribution.
\newblock {\em The Annals of Mathematical Statistics}, 19(3):293 -- 325, 1948.

\bibitem{hoeffding1961strong}
Wassily Hoeffding.
\newblock The strong law of large numbers for {U}-statistics.
\newblock Technical report, North Carolina State University. Dept. of Statistics, 1961.

\bibitem{Hoeffding1963ProbabilityIF}
Wassily Hoeffding.
\newblock Probability inequalities for sums of bounded random variables.
\newblock {\em Journal of the American Statistical Association}, 58(301):13--30, 1963.

\bibitem{janson1984asymptotic}
S.~Janson.
\newblock The asymptotic distributions of incomplete {U}-statistics.
\newblock {\em Zeitschrift f{\"u}r Wahrscheinlichkeitstheorie und Verwandte Gebiete}, 66(4):495--505, 1984.

\bibitem{KLSU19}
G.~Kamath, J.~Li, V.~Singhal, and J.~Ullman.
\newblock Privately learning high-dimensional distributions.
\newblock In {\em Conference on Learning Theory}, pages 1853--1902. PMLR, 2019.

\bibitem{KSSU19}
G.~Kamath, O.~Sheffet, V.~Singhal, and J.~Ullman.
\newblock Differentially private algorithms for learning mixtures of separated {G}aussians.
\newblock {\em Advances in Neural Information Processing Systems}, 32, 2019.

\bibitem{Kamath2020PrivateME}
G.~Kamath, V.~Singhal, and J.~Ullman.
\newblock Private mean estimation of heavy-tailed distributions.
\newblock {\em ArXiv}, abs/2002.09464, 2020.

\bibitem{KV18}
V.~Karwa and S.~Vadhan.
\newblock Finite sample differentially private confidence intervals.
\newblock {\em arXiv preprint arXiv:1711.03908}, 2017.

\bibitem{kasiviswanathan2011can}
S.~P. Kasiviswanathan, H.~K. Lee, K.~Nissim, S.~Raskhodnikova, and A.~Smith.
\newblock What can we learn privately?
\newblock {\em SIAM Journal on Computing}, 40(3):793--826, 2011.

\bibitem{lam2022minimax}
J.~Lam-Weil, B.~Laurent, and J.-M. Loubes.
\newblock Minimax optimal goodness-of-fit testing for densities and multinomials under a local differential privacy constraint.
\newblock {\em Bernoulli}, 28(1):579--600, 2022.

\bibitem{Lee1990UStatisticsTA}
A.~J. Lee.
\newblock {\em U-statistics: Theory and Practice}.
\newblock Routledge, 2019.

\bibitem{li2020degenerate}
C.~Li and X.~Fan.
\newblock On nonparametric conditional independence tests for continuous variables.
\newblock {\em WIREs Computational Statistics}, 12(3):e1489, 2020.

\bibitem{linton2014}
O.~Linton and P.~Gozalo.
\newblock Testing conditional independence restrictions.
\newblock {\em Econometric Reviews}, 33(5-6):523--552, 2014.

\bibitem{minsker2023ustatistics}
S.~Minsker.
\newblock U-statistics of growing order and sub-{G}aussian mean estimators with sharp constants, 2023.

\bibitem{nissim2007smooth}
K.~Nissim, S.~Raskhodnikova, and A.~Smith.
\newblock Smooth sensitivity and sampling in private data analysis.
\newblock In {\em Proceedings of the Thirty-Ninth Annual ACM Symposium on Theory of Computing}, pages 75--84, 2007.

\bibitem{peng2019asymptotic}
W.~Peng, T.~Coleman, and L.~Mentch.
\newblock Asymptotic distributions and rates of convergence for random forests via generalized {U}-statistics.
\newblock {\em arXiv preprint arXiv:1905.10651}, 2019.

\bibitem{pitcan2017note}
Yannik Pitcan.
\newblock A note on concentration inequalities for {U}-statistics.
\newblock {\em arXiv preprint arXiv:1712.06160}, 2017.

\bibitem{politis2012subsampling}
D.~N. Politis, J.~P. Romano, and M.~Wolf.
\newblock {\em Subsampling}.
\newblock Springer Science \& Business Media, 2012.

\bibitem{popoviciu1935approximation}
T.~Popoviciu.
\newblock Sur l’approximation des fonctions convexes d’ordre sup{\'e}rieur.
\newblock {\em Mathematica (Cluj)}, 10:49--54, 1935.

\bibitem{popoviciu1965certaines}
T.~Popoviciu.
\newblock Sur certaines in{\'e}galit{\'e}s qui caract{\'e}risent les fonctions convexes.
\newblock {\em Analele Stiintifice Univ.“Al. I. Cuza”, Iasi, Sectia Mat}, 11:155--164, 1965.

\bibitem{serfling1980}
R.~J. Serfling.
\newblock {\em Approximation Theorems of Mathematical Statistics}.
\newblock Wiley Series in Probability and Mathematical Statistics: Probability and Mathematical Statistics. Wiley, New York, NY [u.a.], [nachdr.] edition, 1980.

\bibitem{Shorack2009EmpiricalPW}
G.~R. Shorack and J.~A. Wellner.
\newblock {\em Empirical Processes with Applications to Statistics}.
\newblock SIAM, 2009.

\bibitem{song2019approximating}
Y.~Song, X.~Chen, and K.~Kato.
\newblock Approximating high-dimensional infinite-order {$U$}-statistics: Statistical and computational guarantees.
\newblock {\em Electronic Journal of Statistics}, 13(2), January 2019.

\bibitem{ullman2019er}
J.~Ullman and A.~Sealfon.
\newblock Efficiently estimating {E}rdos-{R}enyi graphs with node differential privacy.
\newblock {\em Advances in Neural Information Processing Systems}, 32, 2019.

\bibitem{Vaart_1998}
A.~W. van~der Vaart.
\newblock {\em Asymptotic Statistics}.
\newblock Cambridge Series in Statistical and Probabilistic Mathematics. Cambridge University Press, 1998.

\bibitem{wainwright2019high}
M.~J. Wainwright.
\newblock {\em High-Dimensional Statistics: A Non-Asymptotic Viewpoint}, volume~48.
\newblock Cambridge University Press, 2019.

\bibitem{weber1981incomplete}
N.~C. Weber.
\newblock Incomplete degenerate {U}-statistics.
\newblock {\em Scandinavian Journal of Statistics}, 8(2):120--123, 1981.

\end{thebibliography}
\appendix
\renewcommand{\thetheorem}{A.\arabic{theorem}}
\renewcommand{\theproposition}{A.\arabic{proposition}}
\renewcommand{\thelemma}{A.\arabic{lemma}}
\renewcommand{\theremark}{A.\arabic{remark}}
\renewcommand{\thesection}{A.\arabic{section}}
\renewcommand{\thefigure}{A.\arabic{figure}}
\renewcommand{\theequation}{A.\arabic{equation}}
\renewcommand{\thealgorithm}{A.\arabic{algorithm}}
\renewcommand{\thedefinition}{A.\arabic{definition}}


\setcounter{theorem}{0}
\setcounter{section}{0}
\setcounter{figure}{0}
\setcounter{lemma}{0}
\setcounter{proposition}{0}

\newcommand{\cd}{\stackrel{d}{\rightarrow}}
\newcommand{\cp}{\stackrel{P}{\rightarrow}}
\newcommand{\cas}{\stackrel{a.s.}{\rightarrow}}
\newcommand{\cqm}{\stackrel{q.m.}{\rightarrow}}
\newpage
\section*{Roadmap of Appendix}
\begin{enumerate}
\item Section~\ref{AppAuxiliary} has U-statistics and concentration of U-statistics.
\item Section~\ref{SecMain} contains the details on our extensions of the CoinPress algorithm~\cite{biswas2020coinpress}.
\item Section~\ref{sec:lower} has proofs of the lower bounds.
\item Section~\ref{sec:app-mainthm} has proofs of Theorems~\ref{thm:mainthm} and~\ref{thm:mainthm_subsample}, along with concentration bounds for local Hájek projections.
\item Section~\ref{SecAppProofs} has details about proofs of the theorems in Section~\ref{SecApplications}.
\end{enumerate}

\section{U-statistics}
\label{AppAuxiliary}

Let $h:\mathcal{X}^k \to \mathbb{R}$ be a symmetric function, and let $X_1, \dots, X_n \in \mathcal{X}$. The U-statistic on the $n$ variables $X_1, \dots, X_n$, associated with $h$, is defined as 
\ba{\label{aeq:ustat}
U_n = \frac{1}{\binom{n}{k}} \sum_{\{i_1,\dots,i_k\}\in \I_{n,k}}h(X_{i_1},\dots, X_{i_k}).
}

The mean of $U_n$ computed from i.i.d.\ observations $X_1, \dots, X_n \sim \mathcal{D}$, is simply $\theta \defeq \E[h(X_1, \dots, X_k)]$. Moreover, the variance of $U_n$ can be expressed succinctly in terms of conditional expectations \citep{Lee1990UStatisticsTA}. For $c=1,2,\dots, k$, define $h_c:\mathcal{X}^c \to \mathbb{R}$ as
\ba{\label{aeq:condmean}
h_c(X_1, \dots, X_c) \defeq \E[h(X_1,\dots,X_k)|X_1=x_1,\dots,X_c=x_c],
}
\paragraph{Hoeffding Decomposition.} A U-statistic of degree $k$ can be written as the sum of uncorrelated U-statistics of degrees $1,2,\dots, k$.
Define
\bas{
h^{(1)}(X_1) = h(X_1)-\theta,
}
and for all $2 \le c \le k$, define
\bas{
h^{(c)}(X_1, \dots, X_c) = \bb{h_c(X_1, \dots, X_c) - \theta} - \sum_{\phi \subsetneq S \subsetneq \I_{c,i}} h^{(i)}\bb{X_S}.
}
Then, $U_n$ can be written as
\ba{\label{aeq:Hdecomposition}
U_n = \theta + \sum_{c=1}^k \binom{k}{c} U_n^{(c)},
}
where $U_n^{(c)}$ is the U-statistic on $X_1, \dots, X_n$ based on the kernel $h^{(c)}$. Equation~\eqref{aeq:Hdecomposition} is called the H-decomposition of $U_n$ \citep{hoeffding1961strong}. Hoeffding~\cite{hoeffding1961strong} also showed that the $c$ functions $h^{(1)}, \dots, h^{(k)}$ are pairwise uncorrelated. That is, let $1 \le c < d \le k$ and let $S_c$ and $S_d$ be subsets of $[n]$ of sizes $c$ and $d$ respectively. Then
\bas{
\covar\bb{h^{(c)}(X_{S_c}), h^{(d)}(X_{S_d})} = 0.}
This allows us to write the variance of $U_n$ in terms of the variances of $h^{(c)}$. For all $c \in [k]$, define
\ba{\label{aeq:hdecomp_covar}
    \delta_c^2 = \Var(h^{(c)}).
}
Then
\ba{\label{aeq:hdecomp_var}
    \Var(U_n) = \sum_{c=1}^k \binom{k}{c}^2 \binom{n}{c}^{-1} \delta_c^2.
}
Moreover, the conditional covariances $\zeta_c$ and the variances $\delta_c^2$ satisfy the relations
\ba{\label{aeq:zeta_delta}
\zeta_c = \sum_{i=1}^c \binom{c}{i} \delta_i^2, \qquad \delta_c^2 = \sum_{i=1}^c (-1)^{c-i} \binom{c}{i} \zeta_i.
}

\begin{lemma}\label{lem:Ustatvar} 
Suppose $k\leq n/2$.
\begin{itemize}
\item[(i)] If $\zeta_1 > 0$, then
\ba{\label{eq:asympvar}
\Var(U_n)=\frac{k^2\zeta_1}{n}+O\bb{\zeta_k\frac{k^2}{n^2}}.
}
\item[(ii)] If $\zeta_1 = 0$ and $\zeta_2 > 0$, then
\ba{\label{eq:asympvar_deg}
\Var(U_n)=\frac{k^2(k-1)^2\zeta_2}{2n(n-1)}+O\bb{\zeta_k\frac{k^3}{n^3}}.
}
\end{itemize}
\end{lemma}

\begin{proof}
This result follows directly from a calculation appearing in the proof of Theorem~3.1 in~\cite{minsker2023ustatistics}. Note that
\begin{equation*}
\zeta_k = \sum_{j=1}^k \binom{k}{j} \delta^2_j \ge \binom{k}{j} \delta_j^2 
\end{equation*}
for all $j \in [k]$. Moreover, we have
\begin{equation*}
\Var(U_n) = \sum_{j=1}^k \binom{k}{j}^2 {\binom{n}{j}}^{-1} \delta^2_j.
\end{equation*}
For part (i), we write
\begin{align*}
\Var(U_n) & = \frac{k^2\zeta_1}{n} + \sum_{j=2}^k \binom{k}{j}^2 {\binom{n}{j}}^{-1} \delta^2_j \le \frac{k^2\zeta_1}{n} + \sum_{j=2}^k \binom{k}{j} {\binom{n}{j}}^{-1} \zeta_k \\
&\le \frac{k^2\zeta_1}{n} + \zeta_k \sum_{j=2}^k \left(\frac{k}{n}\right)^j \le \frac{k^2 \zeta_1}{n} + \frac{k^2 \zeta_k}{n^2}\bb{1-\frac{k}{n}}^{-1} \le \frac{k^2 \zeta_1}{n} + \frac{2k^2 \zeta_k}{n^2}.
\end{align*}
For part (ii), we write
\begin{align*}
\Var(U_n) & = \frac{k^2\zeta_1}{n} + \frac{k^2(k-1)^2 \zeta_2}{2n(n-1)} + \sum_{j=3}^k \frac{\binom{k}{j}^2}{\binom{n}{j}} \zeta_j \le \frac{k^2(k-1)^2 \zeta_2}{2n(n-1)} + \zeta_k \sum_{j=3}^k \frac{\binom{k}{j}}{\binom{n}{j}} \\
& \le \frac{k^2(k-1)^2 \zeta_2}{2n(n-1)} + \zeta_k \sum_{j=3}^k \left(\frac{k}{n}\right)^j = \frac{k^2(k-1)^2 \zeta_2}{2n(n-1)} + \frac{2k^3 \zeta_k}{n^3}.
\end{align*}
\end{proof}

\begin{lemma}\label{lem:covar_inequality}
For all $1 \le c \le d \le k$, we have
\ba{
\frac{\zeta_c}{c} \le \frac{\zeta_d}{d}.
}
In particular, we have $k\zeta_1 \le \zeta_k$.
\end{lemma}

\begin{proof}
Using equation~\eqref{aeq:zeta_delta}, we have
\bas{
\frac{\zeta_c}{c} = \sum_{i=1}^c \frac{1}{c} \binom{c}{i}\delta_i^2 = \sum_{i=1}^c \frac{1}{i}\binom{c-1}{i-1}\delta_i^2 \le \sum_{i=1}^c \frac{1}{i}\binom{d-1}{i-1}\delta_i^2 \le \sum_{i=1}^d \frac{1}{d}\binom{d}{i}\delta_i^2 = \frac{\zeta_d}{d}.
}
\end{proof}


\paragraph{Concentration of U-Statistics}


\begin{lemma}\label{lem:conc_ustat} \citep{Hoeffding1963ProbabilityIF, pitcan2017note}
\begin{itemize}
\item[(i)] If $\mathcal{H}$ is sub-Gaussian$(\proxy)$, then for all $t > 0$, we have
\ba{\label{eq:hoeffding_ustat}
\P\bb{\bmb{U_n-\theta}\geq t}\leq 2\exp\bb{-\frac{\lfloor \frac{n}{k} \rfloor t^2}{2\proxy}}.
}
\item[(ii)] If $\mathcal{H}$ is almost surely bounded in $(-C,C)$, then for all $t>0$, we have
\ba{\label{eq:bernstein_ustat}
\P\bb{\bmb{U_n-\theta} \ge t} \le \exp\bb{\frac{-{ \left \lfloor \frac{n}{k} \right \rfloor } t^2}{2\zeta_k + 2Ct/3}}.
}
\end{itemize}
\end{lemma}

\begin{proof}
Without loss of generality, let $\theta = 0$. For any permutation $\sigma$ of $[n]$, let
\[
V_\sigma \defeq  \frac{1}{m} \sum_{i=1}^{m} h(X_{\sigma\bb{k(i-1)+1}}, X_{\sigma\bb{k(i-1)+2}}, \dots, X_{\sigma\bb{ki}}),
\]
where $m = \lfloor n/k \rfloor$. By symmetry, $U_n = \frac{1}{n!} \sum_{\sigma} V_\sigma$. For any $s > 0$,
\bas{
\P\bb{U_n \ge t} &= \P\bb{e^{sU_n} \ge e^{st}} \le e^{-st}\E\bbb{e^{sU_n}} = e^{-st}\E\bbb{\exp\bb{\frac{s}{n!}\sum_{\sigma}V_n}} \\
&\le e^{-st}\E\bbb{\frac{1}{n!} \sum_{\sigma} \exp(sV_n)} = e^{-st}\E\bbb{\exp(sV_{\text{id}})} \\
&= e^{-st}\E\bbb{\exp\bb{\frac{s}{m} h(X_1, \dots, X_k)}}^{m}.
}

If $\mathcal{H}$ is sub-Gaussian with variance proxy $\tau$, we can further bound the inequality above as
\bas{
\P\bb{U_n \ge t} \le e^{-st}\bb{\exp\bb{\frac{s^2\proxy}{2m^2}}}^m = \exp\bb{-st + \frac{s^2 \proxy}{2m}}.
}
Set $s = \frac{tm}{\proxy^2}$ to get the desired result. Note how the argument is similar to the classical Hoeffding's inequality argument after applying Jensen's inequality on $V_\sigma$. The second result follows similarly by adapting the techniques of the proof of Bernstein's inequality \citep{bernstein1924modification} to the argument in Hoeffding~\cite{Hoeffding1963ProbabilityIF}; for a detailed proof, see Pitcan~\cite{pitcan2017note}.
\end{proof}

\section{Details for Section~\ref{sec:offtheshelf}}
 \label{SecMain}
 \subsection{General result}

We begin with a general theorem, from which the results of Proposition~\ref{cor:naive}, \ref{lem:alltuples}, and~\ref{lem:ss} are derived.

In Algorithms~\ref{alg:PrivateMean} and~\ref{alg:OneStep}, we present a natural extension of the CoinPress algorithm~\citep{biswas2020coinpress}, which is then used to obtain a private estimate of $\theta$ with the non-private term matching $\Var(U_n)$. Originally, this algorithm was used for private mean and covariance estimation of i.i.d.\ (sub)Gaussian data. We extend the algorithm to take as input data $\{Y_j\}_{j \in [m]}$ such that (i) each $Y_j$ is equal to $h(X_S)$ for some $S$, (ii) the $Y_j$'s are \textit{weakly dependent} on each other, and (iii) each $Y_j$, as well as their mean $\sum_{j \in [m]} Y_j/m$, has sufficiently strong concentration around the population mean. 

For instance, suppose $m = \lfloor n/k \rfloor$ and $Y_j = h(X_{S_j})$ for all $j \in [m]$, where $S_j = \{(j-1)k+1, \dots, (j-1)k+k \}$. Then Algorithm~\ref{alg:PrivateMean} reduces to the CoinPress algorithm applied to $n/k$ independent observations $h(X_{S_1}), h(X_{S_2}), \dots, h(X_{S_m})$. 

We present more general versions of Algorithms~\ref{alg:PrivateMeanArxiv} and~\ref{alg:OneStepArxiv} to include general tail bounds $Q(.)$, $Q_{\text{avg}}(.)$, a general failure probability $\alpha$, which is set to $0.01$ in the main text. This is because we use the median of means wrapper in~\ref{alg:MoM} to boost the probability of success.
\begin{algorithm}
\caption{\label{alg:PrivateMean} \textbf{U-StatMean}$\bb{n, k, h, \left \{X_i \right\}_{i \in [n]}, \F= \{S_1, \dots, S_m\}, R, \eps, \gamma, \Yc{\cdot}, \Uc{\cdot}}$}
\begin{algorithmic}[1]
\STATE $t \gets \log\bb{R/\Yc{\gamma}}, \; [l_0, r_0] \gets [-R,R]$
\FOR{$j = 1, \dots, m$}
\STATE $Y_{0,j} \gets h(X_{S_j})$
\ENDFOR
\FOR{$i = 1,2,...,t$}
\STATE $\{Y_{i,j}\}_{j \in [m]}, [l_i, r_i] \gets \text{U-StatOneStep}\bb{n, k, \{Y_{i-1, j}\}, \F, [l_{i-1}, r_{i-1}], \frac{\eps}{2t}, \frac{\gamma}{t}, \Yc{\cdot}, \Uc{\cdot}}$ 
\ENDFOR
\STATE $\{Y_{t+1,j}\}_{j \in [m]}, [l_{t+1}, r_{t+1}] \gets \text{U-StatOneStep}\bb{n, k, \{Y_{t,j}\}, \F, [l_t, r_t], \eps/2, \gamma, \Yc{\cdot}, \Uc{\cdot}}$
\RETURN $(l_{t+1}+r_{t+1})/{2}$
\end{algorithmic}
\end{algorithm}

\begin{algorithm}
\caption{\label{alg:OneStep} \textbf{U-StatOneStep}$\bb{n, k, \left\{Y_i \right\}_{i \in [m]}, \F, [l,r], \eps', \beta, \Yc{\cdot}, \Uc{\cdot}}$}
\begin{algorithmic}[1]
\STATE $Y_j \gets \proj_{l-\Yc{\beta}, r+\Yc{\beta}}\bb{Y_j}$ for all $1 \le j \le m$.
\STATE $\Delta \gets \dep{n}{k}{\F}\bb{r-l+2\Yc{\beta}}$
\STATE $Z \gets \frac{1}{m} \sum_{j=1}^{m} Y_j + W$, where $W \sim \text{Lap}\bb{\frac{\Delta}{\eps'}}$ \label{step:Laplace}
\STATE $[l,r] \gets \bbb{Z - \bb{\Uc{\beta}+\frac{\Delta}{\eps'} \log \frac{1}{\beta}}, Z + \bb{\Uc{\beta}+\frac{\Delta}{\eps'} \log \frac{1}{\beta}}}$
\RETURN $\{Y_j\}_{j \in [m]}, [l,r]$ 
\end{algorithmic}
\end{algorithm}

\begin{setting}
Let $n$ and $k$ be positive integers with $k \le n/2$, and let $h: \mathcal{X}^k \to \mathbb{R}$ be a symmetric function and let $\mathcal{D}$ be an unknown distribution over $\mathcal{X}$ with $\E\bbb{h(\mathcal{D}^k)} = \theta$ such that $\bmb{\theta} < R$ for some known parameter $R$.
\end{setting}

Let $m$ be an integer and $\F = \bcb{S_1, S_2, \dots, S_m}$ be a family of not necessarily distinct elements of $\I_{n,k}$. Define
\ba{\label{eq:index_frac}
f_i \defeq \frac{\bmb{\bcb{j \in [m]: i \in S_j}}}{m},
}
the fraction of indices $j$ such that $S_j$ contains $i$, and define the maximal dependence fraction
\ba{\label{eq:dep_n_k}
\dep{n}{k}{\mathcal{S}} \defeq \max_{i \in [n]} f_i.
}
For each $j \in [m]$, let $Y_j$ denote $h(X_{S_j})$. Clearly, $\E\bbb{Y_j} = \theta$.
To allow for small noise addition while ensuring privacy, it is desirable to choose $\mathcal{S}$ with small $\dep{n}{k}{\mathcal{S}}$.

Define functions $\Yc{\beta} = Q_{n,k,h,\mathcal{D}, \mathcal{S}}(\beta)$ and $\Uc{\beta} = Q^{\text{avg}}_{n,k,h,\mathcal{D}, \mathcal{S}}(\beta)$ on $\beta \in (0,1]$ such that
\ba{\label{eq:def_yc_uc}\mathbb{P}\bb{\sup_{j \in [m]} \bmb{Y_j-\theta} > \Yc{\beta}} < \beta, ~~~
\mathbb{P}\bb{\bmb{\frac{1}{m}\sum_{j=1}^m Y_j - \theta} > \Uc{\beta}} < \beta.}
We will refer to $\Yc{\beta}$ and $\Uc{\beta}$ as $\beta$-\textit{confidence bounds} for $\sup_{j \in [m]} \bmb{Y_j - \theta}$ and $\bmb{\frac{1}{m}\sum_{j \in [m]} Y_j - \theta}$, respectively.

We apply Theorem~\ref{prop:mother} (specifically, the form obtained in Lemma~\ref{lem:mother_MoM}) to different $\F$ to obtain private estimates of $\theta$, with statistical and computational tradeoffs depending on the family $\F$. As Remark~\ref{rem:sigma_unknown} suggests, we will also need to privately estimate concentration bounds on the $Y_j$'s and their average. Naturally, this requires a private estimate of the variance $\zeta_k$. We provide guarantees from Biswas et al.~\cite{biswas2020coinpress} for private variance estimation and mean estimation here, where we have translated the mean estimation guarantee to fit our setting.
\begin{theorem}\label{prop:mother}
For $\eps > 0$, Algorithm \ref{alg:PrivateMean} with input $\bb{n,k,h, \{X_i\}_{i \in [n]}, \F, R, \eps, \gamma, \Yc{\cdot}, \Uc{\cdot}}$ returns $\tilde{\theta}_{n}$ such that 
\ba{\label{eq:mother_error}
|\tilde{\theta}_{n}-\theta| \le O\bb{\underbrace{
\frac{\sqrt{\Var(\sum_{j \in [m]} Y_j)}}{m \sqrt{\gamma}}}_{\text{non-private error}} + \underbrace{\frac{\dep{n}{k}{\F}\Yc{\gamma}}{\eps} \cdot \log\left(\frac{1}{\gamma}\right)}_{\text{private error}}},
}
with probability at least $1-6\gamma$,\footnote{The following subsection modifies the algorithm so that the error depends polylogarithmically in $\alpha$.} as long as  
\ba{\label{eq:mother_cond}
\dep{n}{k}{\F} \le \frac{\Yc{\gamma} \eps}{10t\Yc{\gamma/t} \log t/\gamma} \qquad\text{ and  }\qquad \Uc{\gamma/t} < \Yc{\gamma},
}
where $t = \lceil C\log \bb{R/\Yc{\gamma}} \rceil$. 
Moreover, Algorithm \ref{alg:PrivateMean} is $\eps$-differentially private and runs in time ${O}(n + m\log \bb{R/\Yc{\gamma}} + k |\F|)$.
\end{theorem}

\begin{remark}\label{rem:sigma_unknown} Theorem~\ref{prop:mother} assumes that $\Yc{\cdot}$ and $\Uc{\cdot}$ are known, despite the mean $\theta$ being unknown. Note that we only need to know the value of these functions at $\gamma$ and $\gamma/t$, for a given $\gamma$. If these bounds are not known, we may first need to (privately) compute $\Yc{x}$ and $\Uc{x}$ 
and then use those privately computed bounds in the algorithm. For example, if the $Y_i's$ are sub-Gaussian with variance proxy 1, then $\Yc{x}= \sqrt{\log \frac{m}{x}}$. We will see how to estimate these parameters for various families $\F$ of indices used in Algorithm~\ref{alg:PrivateMean}.
\end{remark}

\begin{proof}
[Proof of Theorem~\ref{prop:mother}]
We will prove privacy and accuracy guarantees separately.

\textbf{Privacy.} Algorithm~\ref{alg:PrivateMean} makes $t+1$ calls to Algorithm~\ref{alg:OneStep}; let $\Delta_i, W_i$, and $Z_i$ be the values taken by $\Delta, W,$ and $Z$ in the $i^{\text{th}}$ call to Algorithm~\ref{alg:OneStep}, for $1 \le i \le t+1$. Let $\beta \defeq \frac{\gamma}{t}$.
It can be shown inductively that the interval lengths $r_i-l_i$ and the values $\Delta_i$ do not depend on the dataset.
For any $1 \le i \le t$, note that $Y_{i,j} = \proj_{l_{i-1}-\Yc{\beta}, r_{i-1}, \Yc{\beta}}\bb{Y_{i-1,j}}$ for all $1 \le j \le m$. Suppose we change $X_{w}$ to $X'_{w}$ for some index $w$. For any $1 \le i \le t+1$, conditioned on the values of $Z_{i'}$ for $1 \le i' < i$, at most a $\dep{n}{k}{\F}$ fraction of $\{Y_{i,j}\}_{j \in [m]}$ \textit{depend} on $w$ (this is true by the definition of $\dep{n}{k}{\F}$). Since $Y_{i,j} = \proj_{l_{i-1}-\Yc{\beta}, r_{i-1}+\Yc{\beta/m}}\bb{Y_{i-1,j}}$ has range $r_{i-1}-l_{i-1}+2\Yc{\beta}$, the sensitivity of $\frac{1}{m}\sum_{j=1}^m Y_{i,j}$ is at most $\dep{n}{k}{\F}(r_{i-1}-l_{i-1}+2\Yc{\beta}) = \Delta_i$. Therefore, by standard results (cf.\ Lemma~\ref{lem:laplace_mechanism}), for all $1 \le i \le t$, the output $Z_i$ (and therefore the interval $[l_i, r_i]$), conditioned on $Z_{i'}$ for $1 \le i' < i$, is $\frac{\eps}{2t}$-differentially private. Similarly, the output $(l_{t+1}+r_{t+1})/2 = Z_{t+1}$, conditioned on $\{Z_i\}_{i \in [t]}$, is $\frac{\eps}{2}$-differentially private. By Basic Composition (see Lemma~\ref{lem:basic_comp}), Algorithm~\ref{alg:PrivateMean} is $\eps$-differentially private.

\textbf{Utility.} First, we show that if Algorithm \ref{alg:OneStep} is invoked with $\theta \in [l,r]$, it returns an interval $[l', r']$ such that $\theta \in [l', r']$ with probability at least $1-3\beta$. Consider running a variant of Algorithm \ref{alg:PrivateMean} with the projection step omitted in every call of Algorithm \ref{alg:OneStep}. With probability at least $1-\beta$, we have $\bmb{\frac{1}{m}\sum_{i=1}^m Y_i - \theta} \le \Uc{\beta}, $ and with probability at least $1-\beta$, we have $
\left |W \right | \le \frac{\Delta}{\eps'} \log \frac{1}{\beta}
$. Therefore, with probability at least $1-2\beta$, we have
\bas{|Z-\theta| \le \Uc{\beta} + \frac{\Delta}{\eps'}\log \frac{1}{\beta},}
in which case $\theta \in [l', r']$.

Finally, reintroducing the projection step only increases the error probability by at most $\beta$.
Taking a union bound over $t$ steps, we see that $\theta \in [l',r']$, with probability at least $1-3\gamma$.

Next, we claim that if $r-l > 28\Yc{\gamma}$, then $r'-l' \leq (r-l)/2$. Using the assumption, we have
\bas{
\dep{n}{k}{\F} \le \frac{\Yc{\gamma} \eps}{10t\Yc{\gamma/t} \log t/\gamma} \le \min \bb{\frac{\eps'}{5 \log 1/\beta}, \frac{\Yc{\gamma} \eps'}{5 \Yc{\beta} \log 1/\beta}},
}
where the second inequality follows from taking $\epsilon'=\frac{\epsilon}{2t}$ and using the fact that $Q(\gamma) \le Q\left(\frac{\gamma}{t}\right)$, since the quantile function is nonincreasing. Furthermore, by the assumption $\Uc{\beta} < \Yc{\gamma}$, 
we have
\bas{
r'-l' 
&= \frac{2\dep{n}{k}{\F} \log 1/\beta}{\eps'}\bb{r-l} + \bb{2\Uc{\beta} + \frac{4\dep{n}{k}{\F} \Yc{\beta}\log 1/\beta }{\eps'} }  \\
&\le \frac{2(r-l)}{5} + \bb{2\Yc{\gamma} + \frac{4}{5}\Yc{\gamma}} \le \frac{r-l}{2}.
}
Thus, after $t = \Omega\left(\log\left(\frac{R}{Q(\gamma)}\right)\right)$ iterations, we are guaranteed that the length of the final interval $[l_t, r_t]$ is at most $28Q(\gamma)$.

Finally, consider lines $8$ and $9$ of Algorithm \ref{alg:PrivateMean}. The algorithm returns the midpoint of the interval $[l_{t+1}, r_{t+1}]$, which is $Z_{t+1}$ in the final call of Algorithm \ref{alg:OneStep}.
By Chebyshev's inequality, we have
\ba{\label{eq:Y0_conc}
\left |\frac{1}{m} \sum_{j=1}^m Y_{0,j} - \theta \right | \le \sqrt{\frac{1}{\gamma} \cdot \Var\bb{\frac{1}{m}\sum_{i=1}^m Y_{0,j}}},
}
with probability at least $1-\gamma$, and with probability at least $1-\gamma$, none of the $Y_i$'s are truncated in the projection step in the final call of Algorithm~\ref{alg:OneStep}. Finally, with probability at least $1-\gamma$, we have
\bas{
W_{t+1} =  O\bb{\frac{\Delta_{t+1}}{\eps}}  =  O\bb{\frac{\dep{n}{k}{\F}\Yc{\gamma}}{\eps} \log \frac{1}{\gamma}}.
}
The conclusion follows from a union bound over all events.
\end{proof}


\subsection{Boosting the error probability}
Algorithm~\ref{alg:PrivateMeanArxiv} incurs a $1/\sqrt{\gamma}$ multiplicative factor in the non-private error, stemming from an application of Chebyshev's inequality to bound $| \sum_{j=1}^{m} h(X_{S_j})/m - \theta|$. For specific families $\F$, we may be able to provide stronger concentration bounds for $\sum_{j=1}^m h(X_{S_j})/m$ in inequality~\eqref{eq:Y0_conc}. Instead, we complement the result of Theorem~\ref{prop:mother} by applying the following median-of-means procedure that allows for an improved dependence on the failure probability $\alpha$ with only a $\log \bb{1/\alpha}$ multiplicative blowup in the sample complexity:

\begin{lemma}\label{lem:mother_MoM}
Let $\alpha \in (0,1)$ and $\eps \ge 0$. Let $\mathcal{A}$ be an $\eps$-differentially private algorithm. Consider a size $n$ dataset $D_n\stackrel{i.i.d}{\sim} \mathcal{D}$, for a distribution $\mathcal{D}$ with some unknown parameter $\theta$ such that with probability at least 0.75, we have
\bas{
\bmb{\mathcal{A}(D_n) - \theta} \le r_{n}.
}
Split $D_n$ into $q \defeq \qa$ equal independent chunks,\footnote{Assume for simplicity that $n$ is divisible by $q$ and $q$ is an odd integer.} and run $\mathcal{A}$ on each chunk to obtain $\eps$-differentially private estimates $\{\tilde{\theta}_{n,i}\}_{i \in [d]}$ of $\theta$. Let $\tilde{\theta}^{\text{med}}_{n}$ be the median of these $q$ estimates. Then $\tilde{\theta}^{\text{med}}_{n}$ is $\epsilon$-differentially private, and with probability at least $1-\alpha$, we have
\ba{\label{eq:mother_MoM_error}
\bmb{\tilde{\theta}^{\text{med}}_{n} - \theta} \le r_{n/q}.
}
\end{lemma}

\begin{proof}
The privacy of $\tilde{\theta}^{\text{med}}_{n}$ follows from parallel composition (Lemma~\ref{lem:parallel_comp}). For utility, we know from the hypothesis that for each $i \in [q]$, with probability at least $3/4$, the estimate $\tilde{\theta}_{n,i}$ satisfies
\bas{
|\tilde{\theta}_{n,i}-\theta| \le r_{n/q}.
}
If more than half the estimates $\tilde{\theta}_{n,i}$ satisfy the above equation, then so does the median. Let $T_i$ be the random variable that assumes the value $0$ if $\tilde{\theta}_{n,i}$ satisfies the above equation and assumes the value $1$ otherwise. Then, $\E[T_i] \le 1/4$, and it suffices to show that 
\bas{
\Pr\bb{T_1+T_2+\dots+T_q \le q/2} \ge 1-\alpha.
}
This follows from a standard Hoeffding inequality:
\bas{
\Pr\bb{T_1+T_2+\dots+T_q > q/2} &\le \Pr\bb{\sum_i (T_i -E[T_i]) > q/4} \\
&\le e^{-2(1/4)^2 q} \le \alpha,
}
as long as $q \ge \qa$.
\end{proof}

\subsection{Proof of Proposition~\ref{cor:naive}}
\label{AppCorNaive}

First, suppose the variance $\zeta_k$ is known.
It is easy to see that $\dep{n}{k}{\F_{\text{naive}}} = \frac{k}{n}$.
By the assumption that $h(X_S)$ is $\tau$-sub-Gaussian, we have
\bas{P(|h(X_S)-\theta|\geq y)\leq 2\exp\bb{-\frac{y^2}{2\proxy}}.
}
Hence, with probability at least $1-\gamma/m$, we have
\ba{\label{eq:ally}
|Y_j-\theta|\leq \sqrt{2\proxy \log (2m/\gamma)},
}
for each $1 \le j \le m$, where we use the notation as in the setting of Theorem~\ref{prop:mother}.
By a union bound, we can take the quantile function $\Yc{\gamma}= \sqrt{2\proxy \log \left(\frac{2n}{k\gamma}\right)}$. Moreover, since the $Y_j$'s are independent, the average $\frac{1}{m} \sum_{j \in [m]}  Y_j$ is $\frac{\tau}{m}$-sub-Gaussian with variance $\frac{\zeta_k}{m}$. Therefore, we have
\bas{
P\bb{\bmb{\frac{1}{m}\sum_{j=1}^m Y_j -\theta}\geq y}\leq 2\exp\bb{-\frac{my^2}{2\proxy}}.
}
This yields a bound of $\Uc{\gamma} = \sqrt{\frac{2k\proxy\log (2/\gamma)}{n}}$. It remains to verify the conditions of Theorem~\ref{prop:mother}. We have
\bas{
\frac{k}{n} \le \frac{\Yc{\gamma}\epsilon}{10t\Yc{\gamma/t}\log(t/\gamma)} &\iff \frac{k}{n} \le \frac{\epsilon}{10t \log (t/\gamma)} \sqrt{\frac{\log (2n/k\gamma)}{\log (2nt/k\gamma)}},
}
and
\bas{
\Uc{\gamma/t} < \Yc{\gamma} &\iff \sqrt{\frac{2k\proxy\log (2t/\gamma)}{n}} \le \sqrt{2\proxy\log (2n/k\gamma)} \\
&\iff n \ge k \frac{\log (2t/\gamma)}{\log (2n/k\gamma)},
}
which are both true by the sample size assumption, noting that $t = \lceil C \log(R/Q(\gamma))\rceil$. Therefore, with probability at least $1-6\gamma$, we have
\bas{
\bmb{\tn-\theta} \le O\bb{\frac{1}{\sqrt{\gamma}}\Var(\tn) + \frac{k}{n\epsilon}\sqrt{2\proxy\log (2n/k\gamma)} \log \frac{1}{\gamma}}.
}
Choosing $\gamma$ to be an appropriate constant, we arrive at the deisred result.


\subsection{Proof of Proposition~\ref{lem:alltuples}}
\label{AppLemAlltuples}


For any $i \in [n],$ there are exactly $\binom{n-1}{k-1}$ sets $S \in \I_{n,k}$ such that $i \in S$. Following the notation from the setting of Theorem~\ref{prop:mother}, we have $f_i = \binom{n-1}{k-1}/\binom{n}{k} = \frac{k}{n}$ for all $i \in [n]$, so $\dep{n}{k}{\F_{\text{all}}} = \frac{k}{n}$. Moreover, for each $S \in \I_{n,k}$, we have
$\mathbb{P}\bb{\bmb{h(X_S) - \theta} \ge y} \le 2\exp\bb{\frac{-y^2}{2\proxy}}.$ Letting
\bas{
\Yc{\gamma} \defeq \sqrt{2\proxy k \log \bb{\frac{2n}{\gamma}}} > \sqrt{2 \proxy \log \bb{\frac{2}{\gamma} \binom{n}{k}}},
}
we see that each $Y_i$ is within $\Yc{\gamma}$ of $\theta$ with probability at least $\frac{\gamma}{\binom{n}{k}}$. A union bound implies that this choice of $\Yc{\gamma}$ is valid. For the concentration of the average $\frac{1}{m}\sum_{j \in [m]} Y_j$, which is simply $U_n$, we can use Lemma~\ref{lem:conc_ustat})
to see that
\bas{
\Uc{\gamma} \defeq \sqrt{\frac{2 \proxy k \log \frac{2}{\gamma}}{n}}
}
is a valid choice. We now verify the conditions in Theorem~\ref{prop:mother}:

\bas{
\dep{n}{k}{\mathcal{S}} = \frac{k}{n} \le \frac{\Yc{\gamma}\epsilon}{10t\Yc{\gamma/t}\log(t/\gamma)} = \frac{\epsilon}{10t \log(t/\gamma)} \sqrt{\frac{\log 2n/\gamma}{\log 2nt/\gamma}}
}
if and only if
\bas{n \ge \frac{10kt \log (t/\gamma)}{\epsilon} \sqrt{\frac{\log 2nt/\gamma }{\log 2n/\gamma}}.
}
Recalling that $t = \lceil C\log \bb{R/\Yc{\gamma}} \rceil$, we see that this holds under the sample complexity assumption on $n$. Furthermore, we have $\Uc{\gamma/t} \le \Yc{\gamma}$ if and only if
\bas{
\sqrt{\frac{2\proxy k \log \frac{2t}{\gamma}}{n}} \le \sqrt{2\proxy k \log \bb{\frac{n}{\gamma}}} \iff n \ge \frac{\log 2t/\gamma}{\log n/\gamma},
}
which is also true by assumption.
Therefore, with probability at least $1-6\gamma$, we have
\bas{|\tall-\theta|&\leq O\bb{\frac{1}{\sqrt{\gamma}}\sqrt{\Var(U_n)}+\frac{k}{n\epsilon}\sqrt{2\proxy k \log \bb{\frac{2n}{\gamma}}}}.
}
Algorithm~\ref{alg:OneStepArxiv} uses a constant failure probability of $\gamma=0.01$, which assures a success probability of at least $0.75$. This is further boosted by Wrapper~\ref{alg:MoM}. Now, an application of Lemma~\ref{lem:mother_MoM} gives the stated result.

\subsection{Proof of Proposition~\ref{lem:ss}}
\label{AppLemSS}

First, we need the following helper lemmas:
\begin{lemma}\label{lem:ssvar}
Define $\hat{\theta}_{\text{ss}}= \frac{1}{M} \sum_{S\in \F_{\text{ss}}} h(X_S)$. We have $\mathbb{\Var} \left [\tss\right] = \left ( 1 - \frac{1}{M} \right ) \Var(U_n) + \frac{1}{M} \zeta_k$.
\end{lemma}

\begin{proof} Clearly, $\E[\tss] = \theta$. We compute both terms of the following decomposition of the variance of $\tss$ separately; recall that $\X = \{X_i\}_{i \in [n]}$:
\bas{
\Var\bb{\tss} = \Var\bb{\E\bbb{\tss |\X}} + \E\bbb{\Var\bb{\tss | \X}}.
}
Now,
\bas{
\Var\bb{\E\bbb{\tss | \X}} &= \Var\bb{\E\bbb{\left. \frac{1}{M}\sum_{j \in [M]} h\bb{X_{S_i}} \right | \X}} = \Var(U_n),
}
and
\bas{
\E\bbb{\Var\bb{\left. \tss \right| \X}} &= \mathbb{E} \bbb{\Var \bb{\left. \frac{1}{M} \sum_{j=1}^{M} h(X_{S_j}) \right | \X}} = \frac{1}{M} \mathbb{E}\bbb{\Var\bb{\left. h(X_{S}) \right|\X} } \\
&= \frac{1}{M}\E\bbb{\frac{1}{\binom{n}{k}} \sum_{S \in \I_{n,k}} h(X_S)^2 - \bb{\frac{1}{\binom{n}{k}} \sum_{S \in \I_{n,k}} h(X_S)}^{2} ~} \\
&= \frac{1}{M}\bb{\bb{\zeta_k + \theta^2} - \bb{\Var(U_n)+\theta^2}} = \frac{\zeta_k - \Var(U_n)}{M}.
}
Adding the two equalities yields the result.
\end{proof}

\begin{lemma}\label{lem:ssdep}
Let $\gamma > 0$, and let $M = \Omega(\frac{n}{k} \log \frac{n}{\gamma})$. Then $\dep{n}{k}{\F_{\text{ss}}}\le \frac{4k}{n}$ with probability at least $1-\gamma$.
\end{lemma}
\begin{proof}
Let $Z_i$ be the number of sampled subsets of which $i$ is an element. Observe that $Z_i \sim \text{Binom}(M,k/n)$, with mean $\mu := Mk/n$. By a Chernoff bound, for any $\delta > 0$ and any $i \in [n]$, we have
\ba{
\P\bb{Z_i \ge (1+\delta)\mu} \le \bb{\frac{e^\delta}{(1+\delta)^{1+\delta}}}^\mu.
}
By a union bound, we have
\bas{
\P\bb{\dep{n}{k}{\F_{\text{ss}}} > \frac{4k}{n}} &= 
\P\bb{\max_i Z_i > 4\mu} \le n\bb{\frac{e^3}{(1+3)^{1+3}}}^{\mu} \le n\exp\bb{-\frac{Mk}{n}},
}
which is at most $\gamma$ by our choice of $M$. 
\end{proof}
Let $\G_\gamma$ denote the event that $\dep{n}{k}{\F_{\text{ss}}} \le \frac{4k}{n}$, which occurs with high probability by Lemma~\ref{lem:ssdep}.


Note also that conditioned on any family $\F_{\text{ss}}$ of subsets of $\I_{n,k}$, the run of Algorithm~\ref{alg:PrivateMean} is $\epsilon$-differentially private. Since the randomness of $\F_{\text{ss}}$ is independent of the data, the algorithm (along with the private variance estimation) is still $2\epsilon$-differentially private.

Let 
\bas{
\Yc{\gamma}=\sqrt{2\proxy k \log \bb{\frac{4n}{\gamma}}} ~~\text{ and }~~\Uc{\gamma} = 4\sqrt{\frac{\proxy k}{\min\bb{M,n}}} \log \frac{4n}{\gamma}.
}
We show that these are indeed the corresponding confidence bounds for $Y_S,S\in\F_{ss}$ and $\hat{\theta}_{ss}$.

By a sub-Gaussian tail bound, for any $S \in \I_{n,k}$, the probability that $\bmb{h(X_S)-\theta} > \Yc{\gamma}$ is at most $2\bb{\frac{\gamma}{4n}}^k \le \frac{\gamma}{2n^k}$. By a union bound over all $\binom{n}{k}$ sets $S$, we then have $\bmb{h(X_S)-\theta} \le \Yc{\gamma}$ for all $S \in \mathcal{I}_{n,k}$, with probability at least $1-\frac{\gamma}{2}$. Call this event $\mathcal{E}_\gamma$.

Next, $\E[\tss|X_1,\dots X_n]=U_n$. Moreover, for any $c > 0$,
\bas{
\P\bb{|\tss-\theta|\geq c}&\leq \P\bb{|\tss-U_n|\geq c/2}+P\bb{|U_n-\theta|\geq c/2} \\
&\leq \E_{X_1,\dots, X_n}\P\bb{|\tss-U_n|\geq c/2|X_1,\dots,X_n}+2\exp\bb{-\frac{nc^2}{8k\proxy}},
}
where we used Lemma~\ref{lem:conc_ustat} to bound the second term.
For the first term in inequality~\eqref{eq:ssconc}, note that conditioned on the data $X_1, \dots, X_n$, the $h(X_{S_j})$'s are independent draws from a uniform distribution over the $\binom{n}{k}$ values $\{h(X_S)\}_{S \in \I_{n,k}}$, with mean $U_n$, and the $\bmb{h(X_{S_j})-\theta}$'s are bounded by $\max_{S\in\I_{n,k}} |h(X_S)-\theta| \le \Yc{\gamma}$. Therefore, each $h(X_{S_j})-U_n$ is \text{sub-Gaussian}$(\Yc{\gamma}^2)$, implying that
\ba{
\label{eq:ssconc_term1}
\E\left[\P\bb{|\tss-U_n|\geq c/2|X_1,\dots,X_n}\right] &\leq 2 \E\bbb{\exp\bb{-\frac{Mc^2}{8\Yc{\gamma}^2}} | X_1, \dots, X_n,\mathcal{E}_\gamma} + P(\mathcal{E}_\gamma^c) \notag \\
& \le 2\exp\bb{-\frac{Mc^2}{16\proxy k\log(4n/\gamma)}} + P(\mathcal{E}_\gamma^c) \notag \\
&\leq 2 \exp\bb{-\frac{Mc^2}{16\proxy k\log(4n/\gamma)}}+\frac{\gamma}{2}.
}
Combining inequalities~\eqref{eq:ssconc} and~\eqref{eq:ssconc_term1}, we have
\ba{
\label{eq:ssconc}
\P\bb{|\tss-\theta|\geq c} &\le 2 \exp\bb{-\frac{Mc^2}{16 \proxy k\log(4n/\gamma)}}+\frac{\gamma}{2} + 2\exp\bb{-\frac{nc^2}{8k\proxy}} \le \gamma,
}
as long as
\bas{
c \ge 4\sqrt{\frac{\proxy k}{\min\bb{M,n}} \log \left(\frac{2n}{\gamma}\right) \log \left(\frac{8}{\gamma}
\right)}.
}
This justifies the choice of $\Uc{\gamma}$.
We now verify the conditions in Theorem~\ref{prop:mother}. Conditioned on $\mathcal{G}_\gamma$, we have
\bas{
dep_{n,k} (\mathcal{S}) = \frac{4k}{n} \le \frac{\Yc{\gamma}\epsilon}{10t\Yc{\gamma/t}\log(t/\gamma)} = \frac{\epsilon}{10t \log(t/\gamma)} \sqrt{\frac{\log 4n/\gamma}{\log 4nt/\gamma}}
}
if and only if
\bas{n \ge \frac{40kt \log (t/\gamma)}{\epsilon} \sqrt{\frac{\log 4nt/\gamma }{\log 4n/\gamma}}.
}
Recalling that $t = \lceil C\log \bb{R/\Yc{\gamma}} \rceil$, we see that the above holds under the sample complexity assumption on $n$. Furthermore, $\Uc{\gamma/t} \le \Yc{\gamma}$ iff
\bas{
4\sqrt{\frac{\proxy k}{\min\bb{M,n}}} \log \frac{4nt}{\gamma} \le \sqrt{2\proxy k \log \bb{\frac{4n}{\gamma}}} \iff \min(M,n) \ge \frac{\log (4nt/\gamma)^2}{\log 4n/\gamma}.
}
Since $M \ge \frac{n}{k}\log\left(\frac{n}{\gamma}\right)$, the assumption on the sample complexity of $n$ implies the above result.

Conditioned on $\mathcal{E}_\gamma$, the projection steps are never invoked in Algorithm~\ref{alg:PrivateMean} or~\ref{alg:OneStep}, so we have $\tpss = \tss + W_{t+1}$, where $W_{t+1}$ is a Laplace random variable with parameter $\frac{2\Delta}{\eps}$, where $\Delta = \frac{\dep{n}{k}{\F}\Yc{\gamma}}{\eps}$ (coming from the noise added to $\frac{1}{M} \sum_{i=1}^M Y_i$ in the $(t+1)^{th}$ step of Algorithm~\ref{alg:PrivateMean}).
Finally, using Lemma~\ref{lem:ssdep}, we have
\bas{
|\tpss - \tss| & = |W_{t+1}| \le \frac{2\Delta}{\eps} \log \frac{1}{\gamma} = \frac{2\dep{n}{k}{\mathcal{S}} \Yc{\gamma}}{\eps} \log \left(\frac{1}{\gamma}\right) \le \frac{8k}{n\eps} \log \frac{1}{\gamma} \sqrt{2\tau k \log \left(\frac{4n}{\gamma}\right)} 
}
on the event $\mathcal{G}_\gamma$.
Combined with Lemmas~\ref{lem:mother_MoM}, \ref{lem:ssvar}, inequality~\eqref{eq:ssconc}, and Theorem~\ref{prop:mother},
with probability at least $1-7\gamma$, we obtain
\bas{
|\tpss-\theta|\leq O\bb{\frac{1}{\sqrt{\gamma}}\sqrt{\Var(U_n)}+\frac{1}{\sqrt{\gamma}}\sqrt{\frac{\zeta_k}{M}} + \frac{k}{n\eps} \log \frac{1}{\gamma} \sqrt{\tau k \log \left(\frac{4n}{\gamma}\right)}}.
}
The success probability is $1-7\gamma$ instead of $1-6\gamma$ because we also require $\dep{n}{k}{\F_{\text{ss}}}\le \frac{4k}{n}$, which holds with probability $1-\gamma$ as in Lemma~\ref{lem:ssdep}.
Algorithm~\ref{alg:OneStepArxiv} uses a constant failure probability of $\gamma=0.01$, which assures a success probability of at least $0.75$. This is further boosted by Wrapper~\ref{alg:MoM}. Now, an application of Lemma~\ref{lem:mother_MoM} gives the stated result.




\section{Proofs of lower bounds}\label{sec:lower}
In this appendix, we provide the proofs of our two lower bound results.

\subsection{Proof of Theorem~\ref{thm:lowernondegen}}
\label{AppThmLower}

We use the following two results:
\begin{lemma}[Lemma 6.2 in Kamath et al.~\cite{Kamath2020PrivateME}]\label{thm:kamath_lower}
    Let $\mathcal{P} = \{P_1, P_2, \dots\}$ be a finite family of distributions over a domain $\mathcal{X}$ such that for any $i \neq j$, the total variation distance between $P_i$ and $P_j$ is at most $\alpha$. Suppose there exists a positive integer $n$ and an $\eps$-differentially private algorithm $\mathcal{B}:\mathcal{X}^n \to \bbb{\bmb{\mathcal{P}}}$ such that for every $P_i \in \mathcal{P}$, we have
    \bas{
    \Pr_{X_1, \dots, X_n \sim P_i, \mathcal{B}}\bb{\mathcal{B}(X_1, \dots, X_n) = i} \ge 2/3.
    }
    Then, $n = \Omega\bb{\frac{\log |\mathcal{P}|}{\alpha \eps}}$.
\end{lemma}

\begin{lemma}[Proposition 4.1 in Arbel et al.~\cite{arbel2020strict}]\label{thm:bernoulli_proxy}
The Bernoulli distribution $\Bernoulli(p)$ is sub-Gaussian with optimal variance proxy $\tau_p$, where
\bas{
\tau_p = \tau_{1-p} = \frac{\frac{1}{2}-p}{\log \bb{\frac{1}{p}-1}},
}
for $p \in (0,1) \setminus \{1/2 \}$. In particular, if $0 < p < 1/10$, then $\tau_p \le  \frac{1}{2 \log \frac{1}{2p}}.$
\end{lemma}

Define $\mathcal{D}_0 = \Bernoulli(1)$ and $\mathcal{D}_1 = \Bernoulli(1-\beta)$, where $\beta = \frac{c}{n\eps}$ and $c>0$ is small enough such that $k\beta < 1/10$. The TV-distance between $\mathcal{D}_0$ and $\mathcal{D}_1$ is $\beta$. Since $n = \frac{c}{\beta \eps}$, we also choose $c$ small enough such that Lemma~\ref{thm:kamath_lower} is violated: for any $\eps$-differentially private algorithm $\mathcal{B}:\{0,1\}^n \to \{0,1\}$, there exists an $i \in \{0,1\}$ such that 
\ba
{\label{eq:B_guarantee}
\Pr\bb{\mathcal{B}(X_1, \dots, X_n) = i} < 2/3,
}
by Lemma~\ref{thm:kamath_lower}.

Consider now the task of $\eps$-privately estimating the parameter \bas{
\theta(D) \defeq \E_{X_1,\dots,X_k\sim \mathcal{D}}\bbb{h(X_1, X_2, \dots, X_k)},
}
where $X_i \sim \mathcal{D}$ and
\bas{
h(X_1, X_2, \dots, X_k) = \frac{1}{\sqrt{\tau_{(1-\beta)^k}}} \mathbbm{1}\bb{X_1=X_2=\dots=X_k=1},
}
for some distribution $\mathcal{D}$. Suppose there exists an $\eps$-differentially private algorithm $\mathcal{A}$ such that
\ba{\label{eq:A_guarantee}
\E_{X_1, \dots, X_n \sim \mathcal{D}, \mathcal{A}} \Big[\bmb{\mathcal{A}(X_1, \dots, X_n)-\theta(\mathcal{D})}\Big] \le \frac{1}{8} \cdot \frac{1-(1-\beta)^k}{\sqrt{\tau_{(1-\beta)^k}}},
}
for any $\mathcal{D}$ such that the distribution $\mathcal{H}$ of $h(X_1, \dots, X_k)$ is $\text{sub-Gaussian}(1)$.
If $\mathcal{D} = \Bernoulli(1)$ or $\Bernoulli(1-\beta)$, Lemma~\ref{thm:bernoulli_proxy} shows that the distribution $\mathcal{H}$ of $h(X_1, \dots, X_k)$ is indeed $\text{sub-Gaussian}(1)$.
If inequality~\eqref{eq:A_guarantee} holds, then by Markov's inequality, 
\bas{
\Pr_{X_i \sim \mathcal{D}_i}\bb{\bmb{\mathcal{A}(X_1, \dots, X_n)-\theta\bb{\mathcal{D}_i}} \le \frac{3}{8} \cdot \frac{1-(1-\beta)^k}{\sqrt{\tau_{(1-\beta)^k}}}} \ge \frac{2}{3},
}
for $i \in \{0,1\}$.

Also, $\theta\bb{\mathcal{D}_0} = 1/\sqrt{\tau_{1-\beta}}$ and $\theta\bb{\mathcal{D}_1} = (1-\beta)^k/\sqrt{\tau_{1-\beta}}$. The difference between these means is
\bas{
\theta\bb{\mathcal{D}_0} - \theta\bb{\mathcal{D}_1} = \frac{1-(1-\beta)^k}{\sqrt{\tau_{(1-\beta)^k}}}.
}Therefore, the following algorithm violates inequality~\eqref{eq:B_guarantee}: Run $\mathcal{A}$ on $X_1, \dots, X_n$ to obtain $\tilde{\theta}$. Output $0$ if $\tilde{\theta}$ is closer to $\theta(\mathcal{D}_0)$ than to $\theta(\mathcal{D}_1)$, and output $1$ otherwise.

This implies inequality~\eqref{eq:A_guarantee} does not hold, so we have a lower bound on the expected error. Theorem~\ref{thm:lowernondegen} follows from the calculation
\bas{
\frac{1-(1-\beta)^k}{\sqrt{\tau_{(1-\beta)^k}}} \ge \frac{k\beta}{2} \cdot \sqrt{2 \log \frac{1}{2\bb{1-(1-\beta)^k}}} = \Theta\bb{k\beta \sqrt{\log \frac{1}{k\beta}}},
}
where we used $1 - k\beta < (1-\beta)^k < 1 - k\beta/2$ for $k\beta < 1/10$.

\subsection{Proof of Theorem~\ref{thm:lowerdegen}}
\label{AppThmLowerDeg}

Consider two datasets $D_0$ and $D_1$ of size $n$ each, differing in at most $1/\eps$ data points. Suppose $\E_{\mathcal{A}} |\mathcal{A}(D)-U_n(D)| < \frac{1}{10} |U_n(D_0)-U_n(D_1)|$ for $D \in \{D_0, D_1\}$. By Markov's inequality, we have
\bas{\Pr\bb{|\mathcal{A}(D)-U_n(D)| < \frac{1}{2}|U_n(D_0)-U_n(D_1)|} \ge 0.8,}
for $D \in \{D_0, D_1\}$. Moreover, since $\mathcal{A}$ is $\eps$-differentially private, we have
\bas{
& \Pr\bb{|\mathcal{A}(D_1)-U_n(D_0)| < \frac{\bmb{U_n(D_0)-U_n(D_1)}}{2}} \\
& \quad \ge \frac{1}{e} \Pr\bb{|\mathcal{A}(D_0)-U_n(D_0)| < \frac{\bmb{U_n(D_0)-U_n(D_1)}}{2}} \ge \frac{0.8}{e}.
}
By the triangle inequality, the event $\left\{|\mathcal{A}(D_1)-U_n(D_0)| < \frac{|U_n(D_0)-U_n(D_1)|}{2}\right\}$ is disjoint from the event $\left\{|\mathcal{A}(D_1)-U_n(D_1)| < \frac{|U_n(D_0)-U_n(D_1)|}{2}\right\}$. The sum of the probabilities of these two events is at least $0.8 + 0.8e^{-1} > 1$, a contradiction. Therefore, $\mathcal{A}$ has expected error $\Omega(|U_n(D_0)-U_n(D_1)|)$ on at least one of $D_0$ or $D_1$. Next, we will define appropriate choices of $D_0$ and $D_1$.
For simplicity, assume $1/\epsilon$ is an integer. Define $b_n = \lceil k + k^{1/(2k-2)} n^{1-1/(2k-2)} - \frac{1}{\eps} \rceil$. The assumed range of $\eps$ implies that $b_n \ge 2k/\eps$. Let $h(x_1,\dots,x_k)=1(x_1=\dots=x_k)$. We define $D_0$ such that
\begin{equation*}
x_i = \begin{cases}
    1, & i \le b_n, \\
    i, & i > b_n.
\end{cases}
\end{equation*}
We define $D_1 = \{y_1, \dots, y_n\}$ such that $y_i=x_i$ for all $i\not\in \{b_n+1,\dots, b_n+1/\epsilon\}$ and $y_i = 1$ for $b_n < i \le b_n + \frac{1}{\epsilon}$. Hence,
\begin{equation*}
U_n(D_1) - U_n(D_0) = \frac{\binom{b_n + 1/\epsilon}{k}}{\binom{n}{k}} - \frac{\binom{b_n}{k}}{\binom{n}{k}}.
\end{equation*}
Furhtermore, for $\frac{k}{\eps} \le b_n$, \bk
using the fact that $(1+x)^r\leq \frac{1}{1-rx}$ for $x\in [-1,1/r)$, we have
\bas{
& 1-\frac{{b_n\choose k}}{{b_n+1/\epsilon\choose k}} =1-\prod_{i=0}^{k-1} \frac{b_n-i}{b_n+1/\eps-i} \geq 1 - \left(\frac{b_n}{b_n + 1/\epsilon}\right)^k \\
& \quad = 1 - \bb{1-\frac{1/\eps}{b_n+1/\eps}}^{k} \geq 1 - \frac{1}{1+\frac{k/\eps}{b_n + 1/\epsilon}} = \frac{k/\eps}{b_n+(k+1)/\eps},
}
implying that
\ba{\label{eq:ungap}
U_n(D_1)-U_n(D_0)\geq \frac{k/\epsilon}{b_n+(k+1)/\epsilon}\left.{b_n+1/\epsilon\choose k}\right/{n\choose k}.
}

For $i$ with $x_i=1$ in $D_0$, we have $\h{D_0}{}{i}= \frac{{b_n-1\choose k-1}}{{n-1\choose k-1}} \le U_n(D_0)$. For $i$ with $x_i=1$ in $D_1$, we have $\h{D_1}{}{i} = \frac{{b_n+1/\epsilon-1\choose k-1}}{{n-1\choose k-1}} \le U_n(D_1)$. Therefore, we have $|\h{D}{}{i}-U_n(D)|\leq \h{D}{}{i} \le \xi$ for all $i$ and $D \in \{D_0, D_1\}$, where
\bas{
\xi \defeq \frac{{b_n+1/\epsilon-1\choose k-1}}{{n-1\choose k-1}} = \prod_{i=1}^{k-1}\bb{\frac{b_n+1/\epsilon-i}{n-i}} \leq \bb{\frac{b_n+1/\epsilon}{n}}^{k-1} = O\bb{\sqrt{\frac{k}{n}}},\\
}
by our choice of $b_n$. Moreover, we have
\ba{
\label{eq:xilower}
\xi \ge \bb{\frac{b_n+1/\epsilon - k}{n-k}}^{k-1} \ge \bb{\frac{k^{1/(2k-2)}n^{1-1/(2k-2)}}{n}}^{k-1} =\sqrt{\frac{k}{n}}.
}

By inequality~\eqref{eq:ungap} and the definition of $\xi$, we see that
\bas{
U_n(D_1)-U_n(D_0)\geq \frac{k/\epsilon}{b_n+2k/\epsilon}\frac{b_n+1/\epsilon}{n}\xi\geq \frac{k}{3n\epsilon}\xi,
}
where the second inequality follows from the assumption that $k/\epsilon\leq b_n$.
Using the lower bound on $\xi$ as in inequality~\eqref{eq:xilower}, we obtain the desired result.
\section{Proof of Theorems~\ref{thm:mainthm} and~\ref{thm:mainthm_subsample}}
\label{sec:app-mainthm}

We first prove Theorem~\ref{thm:mainthm_subsample} with $\SubS$ equal to any subsampled family that satisfies the inequalities~\eqref{eq:subsample_conc}.
In particular, the following lemma guarantees that the required bounds hold with high probability for a subsampled family chosen uniformly at random from $\I_{n,k}$:

\begin{lemma}\label{lem:ss_mi_m}
Let $\gamma > 0$, and let $M = \Omega\left(\frac{n^2}{k^2}\log\left(\frac{n}{\gamma}\right)\right)$.
Let $\SubS$ be a collection of $M$ i.i.d\ sets sampled uniformly from $\I_{n,k}$. For each $i \in [n]$, let $\SubS_i$ be the number of sets in $\SubS$ containing $i$, and define $M_i = |\SubS_i|$. For each $i \neq j \in [n]$, let $\SubS_{ij}$ be the number of sets in $\SubS$ containing $i$ and $j$, and define $M_{ij} = |\SubS_{ij}|$. With probability at least $1-\gamma$, for all distinct indices $i$ and $j$, we have
\bas{
\frac{k}{2n} \le \frac{M_i}{M} \le \frac{2k}{n}, \qquad \frac{k}{2n} \le \frac{M_{ij}}{M_i} \le \frac{2k}{n}.
}
\end{lemma}
\begin{proof}
Note that $M_i \sim \text{Binom}(M,k/n)$. By a Chernoff bound, for any $\delta > 0$ and any $i \in [n]$, we have
\bas{
\Pr\bb{\bmb{M_i - \frac{kM}{n}} > \frac{kM}{2n}} \le e^{-(0.5)^2(kM/n)/3} \le e^{-\Omega\left(\frac{n}{k} \log \frac{n}{\gamma}\right)},
}
which is much smaller than $\frac{\gamma}{2n}$. Call this event $\mathcal{E}_i$, and for the remaining argument, assume $\mathcal{E}_i$ holds for all $i$ (which, by a union bound, holds with probability at least $1-\frac{\gamma}{2}$). This gives the first inequality.

For any distinct $i,j \in [n]$, conditioned on the value of $M_i$, we have $M_{ij} \sim \text{Binom}(m_i,(k-1)/(n-1))$. By a Chernoff bound, for any $\delta > 0$ and $i, j \in [n]$ with $j \neq i$, we have
\bas{
\Pr\bb{\bmb{M_{ij} - \frac{(k-1)M_i}{(n-1)}} > \frac{(k-1)M_i}{2(n-1)} \Big|~ M_i} &\le e^{-\frac{(0.5)^2((k-1)M_i/(n-1))}{3}} \le e^{-\frac{k^2 M}{48n^2}} \le \frac{\gamma}{2n^2},
}
using the fact that $M_i \ge \frac{kM}{2n}$ and our assumption on $M$. The second inequality then follows from a union bound over all pairs of indices.
\end{proof}


\subsection{Proof of Theorem~\ref{thm:mainthm_subsample}}
\label{AppThmMain}






Consider two adjacent datasets $\X = (X_1, X_2, \ldots, X_n)$ and $\X' = (X'_1, X'_2, \ldots, X'_n)$ differing only in the index $i^*$, that is, $X'_i = X_i$ for all $i \neq i^*$. Throughout the proof, we will use the superscript prime to denote quantities related to $\X'$.

Let $B \defeq \bb{\Bad(\X ) \cup \Bad(\X ')} \setminus \{i^*\}$ and $G \defeq \bb{\Good(\X ) \cap \Good(\X ') } \setminus \{i^*\}$. Then
\ba{\label{eq:sum_decomp}
&m \bb{\tSubA - \tSubA'} = \sum_{\substack{S \in \SubS }}\bb{g(X_S)-g(X'_S)} \notag \\
&= \underbrace{\sum_{\substack{S \cap B \neq \varnothing \\ i^* \notin S}} \bb{g(X_S)-g(X'_S)}}_{T_1} + \underbrace{\sum_{\substack{S \cap B = \varnothing \\ i^* \notin S}} \bb{g(X_S)-g(X'_S)} }_{T_2} + \underbrace{\sum_{\substack{i^* \in S}} \bb{g(X_S)-g(X'_S)}}_{T_3}.
}
We bound each of the three terms separately. The term $T_2$ is equal to $0$: all indices $i \in S$ have weight $1$, and $i^* \notin S$, so $g(X_S) = h(X_S) = h(X'_S) = g(X'_S)$. We prove some preliminary lemmas before bounding the first and last terms. Recall the definitions of $L$ and $\wt$ in equations~\eqref{eq:buf_definition} and~\eqref{eq:weight}, respectively.

\begin{lemma} \label{lem:bufX_sensitivity}
We have:
\begin{enumerate}
\item[(i)] $|A_n - A_n'| \le \frac{2kC}{n}$ and $|\buf{}-\buf{}'| \le 1$.
\item[(ii)] For all $i\neq i^*$, we have 
\ba{\label{eq:hXdiff}\bmb{|\h{'}{X'}{i}-\SubA'| - |\h{}{X}{i}-\SubA|} \le \frac{4kC}{n},} 
\ba{\label{eq:wtidiff}
\bmb{\wt(i)-\wt'(i)} \le \eps,
}
and for $S$, such that $i^*\not\in S$,
\ba{\label{eq:wtSdiff}
\bmb{\wt(S)-\wt'(S)} \le \eps.
}
\end{enumerate}
\end{lemma}

\begin{proof}
For (i), note that 
\ba{\label{eq:Udiff}
\bmb{\SubA - \SubA'} = \frac{1}{M} \bmb{\sum_{S \in \SubS_{i^*}} \bb{h(X_S) - h(X'_S)}} \le \frac{M_i C}{M} \le \frac{2kC}{n},
}
where the last inequality comes from Lemma~\ref{lem:ss_mi_m}.
Similarly, for any $i \neq i^*$, we have
\ba{\label{eq:hhajekdiff}
\bmb{\h{}{X}{i}-\h{'}{X'}{i}} = \frac{1}{M_i} \bmb{\sum_{S \in \SubS_{ij}} \bb{h(X_S) - h(X'_S)}} \le \frac{M_{ij} C}{M_i} \le \frac{2kC}{n}.
}
Therefore, if an index $i \neq {i^*}$ is in $\Good(\X)$, using inequalities~\eqref{eq:Udiff} and~\eqref{eq:hhajekdiff}, we have
\ba{
\bmb{\h{'}{X'}{i}-\SubA'} &\le \bmb{\h{'}{X'}{i}-\h{}{X}{i}}+\bmb{\h{}{X}{i}-\SubA}+\bmb{\SubA - \SubA'} \label{eq:hu_triangle}\\
&\le \frac{2kC}{n} + \bb{\conc +  \frac{6kC\buf{}}{n}} + \frac{2kC}{n} = \conc +  \frac{kC\bb{4+6\buf{}}}{n} \notag,
}
which leaves at most $1+\buf{}$ potential indices $i$ for which $|\h{'}{X'}{i}-\SubA'| > \conc +\frac{6kC(1+\buf{})}{n}$: the indices in $\Bad(\X)$ and the index ${i^*}$. Therefore, $\buf{}' \le \buf{}+1$. Similarly, $\buf{} \le \buf{}'+1$.

For (ii), note that from inequalities~\eqref{eq:Udiff} and~\eqref{eq:hhajekdiff}, we have $\bmb{|\h{'}{X'}{i}-\SubA'| - |\h{}{X}{i}-\SubA|} \le \frac{4kC}{n}$ for $i \neq i^*$. Recalling the definition~\eqref{eq:weight}, this implies that the difference between the weights on an index $i$ can never be greater than $\epsilon$.

Finally, note that the weight of a subset $S$, $wt(S)=\min_{i\in S}wt(i)$. Now, by inequality~\eqref{eq:wtidiff}, each $wt(i)$ differs by $\epsilon$. Say $a = \arg\min_{i\in S} wt(i)$. In order to make the difference between $wt(S)$ and $wt(S')$ large we will set $wt'(a)=wt(a)+\epsilon$ and take some other $b$ and set $wt'(b)=wt(b)-\epsilon$ such that $wt'(b)\leq wt'(a)$. But then, $wt(a)\leq wt(b)\leq wt(a)+2\epsilon$. This completes the proof of inequality~\eqref{eq:wtSdiff}.
\end{proof}

Next, we show that the weighted Hájek variants $\g{}{X}{i}$ are close to the empirical mean $\SubA$ and have low sensitivity.
\begin{lemma}\label{lem:g_proximity}
For all indices $i$, we have $\bmb{\g{}{X}{i}-\SubA} \le \conc + \frac{9kCL}{n} + \frac{6kC}{n\eps}$. Moreover, if $i \neq i^*$, we have
\bas{
\bmb{(\g{}{X}{i}-\SubA)-(\g{'}{X'}{i}-\SubA')} \le \bb{\conc + \frac{kC(14+6\buf{})}{n}}\eps + \frac{10kC}{n} +  \frac{4k^2 C}{n^2}(1+2\buf{}).
}
\end{lemma}
\begin{proof}
If $\wt(i) = 0$, then $g(X_S) = \SubA$ for all $S \ni i$ and $\g{}{X}{i} = \SubA$. 
For clarity, we add a picture of this weighting scheme here:
\begin{figure}
    \centering
    \includegraphics[width=0.7\textwidth]{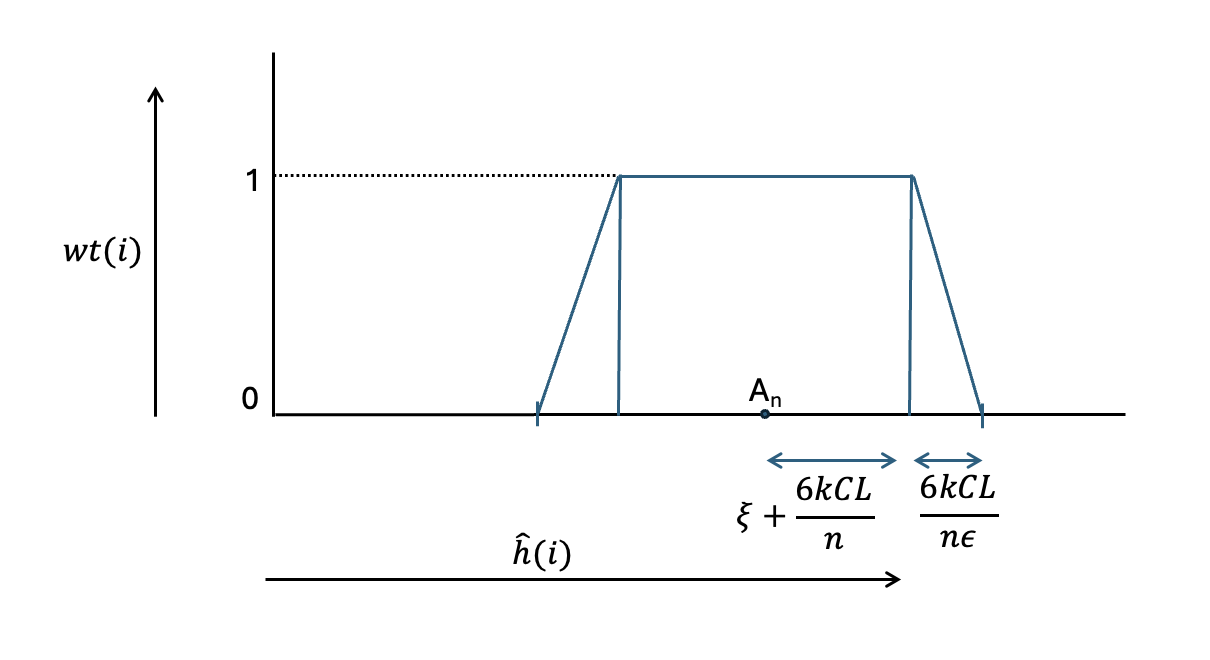}
    \caption{Weighting scheme in Eq~\eqref{eq:weight}}
    \label{fig:weight}
\end{figure}
Otherwise, we write
\ba{\label{eq:g_decomposition}
\g{}{X}{i}-\SubA &= \frac{1}{M_i}\sum_{S \in \SubS_i} (h(X_S)-\SubA)\wt(S) \notag \\
&= \frac{1}{M_i}\sum_{S \in \SubS_i} \bb{h(X_S)-\SubA}\wt(i) + \frac{1}{M_i}\sum_{S \in \SubS_i} \bb{h(X_S)- \SubA}\bb{\wt(S)-\wt(i)} \notag \\
&= (\h{}{X}{i}-\SubA)\wt(i) + \frac{1}{M_i}\sum_{S \in \SubS_i} \bb{h(X_S)- \SubA}\bb{\wt(S)-\wt(i)}.
}

From equation~\eqref{eq:weight} and the assumption that the weight of index $i$ is strictly positive, the magnitude of the first term in equation~\eqref{eq:g_decomposition} is bounded by $\conc + \frac{6kCL}{n} + \frac{6kC}{n\eps}$. For the second term, note that $\wt(S) = \wt(i)$ unless an index $j$ with a lower weight than $i$ exists. Note that such an index $j$ is necessarily in $\Bad(\X)$. Therefore, the absolute value of the second term is bounded above by $\frac{1}{m_i} \sum_{j \in \Bad(\X), j \neq i} C \le \frac{2kC L}{n}$. This proves the first part of the lemma.

To bound the sensitivity of $\g{}{X}{i}$, by the triangle inequality, we have 
\ba{\label{eq:g_decomposition_term1}
&\bmb{(\h{}{X}{i}-\SubA)\wt(i) - (\h{'}{X'}{i}-\SubA')\wt'(i)} \notag \\
&\le \bmb{\h{}{X}{i}-\SubA} \bmb{\wt(i) - \wt'(i)} + \bmb{(\h{}{X}{i}-\SubA) - (\h{'}{X'}{i}-\SubA')}\wt'(i) \notag \\
&\le \bb{\conc + \frac{kC\bb{4+6\buf{}}}{n} + \frac{6kC}{n\eps}} \eps + \frac{6kC}{n} \notag \\
& = \bb{\conc + \frac{kC\bb{10+6\buf{}}}{n}}\eps + \frac{6kC}{n},
}
where the argument for the second inequality is as follows: To bound the first term, note that when $\wt(i) =\wt'(i)$, it is zero. If $\wt(i)>0$, then $|\h{}{X}{i}-A_n|\leq \xi+\frac{6kCL}{n}+\frac{6kCL}{n \epsilon}$. Now, if $|\h{}{X}{i}-A_n| > \xi+\frac{6kCL}{n}+\frac{6kCL}{n\epsilon}+\frac{4kC}{n}$, then $\wt(i) = 0$, and by inequality~\eqref{eq:hXdiff} of Lemma~\ref{lem:bufX_sensitivity}, we see that $|\h{}{X'}{i}-A'_n|>\xi+\frac{6kCL}{n}+\frac{6kCL}{n\epsilon}$, so $\wt'(i)$ will also be zero. 
The second term is bounded directly by Lemma~\ref{lem:bufX_sensitivity}. 
Overall, we arrive at a bound on the sensitivity of the first term in equation~\eqref{eq:g_decomposition}.

For the sensitivity of the second term of equation~\eqref{eq:g_decomposition}, note that if $i$ has minimum weight among the indices in $S$, then $\wt(S) = \wt(i)$. Otherwise, some index $j \in S$ has strictly lower weight than $i$. Such an index $j$ is necessarily in $\Bad(\X) \cup \Bad(\X')$ because it has weight less than $1$. If $S$ also does not contain the index $i^*$, then $h(X_S) = h(X'_S)$, so by Lemma~\ref{lem:bufX_sensitivity}, we have
\bas{
\bmb{(h(X'_S) - \SubA') - (h(X_S) - \SubA)} \le \bmb{\SubA-\SubA'} \le \frac{2kC}{n},
}
\bas{
\bmb{(\wt'(S)-\wt'(i))-(\wt(S)-\wt(i))} \le 2\eps,
}
and letting
\bas{
T_S:=\bb{\bb{h(X_S)- \SubA}\bb{\wt(S)-\wt(i)}- \bb{h(X'_S)- \SubA'}\bb{\wt'(S)-\wt'(i)}},
}
we have $|T_S| \leq \frac{2kC}{n}+2C\epsilon$.

Moreover, there are at most $M_{i,i^{*}}$ sets $S$ containing both $i$ and $i^*$, and for each such set $S$, the change in $\bb{h(X_S)- \SubA}\bb{\wt(S)-\wt(i)}$ is at most $2C$, since the weights lie in $[0,1]$ and $|h(X_S) - \SubA| \le C$. Combining these bounds,
we obtain
\ba{\label{eq:g_decomposition_term2}
&\bmb{
\frac{1}{M_i}\sum_{S \in \SubS_i} \underbrace{\bb{\bb{h(X_S)- \SubA}\bb{\wt(S)-\wt(i)}- \bb{h(X'_S)- \SubA'}\bb{\wt'(S)-\wt'(i)}}}_{T_S}}
\notag \\
&\leq \frac{1}{M_i}\bb{\sum_{S\in \mathcal{S}_{i}\setminus \SubS_{ii^*}} |T_S|1(|S\cap (\Bad(\X)\cup \Bad(\X'))|>0) +\sum_{S\in \mathcal{S}_{ii^*}} |T_S|} \notag \\
&\leq \bb{2C\eps + \frac{2kC}{n}}\frac{1}{M_i}\sum_{S\in \mathcal{S}_{i}\setminus \SubS_{ii^*}} 1(|S\cap (\Bad(\X)\cup \Bad(\X'))|>0)  + \frac{1}{M_i} \sum_{S\in  \mathcal{S}_{ii^*}} 2C \notag \\
&\le \bb{2C\eps + \frac{2kC}{n}} \frac{1}{M_i} \sum_{j\in \Bad(\X)\cup \Bad(\X')}\sum_{S\in \SubS_{ij}}1+ 2C \frac{M_{i,i^*}}{M_i} \notag \\
&\le \bb{2C\eps + \frac{2kC}{n}}\bb{\bmb{\Bad(\X)}+\bmb{\Bad(\X')}} \sup_{j \in \Bad(\X) \cup \Bad(\X')} \frac{M_{i,j}}{M_i} + 2C \frac{M_{i,i^*}}{M_i} \notag \\
&\le \frac{2k}{n} \bb{2C\eps + \frac{2kC}{n}}(1+2\buf{}) + \frac{4kC}{n},
}
where the first inequality uses the fact that the weights are all equal if $S \cap (\Bad(\X)\cup \Bad(\X')) = \emptyset$, 
and the last inequality uses Lemma~\ref{lem:ss_mi_m}.
Combining inequalities~\eqref{eq:g_decomposition_term1} and~\eqref{eq:g_decomposition_term2} into equation~\eqref{eq:g_decomposition} yields the result.
\end{proof}

To bound the term $T_1$ in~\eqref{eq:sum_decomp}, we decompose it as
\ba{\label{eq:T1_decomp}
T_1 &= \sum_{i \in B} \sum_{\substack{S \in \SubS_i \\ i^* \notin S}}\bb{g(X_S)-g(X'_S)} - \sum_{a=2}^{\min(k,\bmb{B})} \sum_{\substack{S \in \SubS \\ |S \cap B| = a \\ i^* \notin S}} (a-1) \bb{g(X_S)-g(X'_S)}.
}
The first term sums over all subsets that contain some element in $B$. However, this leads to over-counting every subset with $a$ elements in common with $B$ exactly $a-1$ times. The second term corrects for this over-counting, akin to an inclusion-exclusion argument.
The following lemmas bound each of the two terms:
\begin{lemma}\label{lem:T1_decomp_term1}
For all $i \in B$, we have
\bas{\frac{1}{M_i} \bmb{\sum_{S \in \SubS_i, i^* \notin S} \bb{g(X_S)-g(X'_S)}} \le \bb{\conc + \frac{20kC\buf{}}{n}}\eps + \frac{12kC}{n} +  \frac{14k^2 C \buf{}}{n^2}.}
\end{lemma}

\begin{proof}
We have
\ba{
\label{EqnBunny}
\sum_{\substack{S \in \SubS_i \\ i^* \notin S}} \bb{g(X_S)-g(X'_S)} = M_i\bb{\g{}{X}{i}-\g{}{X'}{i}} - \sum_{\substack{S \in \SubS_{i,i^*}}} \bb{g(X_S)-g(X'_S)}.
}
By Lemmas~\ref{lem:bufX_sensitivity} and~\ref{lem:g_proximity}, we have
\bas{\bmb{\g{}{X}{i}-\g{}{X'}{i}} \le 
\bb{\conc + \frac{kC(14+6\buf{})}{n}}\eps + \frac{12kC}{n} +  \frac{4k^2 C}{n^2}(1+2\buf{}).
}
Moreover, the second term in equation~\eqref{EqnBunny} is clearly upper-bounded by
$$2CM_{i,i^*} \le \frac{4kC}{n}M_i \le \frac{8k^2C}{n^2}M_i.$$
Summing these two bounds yields the result.
\end{proof}

For the second term of equation~\eqref{eq:T1_decomp}, we use the following lemma:

\begin{lemma}\label{lem:T1_decomp_term2}
We have
\bas{\sum_{a=2}^{\min(k,|B|)} \sum_{\substack{S \in \SubS \\ |S \cap B| = a \\ i^* \notin S}} (a-1) \bmb{g(X_S)-g(X'_S)} \le \frac{36k^2 \buf{}^2}{n^2}\bb{2C\eps + \frac{6kC}{n}} \min(k,2\buf{}) M.}
Moreover, if $\SubS = \I_{n,k}$, with $m = \binom{n}{k}$, we have the stronger inequality
\bas{\sum_{a=2}^{\min(k,|B|)} \sum_{\substack{S \in \SubS \\ |S \cap B| = a \\ i^* \notin S}} (a-1) \bmb{g(X_S)-g(X'_S)} \le \frac{9k^2 \buf{}^2}{n^2} \bb{2C\eps + \frac{6kC}{n}} m.}
\end{lemma}

\begin{proof}
For any $S$ not containing $i^*$, we have
\bas{
\bmb{g(X_S)-g(X'_S)} &= \bmb{(h(X_S)-\SubA)\wt(S)+\SubA - (h(X'_S)-\SubA')\wt'(S)-\SubA'} \\
&\le \bmb{(h(X_S)-\SubA)(\wt(S)-\wt'(S))} + \bmb{(h(X_S)-h(X'_S))\wt'(S)} + \bmb{\SubA-\SubA'} \\
&\le 2C\eps + \frac{6kC}{n},
}
using Lemma~\ref{lem:bufX_sensitivity} and the fact that the second term is zero.
Moreover, we have
\bas{
\sum_{a=2}^{\min(k,|B|)} \sum_{\substack{S \in \SubS \\ |S \cap B| = a \\ i^* \notin S}} (a-1) &\le \sum_{a=2}^{k} \sum_{\substack{S \in \SubS \\ |S \cap B| = a}} \min(k,|B|) = \sum_{\substack{S \in \SubS \\ |S \cap B| \ge 2}} \min(k,|B|) \\
&\le \sum_{\substack{i, j \in B \\ i \neq j}} \sum_{S \in \SubS_{ij}} \min(k,|B|) \le \frac{9k^2 }{n^2}\min(k,|B|)|B|^2 m. \\
}
The last inequality follows from Lemma~\ref{lem:ss_mi_m}, which implies that $\frac{M_{ij}}{M}=O\left(\frac{k^2}{n^2}\right)$.

In the case when $\SubS = \I_{n,k}$, we have
\bas{
\sum_{a=2}^{\min(k,|B|)} \sum_{\substack{S \in \SubS \\ |S \cap B| = a \\ i^* \notin S}} (a-1) &\le \sum_{a=2}^{\min(k,|B|)} (a-1) \binom{|B|}{a}\binom{n-|B|}{k-a} \\
&= \binom{n}{k}\bb{\frac{k|B|}{n}-1} + \binom{n-|B|}{k} \\
&\le \binom{n}{k}\bb{\frac{k|B|}{n}-1} + \binom{n}{k}\bb{1-\frac{|B|}{n}}^k \\
&\le \binom{n}{k}\bb{\frac{k|B|}{n}-1} + \binom{n}{k} \frac{1}{\frac{k|B|}{n}+1} \le \frac{k^2 |B|^2}{n^2} \binom{n}{k},
}
where the first equality used the identities $\sum_{a=0}^{k} \binom{n}{a}\binom{m}{k-a} = \binom{n+m}{k}$ and $ \sum_{a=0}^{k} a\binom{n}{a}\binom{m}{k-a} = \frac{nk}{m+n}\binom{n+m}{k}$, and
the third inequality used the fact that $(1-x)^k \le \frac{1}{1+kx}$ for all $x \in [0,1]$. The statement in the lemma follows because $|B| \le 2L+1 \le 3L$.
\end{proof}

Combining the results of Lemma~\ref{lem:T1_decomp_term1} and~\ref{lem:T1_decomp_term2} yields Lemma~\ref{lem:T1_upper_bound}.
\begin{lemma}\label{lem:T1_upper_bound}
We have
$$\bmb{T_1} \le \frac{k M}{n} \bb{\bb{\conc + \frac{20kC\buf{}}{n}}2\eps + \frac{24kC}{n} +  \frac{28k^2 C \buf{}}{n^2}+\frac{9k \buf{}^2}{n}\bb{2C\eps + \frac{6kC}{n}} \gamma},$$ where $\gamma = 4\min(k,2L)$. If $\SubS = \I_{n,k}$, then the bound also holds for $\gamma = 1$.
\end{lemma}
It remains to bound the third term, $T_3$, of equation~\eqref{eq:sum_decomp}, which we do in the following lemma:

\begin{lemma}\label{lem:T3_decomp_term1}
We have $$\bmb{T_3} \le \frac{2k}{n}\bb{2\conc + \frac{kC(11+18L)}{n} + \frac{12kC}{n\eps}}M.$$ 
\end{lemma}

\begin{proof}
Using Lemmas~\ref{lem:bufX_sensitivity} and~\ref{lem:g_proximity}, we have
\bas{
\frac{1}{M_{i^*}} \bmb{T_3} &= \bmb{\g{}{X}{i^*}-\g{'}{X'}{i^*}} \le \bmb{\g{}{X}{i^*}-\SubA} + \bmb{\g{'}{X'}{i^*} - \SubA'} + \bmb{\SubA-\SubA'} \\
&\le \bb{\conc + \frac{9kCL}{n} + \frac{6kC}{n\eps}}+\bb{\conc + \frac{9kCL'}{n} + \frac{6kC}{n\eps}}+\frac{2kC}{n} \\
&\le 2\conc + \frac{kC(11+18L)}{n} + \frac{12kC}{n\eps}.
}
The lemma follows after using the fact that $M_{i^*} \le \frac{2k}{n} M$.
\end{proof}

Combining the bounds on $T_1$ and $T_3$ from Lemmas~\ref{lem:T1_upper_bound} and~\ref{lem:T3_decomp_term1} in equation~\eqref{eq:sum_decomp}, the local sensitivity of $\tSubA$ at $\X$ is then bounded as
\bas{
\buf{}S_{\tSubA}\bb{\X} = O\bb{\frac{k}{n} \bb{\conc + \frac{kC\buf{}}{n}}\bb{1+\eps \buf{}} + \frac{k^2 C \buf{}^2 \min(k,\buf{})}{n^2}\bb{\eps+\frac{kC}{n}} + \frac{k^2C}{n^2\eps}}.
}
Let $g(\xi, L, n)$ denote the upper bound, where to simplify the following argument, we assume the constant prefactor is $1$, i.e., 
\bas{
g(\conc ,\buf{},n) :=  \frac{k}{n} \bb{\conc + \frac{kC\buf{}}{n}}\bb{1+\eps \buf{}} + \frac{k^2 C \buf{}^2 \min(k,\buf{})}{n^2}\bb{\eps+\frac{k}{n}} + \frac{k^2C}{n^2\eps}.
}
Note that $g$ is strictly increasing in $\buf{}$. Also define
\ba{\label{eq:smooth_sensitivity}
S(\X) = \max_{\ell \in \mathbb{Z}_{\ge 0}} e^{-\eps \ell} g(\conc , \buf{\X}+\ell, n).
}
\begin{lemma}\label{lem:smooth_bound}
The function $S(\X)$ is an $\eps$-smooth upper bound on $LS_{\tSubA}(\X)$.
\begin{equation*}
S(\X) = O\bb{\frac{k}{n} \bb{\conc + \frac{kC(1/\eps + \buf{})}{n}}\bb{1+\eps \buf{}} + \frac{k^2 C (1/\eps + \buf{})^2 \min(k,1/\eps + \buf{})}{n^2}\bb{\eps+\frac{k}{n}} + \frac{k^2C}{n^2\eps}}.
\end{equation*}
\end{lemma}

\begin{proof}[Proof of Lemma~\ref{lem:smooth_bound}]
Clearly, we have $S(\X) \ge g(\conc , \buf{\X}, n) \ge LS_{\tSubA}(\X)$, and for any two adjacent $\X$ and $\X'$, we have
\bas{
S(\X') &= \max_{\ell \in \mathbb{Z}_{\ge 0}} e^{-\eps \ell} g(\conc , \buf{\X'}+\ell, n) \le \max_{\ell \in \mathbb{Z}_{\ge 0}} e^{-\eps \ell} g(\conc , \buf{\X}+\ell+1, n) \\
&= \max_{\ell \in \mathbb{Z}_{> 0}}e^{-\eps \bb{\ell-1}} g(\conc , \buf{\X}+\ell, n) \le e^{\eps}\max_{\ell \in \mathbb{Z}_{\ge 0}}e^{-\eps\ell} g(\conc , \buf{\X}+\ell, n) = e^{\eps}S(\X).
}
This shows that $S$ is indeed a $\eps$-smooth upper bound on the local sensitivity. As for the upper bound on $S$, for any $\ell \ge 0$, we have
\bas{
&e^{-\eps\ell}g(\conc , \buf{\X}+\ell, n) = \frac{k}{n} \bb{\conc e^{-\eps\ell/2} + \frac{kC(\ell e^{-\eps\ell/2}  + \buf{} e^{-\eps\ell/2})}{n}}\bb{e^{-\eps\ell/2}+\eps (\ell e^{-\eps\ell/2}+\buf{} e^{-\eps\ell/2})} \\
&+ \frac{k^2}{n^2} C (\ell e^{-\eps\ell/3} + \buf{} e^{-\eps\ell/3})^2 \min(k e^{-\eps\ell/3} ,\ell e^{-\eps\ell/3} + \buf{} e^{-\eps\ell/3})\bb{\eps+\frac{k}{n}} + \frac{k^2C e^{-\eps\ell}}{n^2\eps} \\
&\le \frac{k}{n} \bb{\conc + \frac{kC(1/\eps  + \buf{})}{n}}\bb{1+\eps(1+ \buf{})} + \frac{k^2}{n^2} C \bb{\frac{2}{\eps} + \buf{}}^2 \min\bb{k  , \frac{2}{\eps} + \buf{}}\bb{\eps+\frac{k}{n}} + \frac{k^2C}{n^2\eps},
}
where we used in multiple places the inequalities $e^{-\eps \ell} \le 1$ and  $\ell e^{-\ell/c} \le c/e$, for any $c > 0$.
\end{proof}

By Lemma~\ref{lem:smooth_bound}, it is clear that the term $S(\X)$ added to $\tSubA$ in Algorithm~\ref{alg:PrivateMeanDegenerate} is the smoothed sensitivity defined in equation~\eqref{eq:smooth_sensitivity}.

Therefore, $\tSubA+\frac{S(\X)}{\epsilon} \cdot Z$, where $Z$ is sampled from the distribution with density $h(z) \propto 1/(1+|z|^4)$, is $O(\epsilon)$-differentially private, by Lemma~\ref{lem:smoothed_mechanism}. Moreover, if $\SubS = \I_{n,k}$, the above bound on the smooth sensitivity holds without the $\min(k,1/\eps+L)$ term, due to Lemma~\ref{lem:T1_decomp_term2}.

\paragraph{Utility.}
By Chebyshev's inequality, we have
\bas{
\bmb{\SubA-\theta} \le \frac{1}{\sqrt{\gamma}}\sqrt{\Var(\SubA)} = \frac{1}{\sqrt{\gamma}}\bb{\sqrt{\Var(U_n)} + \sqrt{\frac{\zeta_k}{m}}},
}
with probability at least $1-\gamma$. Moreover, with probability at least $1-\gamma$, each of the Hájek projections is within $\conc$ of $\SubA$. This implies that every index $i$ has weight $1$, which further implies that $g(X_S) = h(X_S)$ for all $S \in \SubS$, and consequently, $\tSubA = \SubA$. Also, $L = 1$ for such an $\X$. Finally, with probability at least $1-\gamma$, we have $Z \le \frac{3}{\sqrt{\gamma}}$. Combining these inequalities and using Lemma~\ref{lem:smooth_bound}, we have
\ba{\label{eq:degen_subsample_error}
&\bmb{\mathcal{A}(\X)-\theta} \le \bmb{\tSubA-\SubA}+\bmb{\SubA-\theta}+\bmb{S(\X)/\epsilon \cdot Z} \notag\\
&= \frac{1}{\sqrt{\gamma}} O\bb{
\sqrt{\Var(U_n)} + \sqrt{\frac{\zeta_k}{m}} + \frac{k \conc}{n \eps} + \bb{\frac{k^2 C}{n^2 \eps^{2}} + \frac{k^3 C}{n^3 \eps^{3}}} \min \bb{k, \frac{1}{\eps}}},
}
with probability at least $1-4\gamma$, recalling that $\epsilon = O(1)$ when simplifying the expression.
Algorithm~\ref{alg:PrivateMeanDegenerate} uses a constant failure probability of $\gamma=0.01$, which ensures a success probability of at least $0.75$. This is further boosted by Wrapper~\ref{alg:MoM}. Now, an application of Lemma~\ref{lem:mother_MoM} gives the stated result.

\subsection{Proof of Theorem~\ref{thm:mainthm}}
\label{SecAppMainThm}

The proof of this theorem proceeds nearly identically to that of Theorem~\ref{thm:mainthm_subsample} with some exceptions.
If $\SubS = \I_{n,k}$,
the smoothed sensitivity bound has no $\min(k,1/\eps)$ term, owing to Lemmas~\ref{lem:T1_decomp_term2} and~\ref{lem:T1_upper_bound}, which gives the bound
\ba{\label{eq:degen_all_error}
&\bmb{\mathcal{A}(\X)-\theta} \le \frac{1}{\sqrt{\gamma}} O\bb{
\sqrt{\Var(A_n)} + \frac{k \conc}{n \eps} + \frac{k^2 C}{n^2 \eps^{2}} + \frac{k^3 C}{n^3 \eps^{3}}}.
}
Furthermore, since $\SubS=\I_{n,k}$, we have $A_n=U_n$ with probability at least $1-3\gamma$. 
Algorithm~\ref{alg:PrivateMeanDegenerate} uses a constant failure probability of $\gamma=0.01$. This ensures a success probability of at least $0.75$, which is further boosted by Wrapper~\ref{alg:MoM}. An application of Lemma~\ref{lem:mother_MoM} gives the stated result.

\subsection{Concentration for local Hájek projections}
\label{AppHajek}

\begin{proof}[Proof of Lemma~\ref{lem:local_hajek_conc}]
We first show that if $Y_1, Y_2, \dots, Y_t$ are random variables such that each $Y_j$ is $\tau_j$-sub-Gaussian, the sum ${Y_1+\dots+Y_t}$ is $\bb{\sqrt{\tau_1}+\sqrt{\tau_2}+\dots+\sqrt{\tau_t}}^2$-sub-Gaussian.

Define $p_j = \frac{\sum_{i=1}^t \sqrt{\tau_i}}{\sqrt{\tau_j}}$. Clearly, we have $\sum_{j=1}^t 1/p_j = 1$. By Hölder's inequality, for any $\lambda > 0$, we have
\bas{
\E\bbb{\exp\bb{\lambda \sum_{i=1}^t Y_i}} &= \E\bbb{\prod_{i=1}^t \exp(\lambda Y_i)} \le \prod_{i=1}^t \E\bbb{\exp(\lambda Y_i)^{p_i}}^{1/p_i} \le \prod_{i=1}^t \bb{\exp\bb{\frac{\lambda^2 p_i^2 \tau_i}{2}}}^{1/p_i} \\
&= \prod_{i=1}^t \exp\bb{\frac{\lambda^2 p_i \tau_i}{2}} =  \exp\bb{\frac{\lambda^2 (\sqrt{\tau_1}+\sqrt{\tau_2}+\dots+\sqrt{\tau_t})^2}{2}}.
}

Now, $h(X_S)$ is sub-Gaussian$(\tau)$ for all $S \in \I_{n,k}$. Since $\h{}{}{i}$ is the average of $t = \binom{n-1}{k-1}$ such quantities, it is clear from the previous claim that $\h{}{}{i}$ sub-Gaussian with parameter
$$ \frac{1}{t^2} (\underbrace{\sqrt{\tau}+\sqrt{\tau}+\dots+\sqrt{\tau}}_{t \text{ terms}})^2 = \tau.$$
\end{proof}



\begin{proof}[Proof of Lemma~\ref{lem:degen_conc_hajek}]
Define $\sigma_i^2:=\Var\bb{h(X_{S_i})|X_i=x_i}$.
First, conditioned on $X_i$, the projection $\h{}{X}{i}$ can be viewed as a U-statistic on the other $n-1$ data points. First, for this lemma, we use $\SubS=\I_{n,k}$, and 
\bas{
\h{}{X}{i}=\widehat{\E}\bbb{h(X_S)|X_i}=\frac{\sum_{S\in \SubS_i} h(X_S)}{{n-1\choose k-1}}
}

Since $h$ is bounded, the random quantity $\h{}{X}{i}-\E\bbb{h(X_S)|X_i} \in [-2C, 2C]$ satisfies the Bernstein moment condition and also the Bernstein tail inequality (cf.\ Proposition 2.3 in~\cite{wainwright2019high}). 
By Bernstein's inequality for U-statistics (see inequality~\eqref{eq:bernstein_ustat}), for all $t > 0$, we have
\bas{
\mathbb{P}\bb{\left.\bmb{\h{}{X}{i}-\E\bbb{h(X_S)|X_i}} \ge t\right|X_i} \le 2\exp\bb{\frac{-{\left \lfloor \frac{n-1}{k-1} \right \rfloor } t^2}{2\sigma_i^2 + 4Ct/3}}.
}

\bk

We can check that the choice $t = \sigma_i\sqrt{\frac{4k}{n}}\sqrt{\log \left(\frac{2n}{\beta}\right)}+\frac{8Ck}{n}\log \left(\frac{2n}{\beta}\right)$
makes this latter probability at most $\frac{\beta}{n}$.
\end{proof}

\section{U-statistic applications}
\label{SecAppProofs}

\subsection{Uniformity testing}\label{sec:uniformity}

To motivate the test, consider the expectation $\theta \defeq \E[h(X_i, X_j)]$ and the variance $\var(U_n)$:
\begin{lemma}\label{lem:unif_mean}
We have $\E[h(X_1, X_2)] = \frac{1}{m} + \frac{\norm{a}^2}{m^2}$. In particular, the means under the two hypothesis classes differ by at least $\frac{\delta^2}{2m}$.
\end{lemma}
\begin{proof} We have
\bas{
\E[h(X_1, X_2)] = \sum_{i=1}^m p_i^2 = \sum_{i=1}^m \frac{1+2a_i+a_i^2}{m^2} = \frac{1}{m} + \frac{\norm{a}^2}{m^2}.
}
Under approximate uniformity, this is at most $\frac{\delta^2}{2m}$; and under the alternative hypothesis, this quantity is at least $\frac{\delta^2}{m}$.
\end{proof}

\begin{lemma}\label{lem:unif_variance}
The variance of $U_n$ is
\bas{
\var(U_n) = \frac{2}{n(n-1)} \bb{2(n-2)\sum_{i<j} p_ip_j(p_i-p_j)^2 + \sum_{i=1}^m p_i^2 - \bb{\sum_{i=1}^m p_i^2}^2}.
}
\end{lemma}

\begin{proof}
The conditional variances $\zeta_1$ and $\zeta_2$ can be written as 
\begin{align*}
\zeta_1 &= \text{cov}(h(X_1,X_2), h(X_1, X_3)) \\
&= \mathbb{E}[{\mathbbm{1}[X_1 = X_2]\mathbbm{1}[X_1=X_3]}] -  \mathbb{E}[{\mathbbm{1}[X_1 = X_2]}]\mathbb{E}[{\mathbbm{1}[X_1 = X_3]}]\\
& = \sum_i p_i^3 - \left(\sum_i p_i^2\right)^2 = \sum_{i<j} (p_i^3p_j + p_ip_j^3 - 2p_i^2p_j^2) = \sum_{i<j} p_ip_j(p_i-p_j)^2 \ge 0, \text{ and} \\
\zeta_2 &= \text{cov}(h(X_1, X_2), h(X_1, X_2)) 
= \sum_{i=1}^m p_i^2 - \left (\sum_{i=1}^m p_i^2 \right )^2.
\end{align*}
Also from equation~\eqref{aeq:Ustatcovar}, we have
\bas{
\var(U_n) = \binom{n}{2}^{-1}\bb{2(n-2)\zeta_1 + \zeta_2}.
}
Combining the above bounds with equation~\eqref{aeq:Ustatcovar} shows the result.
\end{proof}

\paragraph{Proof of Theorem~\ref{thm:uniformity_testing}:}

Recall that $\tilde{\theta}$ denotes the private test statistic, which is thresholded at the value $\frac{1+3\delta^2/4}{m}$ to determine the output of the hypothesis test. We claim that the validity of the test is established if we can show that
\begin{equation}
\label{EqnCheb}
\mathbb{P}\left(|\tilde{\theta} - \E[U_n]| \le \frac{\delta^2}{4m}\right) \ge 1 - O(\gamma)
\end{equation}
under both hypotheses. Indeed, it would then hold that:
\begin{itemize}
\item[(i)] Under approximate uniformity,
\begin{equation*}
\tilde{\theta} < \frac{1}{m} + \frac{\delta^2}{2m} + \frac{\delta^2}{4m} = \frac{1 + 3\delta^2/4}{m}.
\end{equation*}
\item[(ii)] Under the alternative hypothesis,
\begin{equation*}
\tilde{\theta} \ge \frac{1}{m} + \frac{\delta^2}{m} - \frac{\delta^2}{4m} = \frac{1 + 3\delta^2/4}{m}.
\end{equation*}
\end{itemize}
To establish inequality~\eqref{EqnCheb}, we further write
\begin{equation}
\label{EqnCheb2}
\mathbb{P}\left(|\tilde{\theta} - \E[U_n]| > \frac{\delta^2}{4m}\right) \le \mathbb{P}\left(|\tilde{\theta} - U_n| > \frac{\delta^2}{8m}\right) + \mathbb{P}\left(|U_n - \E[U_n]| > \frac{\delta^2}{8m}\right).
\end{equation}
The second term can be controlled using an argument in Diakonikolas et al.~\cite{diakonikolas2016collision}, which further develops the variance bound in Lemma~\ref{lem:unif_variance} for the two hypothesis classes and then uses Chebyshev's inequality. It is shown that the second probability term in inequality~\eqref{EqnCheb2} can be bounded by $\alpha$ if $n = \Omega\left(\frac{\sqrt{m}}{\gamma\delta^2}\right)$.

To bound the first probability term in inequality~\eqref{EqnCheb2}, we study the concentration parameter $\xi$ for the local Hájek projection $\h{}{}{i} = \frac{1}{n-1}\sum_{j \neq i} h(X_i, X_j)$. We have the following result:
\begin{lemma}\label{lem:unif_conc}
If $\conc= \frac{6}{m} + \frac{8 \log (4n/\gamma)}{n}$ 
and $n \ge \frac{16}{\gamma}$, then $|\h{}{}{i}-U_n| \le \conc$ for all $i$, with probability at least $1-\gamma$.
\end{lemma}

\begin{proof}
By the triangle inequality, we have
\ba{
\label{EqnBear}
|\h{}{}{i}-U_n| \le |\h{}{}{i}-\E[h(X_1, X_2)|X_1]| +|\E[h(X_1, X_2)|X_1]-\theta|+ |U_n-\theta|.
}
We will provide a bound on each of these three terms. Note that
\begin{equation*}
h(X_1, X_2) | X_1 \sim \Bernoulli(p_{X_1}),
\end{equation*}
which has variance
\bas{
\sigma_i^2 &=\var(h(X_i,X_j)|X_i) 
=p_{X_i}(1-p_{X_i}) \le \frac{2}{m}.
}
Hence, with probability at least $1-\frac{\gamma}{2}$, the first term in inequality~\eqref{EqnBear} can be bounded as
\bas{
|\h{}{}{i}-\E[h(X_1, X_2)|X_1]| &\le 2 \sqrt{\frac{2}{mn}\log\left(\frac{4n}{\gamma}\right)}  + \frac{16}{3n}\log \left(\frac{4n}{\gamma}\right)
\leq \frac{2}{m}+\frac{7}{n}\log\bb{\frac{4n}{\gamma}},
}
where we have used the AM-GM inequality.
The second term in inequality~\eqref{EqnBear} can be bounded as
\bas{
\bmb{\frac{1+a_{X_i}}{m} - \bb{\frac{1}{m} + \frac{\norm{a}^2}{m^2}}} \le \frac{a_{X_i}}{m} + \frac{\norm{a}^2}{m^2} \le \frac{2}{m}.
}
Finally, by Chebyshev's inequality, the third term can be bounded as $|U_n-\theta| \le  \sqrt{\frac{2\var(U_n)}{\gamma}}$ with probability at least $1-\frac{\gamma}{2}$. It remains to find the variance of $U_n$.

By Lemma~\ref{lem:unif_variance}, if $\bmb{a_i} \le 1$ for all $i$, we have
\ba{
\var(U_n) &= \frac{2}{n(n-1)} \bb{2(n-2) \sum_{i < j} \frac{(1+a_i)(1+a_j)(a_i-a_j)^2}{m^4} + \sum_{i} \frac{(1+a_i)^2}{m^2} - \bb{\sum_{i} \frac{(1+a_i)^2}{m^2}}^2} \notag \\
&\le  \frac{2}{n(n-1)} \bb{2(n-2) \sum_{i < j} \frac{(1+a_i)(1+a_j)(a_i-a_j)^2}{m^4} + \sum_{i} \frac{(1+a_i)^2}{m^2}} \notag \\
&\le  \frac{2}{n(n-1)} \bb{2(n-2) \frac{4}{m^4} \binom{m}{2} + \frac{4m}{m^2}} \le \frac{8}{m^2 n} + \frac{8}{mn^2}. \label{eq:var_expression}
}
Combining the three bounds into inequality~\eqref{EqnBear}, with probability at least $1-\gamma$, we have
\bas{
|\h{}{}{i}-U_n| &\le \bb{\frac{2}{m}+\frac{7}{n}\log\bb{\frac{4n}{\gamma}}} + \frac{2}{m} + \frac{\sqrt{2}}{\sqrt{\gamma}}\bb{\frac{2\sqrt{2}}{m\sqrt{n}} + \frac{2\sqrt{2}}{\sqrt{m}n}} \\
&= \frac{4}{m} + \frac{7 \log (4n/\gamma)}{n} + \frac{4/\sqrt{\gamma}}{m\sqrt{n}} + \frac{4/\sqrt{\gamma}}{\sqrt{m} n} \\
&\le \frac{6}{m} + \frac{8 \log (4n/\gamma)}{n},
}
where in the second inequality, we used the AM-GM inequality and the assumption $n \ge \frac{16}{\gamma}$. 
The statement of the lemma follows after discarding lower-order terms.
\end{proof}
By Lemma~\ref{lem:unif_conc}, with probability at least $1-\gamma$, the weights of all projections in Algorithm~\ref{alg:PrivateMeanDegenerate} are equal to $1$ and $U_n = \tSubA$. Then $|\tilde{\theta} - U_n|$ is simply the magnitude of the noise added in the final step of Algorithm~\ref{alg:PrivateMeanDegenerate} (which uses a constant $\gamma=0.01$), which (cf.\ the proof of Theorem~\ref{thm:mainthm}) takes the form 
\bas{ O\bb{\frac{\conc}{n\eps} + \frac{1}{n^2\eps^2}+\frac{1}{n^3\eps^3}} = O\bb{\frac{\log n}{n^2\eps} + \frac{1}{mn\eps} + \frac{1}{n^2\eps^2}+\frac{1}{n^3\eps^3}},
}
with probability at least $0.75$. This is bounded by $\frac{\delta^2}{8m}$ as long as 
\bas{
n = \Omega\bb{\frac{m^{1/2}}{\delta \eps^{1/2}} \log \bb{\frac{m^{1/2}}{\delta \eps^{1/2}}} + \frac{m^{1/2}}{\delta \eps} + \frac{m^{1/3}}{\delta^{2/3} \eps} + \frac{1}{\delta^2}}.
}
This constant probability of success is further boosted by Wrapper~\ref{alg:MoM}. Now, an application of Lemma~\ref{lem:mother_MoM} gives the stated result.

\subsection{Proof of Theorem~\ref{cor:triangle}}\label{sec:triangle}

The privacy of the algorithm follows by composing (see Lemma~\ref{lem:basic_comp}) the $\eps$-privacy of $\nu$ and the $O(\eps)$-privacy of $\tilde{\theta}$ conditioned on $\nu$. It remains to show the utility of the algorithm.

The kernel $h(x,y)=1(\|x-y\|\leq r_n)$ is degenerate, since $P(\|X_i-X_j\|\leq r_n|X_i)$ does not depend on $X_i$. So $\var(\E[h(X_i,X_j)|X_i=x])=0$. 
We have $\var[h(X_i,X_j)]=O(r_n^2)$, so the non-private error is $O(r_n/n)$ (Eq~\ref{aeq:Ustatcovar}). 
Using Proposition 2.3 from Arcones and Gine~\cite{arcones93uprocess}, there exist universal constants $c_1,c_2,$ and $c_3$ such that
\bas{
    P\bb{\left|\frac{n-1}{2} (U'_n- r_n^2/4)\right|\geq t}\leq c_1\exp\bb{-\frac{c_2 t}{c_3 r_n+(t/n)^{1/3}}}.
}
Setting $t=nr_n^2/16$, we have, for large enough $n$, since $r_n=\Omega(n^{-1/2})$,
\bas{
    P\bb{\left| U'_n- r_n^2/4\right|\geq r_n^2/8}\leq c_1\exp\bb{-\frac{c_2 nr_n^2}{c_3 r_n+r_n^{2/3}}}\leq c_1 \exp(-c' n r_n^{4/3})=\tilde{O}\bb{\exp(-n^{1/3})}.
}
Therefore, with probability $1-o(1)$,
we have
\ba{\label{eq:nubound}
    U'_n &\in [r_n^2/8,3r_n^2/8]\notag\\
    \nu &:=U'_n+ \frac{Z}{n\epsilon}\in [r_n^2/9,r_n^2/2].
}
In particular, the probability that $U'_n+\frac{Z}{n\epsilon}$ computed in step 3 of Algorithm~\ref{alg:PrivateGraph} is then positive.

From Lemma~\ref{lem:degen_conc_hajek}, we have
$$\max_i|\g{}{X}{i}-\E[g(X_i,X_j,X_k)|X_i]| = \tilde{O}\bb{\frac{r_n^2}{\sqrt{n}} + \frac{1}{n}},$$
with probability at least $1-\gamma$. 
Moreover, since $g$ is degenerate, we have $\E[g(X_i,X_j,X_k)|X_i]=\theta_n$. Using the fact that $|U_n-\theta_n|\leq \sqrt\frac{2\Var(U_n)} {\gamma} \le \frac{r_n^2}{n} \sqrt{\frac{2}{\gamma}}$ with probability at least $1-\frac{\gamma}{2}$, we have
\bas{
\max_i |\g{}{X}{i}-U_n|&\leq \max_i |\g{}{X}{i}-\theta_n|+|\theta_n-U_n|\\
& = 2\sigma_i \sqrt{\frac{2}{n}\log\left(\frac{2n}{\gamma}\right)}  + \frac{16}{3n}\log \left(\frac{2n}{\gamma}\right) + \frac{r_n^2}{n} \sqrt{\frac{2}{\gamma}} \\
& \le 18\nu \sqrt{\frac{2}{n}\log\left(\frac{2n}{\gamma}\right)}  + \frac{16}{3n}\log \left(\frac{2n}{\gamma}\right) + \frac{9\nu}{n} \sqrt{\frac{2}{\gamma}} =: \xi,
}
with probability $1-O(\gamma)$.
Hence, using Theorem~\ref{thm:mainthm}, and noting that $\xi = \tilde{O}\left(\frac{r_n^2}{\sqrt{n}}\right)$,
we can ensure that
the estimate $\tilde{\theta}$ output by Algorithm~\ref{alg:PrivateMeanDegenerate} satisfies
\bas{
 \bmb{\tilde{\theta}-\theta} &= O\bb{\frac{r_n^2}{n} + \frac{\xi}{n\eps} + \frac{1}{n^2 \eps^2}+\frac{1}{n^3 \eps^{3}}}=\tilde{O}\bb{\frac{r_n^2}{n} + \frac{r_n^2}{n^{3/2}\epsilon}+\frac{1}{n^2\epsilon} + \frac{1}{n^2\eps^2}+ \frac{1}{n^3\eps^3}}\\
 &=\tilde{O}\bb{\frac{r_n^2}{n}  + \frac{1}{n^2\eps^2}},
 }
with probability at least $0.75$,
where we used the AM-GM inequality to eliminate the term $\frac{r_n^2}{n^{3/2}\epsilon}$.
This constant probability of success is further boosted by Wrapper~\ref{alg:MoM}. Now, an application of Lemma~\ref{lem:mother_MoM} gives the stated result.

\end{document}